\documentclass[sn-mathphys,Numbered]{sn-jnl}

\usepackage{graphicx}%
\usepackage{multirow}%
\usepackage{amsmath,amssymb,amsfonts}%
\usepackage{amsthm}%
\usepackage{mathrsfs}%
\usepackage[title]{appendix}%
\usepackage{xcolor}%
\usepackage{textcomp}%
\usepackage{manyfoot}%
\usepackage{booktabs}%
\usepackage{algorithm}%
\usepackage{algorithmicx}%
\usepackage{algpseudocode}%
\usepackage{listings}%
\usepackage[normalem]{ulem}
\usepackage{enumitem}
\usepackage{comment}
\usepackage{tabularx}
\usepackage{mathtools}
\usepackage{diagbox}
\usepackage{rotating}

\newtheorem{theorem}{Theorem}
\newtheorem{lemma}[theorem]{Lemma}%

\newtheorem{definition}{Definition}%

\newcommand{\Abar}{\overline{A}}

\newcommand{\Bbar}{\overline{B}}

\newcommand{\Es}{\mathcal{E}}

\newcommand{\Hs}{\mathcal{H}}
\newcommand{\Ib}{\mathbf{I}}

\newcommand{\Lb}{\mathbf{L}}
\newcommand{\Ls}{\mathcal{L}}

\newcommand{\Mb}{\mathbf{M}}

\newcommand{\Qb}{\mathbf{Q}}

\newcommand{\R}{\mathbb{R}}

\newcommand{\T}{\mathrm{T}}

\newcommand{\ub}{\mathbf{u}}

\newcommand{\Us}{\mathcal{U}}
\newcommand{\Vs}{\mathcal{V}}
\newcommand{\vb}{\mathbf{v}}

\newcommand{\xb}{\mathbf{x}}
\newcommand{\Zbar}{\overline{Z}}

\newcommand{\0}{\mathbf{0}}
\newcommand{\1}{\mathbf{1}}
\renewcommand{\emptyset}{\varnothing}

\newcommand{\ra}{\rightarrow}
\newcommand{\tos}{\leadsto}
\newcommand{\lap}{\mbox{$\cal L$}}

\newcommand{\Xmp}{X(\cdot)}

\newcommand{\deltab}{\boldsymbol{\delta}}
\newcommand{\lambdab}{\boldsymbol{\lambda}}
\newcommand{\rhob}{\boldsymbol{\rho}}

\newcommand{\Pp}{\mathcal{V}_{\text{out}}}

\newcommand{\np}{\text{$N\!+\!1$}}

\DeclareMathOperator{\tr}{tr}

\begin{document}

\title[Article Title]{Algebraic formulas for first-passage times of Markov processes in the linear framework: generalising the work of Hill and Kac}


\author[1,2]{\fnm{Kee-Myoung} \sur{Nam}} \email{keemyoung.nam@yale.edu}

\author*[1,3]{\fnm{Jeremy} \sur{Gunawardena}} \email{jeremy.gunawardena@upf.edu}


\affil[1]{\orgdiv{Department of Systems Biology}, \orgname{Harvard Medical School}, \orgaddress{\street{200 Longwood Ave.}, \city{Boston}, \state{MA} \postcode{02115}, \country{USA}}}
\affil[2]{Current address: \orgdiv{Department of Molecular, Cellular and Developmental Biology}, \orgname{Yale University}, \orgaddress{\street{260 Whitney Ave.}, \city{New Haven}, \state{CT} \postcode{06511}, \country{USA}}}
\affil[3]{\orgdiv{Current address: Department of Medicine and Life Sciences}, \orgname{Pompeu Fabra University}, \orgaddress{\street{Dr. Aiguader 88}, \postcode{08003}, \city{Barcelona}, \country{Spain}}}




\abstract{In a preceding paper, we used the graph-theoretic linear framework to show how transient properties of continuous-time Markov processes---splitting probabilities and the moments of first-passage time (FPT) distributions---could be expressed as rational algebraic functions of the transition rates, by using spanning forests of the underlying graph. This contrasts with the related rational formulas for steady-state (s.s.) probabilities, which use only spanning trees. The biophysicist Terrell Hill sketched a procedure for calculating mean FPTs and splitting probabilities in terms of the s.s.~probabilities of a modified Markov process, thereby converting calculations using ensembles of trajectories to those using a single trajectory. Similarly, Mark Kac showed that the mean recurrence time to a state of a Markov process could be expressed in terms of the s.s.~probability of that state. Here, we explore further the relationships between transient and s.s.~properties, forests and trees, and ensemble and single-trajectory calculations. We formalise Hill's procedure by introducing a Hill operator, $\Hs_u[G]$, on a graph, $G$, and use it to calculate all moments of the conditional and unconditional FPTs from $u$ as rational functions of the s.s.~probabilities of $\Hs_u[G]$, which arise from trees and ``exchange factors'' which arise from forests. We then combine this with an unravelling operator, $\Us_v[G]$, to calculate all moments of the recurrence time distribution as rational functions of the s.s.~probabilities of $G$ and related exchange factors. Surprisingly, the Hill operator turns out to be a left-inverse to the unravelling operator, suggesting that the algebra of operators on linear framework graphs may be of broader interest. Our results integrate and generalise previously disparate findings into a common repertoire of rational algebraic formulas for FPTs of Markov processes.}




\maketitle

\newpage
\section{Introduction}
\label{sec:intro}

Finite-state, continuous-time, time-homogeneous Markov processes---from now on, simply ``Markov processes''---are widely used in biology as mathematical models of stochastic behaviour. Historically, their steady-state (s.s.) probabilities were particularly exploited but, more recently, as experimental methods have become more powerful, measures of transient behaviour, such as first-passage times (FPTs), have been used to quantify functional behaviour in a variety of biological settings \cite{kou:05,kolomeisky:07,bel:10, moffitt:14,banerjee:17,ghusinga:17,lammers:20,ham:24,gqo25}.

In previous work, we developed the linear framework \cite{gunawardena:12,mirzaev:13}, which may be thought of, for present purposes, as a graph-theoretic approach to Markov processes. We have used the framework to study mechanisms of cellular information processing, such as post-translational modification \cite{nam:20,owen:23}, conformational allostery \cite{biddle:21}, gene regulation \cite{estrada:16,biddle:19,nasser:25} and cellular input-output functions \cite{wong:18b,martinez-corral:24}; for reviews and further background, see \cite{gunawardena:14,wong:20,nam:22,nam:23}. The framework focussed at first on steady-state properties. It showed, for instance, how the s.s.~probabilities of a Markov process could be expressed as rational algebraic functions of the transition rates. The \emph{Matrix-Tree theorem} of graph theory, versions of which go back to Kirchhoff's 19th century work on electrical networks \cite{kirchhoff:1847}, allows these rational functions to be constructed from the spanning trees of the linear framework graph associated with the Markov process. These algebraic formulas hold whether or not the Markov process relaxes to a s.s.~of thermodynamic equilibrium. If it does, the formulas correspond to the prescription of equilibrium statistical mechanics but, if not, they provide a context where Markov processes in which energy is being expended away from thermodynamic equilibrium admit exactly solvable s.s.~probabilities. These exact solutions have been particularly useful for analysing the impact of energy expenditure in biological information processing \cite{ahsendorf:14,estrada:16,wong:18a,wong:18b,martinez-corral:24}. The combination of graphs and exact solutions has made it possible to prove theorems about biological systems that rise above their ever-present molecular complexity \cite{wong:18b,biddle:21,martinez-corral:24}.

Recently, in a prequel to the present paper \cite{nam:25}, we showed that certain transient properties of Markov processes---splitting probabilities and moments of FPT distributions---could also be expressed as rational algebraic functions of the transition rates. This advance relies on a generalisation of the classical Matrix-Tree theorem, the \emph{All-Minors Matrix-Tree theorem} \cite{fiedler:58,chaiken:82,moon:94}, which uses spanning forests of the graph, in contrast to the spanning trees needed for s.s.~probabilities. These results were previewed in \cite{nam:23}, which also provides additional background, and they have already provided a foundation on which \cite{vus25, vus26} have recently built further.

An interesting feature of these rational algebraic functions is that they are \emph{manifestly integrally positive}: the constituent polynomials are sums of monomials in the transition rates with positive integer coefficients. Since the quantities in question are positive, as are the transition rates, the positivity of the functions is to be expected, but the \emph{manifest} positivity is an additional property that relates to broader mathematical concerns. The \emph{integral} nature of the manifest positivity reflects the fact that the polynomials are counting combinatorial objects, in this case spanning trees or forests of the graph \cite{lkw22}.

The purpose of the present paper is to integrate the formulas for transient properties described in our prequel paper \cite{nam:25} with historical approaches to calculating them by different methods. To describe these approaches, consider a Markov process with a source state, $u$, and a set of absorbing target states, $Z$, which all stochastic trajectories from $u$ eventually reach with probability one. Equivalently, we can consider these states to be source and target \emph{vertices} in the corresponding linear framework graph. We will describe the relationship between the Markov process and its graph more formally in the Results. 

In the first historical approach, the biophysicist Terrell Hill sketched a procedure for calculating the splitting probabilities, $\pi_{u,z}$, from $u$ to an individual target state, $z \in Z$, and the mean unconditional FPT from $u$ to any target state, $\mu^{(1)}_{u,Z}$ \cite{hill:88}. (These notations, and the others introduced below, are taken from \cite{nam:25} and will be used throughout the present paper.) Hill's procedure amounted to constructing a new Markov process whose s.s.~probabilities are then used to determine $\pi_{u,z}$ and $\mu^{(1)}_{u,Z}$. Hill's rationale seems to have been, broadly speaking, ``ergodic'': he sought to replace the standard definition of transient quantities in terms of an ensemble of finite stochastic trajectories, in the limit as the ensemble size goes to infinity, with a calculation based on a single stochastic trajectory, in the limit as the trajectory length goes to infinity. Hill gave no proofs of his procedure but worked out an example and ``waved his hands'' to suggest that it worked in general. His procedure did not extend to the mean conditional FPT to a single target state, $z \in Z$, conditioned on that state being reached, $\mu^{(1)}_{c,u,z}$; and he said nothing about the higher moments, $\mu^{(r)}_{u,Z}$ and $\mu^{(r)}_{c,u,z}$ for $r > 1$.

As usual for Hill's pioneering work, his insights were perfectly on the mark. Here, we rigorously prove and generalise his claims. We introduce a graph operator, $\Hs_u$, that takes a suitably defined graph, $G$, with source vertex $u$, to a modified graph, $\Hs_u[G]$, which we call the \emph{Hill construction} on $G$. $\Hs_u[G]$ provides a formal definition of the intuition underlying Hill's modified Markov process. We use the results of our prequel paper to show how the corresponding splitting probabilities, $\pi_{u,z}$, and all moments of the unconditional FPTs, $\mu^{(r)}_{u,Z}$, and of the conditional FPTs, $\mu^{(r)}_{c,u,z}$, of $G$---which is to say, of the corresponding Markov process---can be calculated in terms of the s.s.~probabilities of $\Hs_u[G]$, together with \emph{exchange factors} that describe the relationships between Hill constructions with different source vertices. Hill was not aware of these exchange factors, as they are not needed for the mean quantities that he studied. Interestingly, the exchange factors are given by spanning forests and trees of $\Hs_u[G]$, while the s.s.~probabilities, as noted above, are given solely by the spanning trees of $\Hs_u[G]$. 

In the second historical approach, the mathematician Mark Kac showed that the mean recurrence time to a state $u$ in a Markov process, $\tau^{(1)}_u$, could be expressed in terms of the reciprocal of the s.s.~probability of $u$ \cite{kac:47}. (Kac's original result was for discrete-time Markov chains but the extension to continuous-time Markov processes is well-known \cite{serfozo}.) Once again, in analogy to Hill's procedure, a transient quantity is expressed in terms of a s.s.~quantity. In the present paper, we generalise Kac's result to calculate all moments of the recurrence time distribution, $\tau^{(r)}_{u}$, in terms of s.s.~probabilities of the Markov process and the exchange factors described above. The proof relies on combining $\Hs_u$ with another graph operator, $\Us_v$, which we call the \emph{unravelling construction} with source vertex $v$. 

We were surprised to find that the Hill and unravelling constructions, despite being defined independently to calculate different quantities, are closely related to each other. The former is a left-inverse for the latter, so that $\Hs_v[\,\Us_v[G]] = G$ (Lemma~\ref{t-hug}). This unexpectedly elegant relationship suggests how the use of operators on linear framework graphs can clarify some of the underlying mathematical properties of Markov processes. The results presented here integrate our previous work with that of Hill and Kac and clarify the relationships between transient and s.s.~properties, forests and trees, and ensemble and single-trajectory calculations. 

In the first section of the Results that follow, we introduce the linear framework and the definitions and notation that we use throughout this paper. To keep the paper self-contained, we restate the necessary results from our prequel paper \cite{nam:25}. We then introduce the Hill and unravelling constructions and establish their basic properties, including the left-inverse relationship described above, using examples to clarify the details. Two key results for the Hill construction are then proved, the \emph{forest-tree lemma} (Lemma \ref{lem:hill}) and the \emph{exchange formula} (Lemma \ref{lem:exchange}). With these preliminaries in hand, we rigorously derive Hill's formula for the splitting probabilities, $\pi_{u,z}$, and establish new Hill-like formulas for all moments of the unconditional FPTs, $\mu^{(r)}_{u,Z}$, and conditional FPTs, $\mu^{(r)}_{c,u,z}$ (Theorems \ref{thm:hill-split}, \ref{thm:hill-moment-Z} and \ref{thm:hill-moment}). We then establish new Kac-like formulas for all moments of the recurrence time, $\tau^{(r)}_{u}$ (Theorem \ref{thm:kac-moment}). To demonstrate our results in practice, we then work through the calculation of the first and second moment of the recurrence time for the \emph{cycle graph} example in Fig.~\ref{fig:example}. A significant challenge in direct calculations is the combinatorial explosion in spanning trees and forests. We address this here using a recursive technique introduced by Chebotarev and Agaev in \cite{chebotarev:02}, which we have pointed out in previous papers but have not previously shown at work. This last section may therefore be of independent interest as a guide to undertaking calculations with graphs. Our Discussion expands on the themes introduced here. 

\section{Results}
\label{sec:results}

\subsection{Background and preliminaries}
\label{sec:background}

To keep the present paper self-contained, we bring together in this section the concepts, terminology and notation that we will need. Most of this material comes from our prequel paper \cite{nam:25}, which should be consulted for more details. A summary of the notation we use throughout the paper is given in Table \ref{table:notation} in Appendix \ref{app:notation}.

\subsubsection{Graphs, Laplacians and Markov processes}

A linear framework graph---from now on, simply a ``graph''---is a finite, simple, directed graph with labelled edges (Fig.~\ref{fig:intro}A, left). We will generally assume that graphs are \emph{connected}, so that they do not fall into separate pieces, and will make further specific assumptions about the \emph{strongly connected components}, as explained below. In the stochastic context considered in this paper, if $G$ is a graph, then the vertices, $\Vs(G)$, typically indexed by $i = 1, \dotsc, N$, represent states of a biomolecular system; the edges, $\Es(G) \subseteq \Vs(G) \times \Vs(G)$, denoted $i \ra j$, represent stochastic transitions between states; and the edge labels, denoted $\ell(i \ra j)$, are positive transition rates with units of (time)$^{-1}$. The edge labels can incorporate complex, possibly time-dependent, interactions between the system and its environment but, for the context considered here, we assume that the labels are constants; see \cite{nam:22} for a broader discussion and the reasoning behind these assumptions. 

\begin{figure}
\centering
\includegraphics[trim={1.3in 3.5in 1.3in 3.1in}, clip, width=0.9\textwidth]{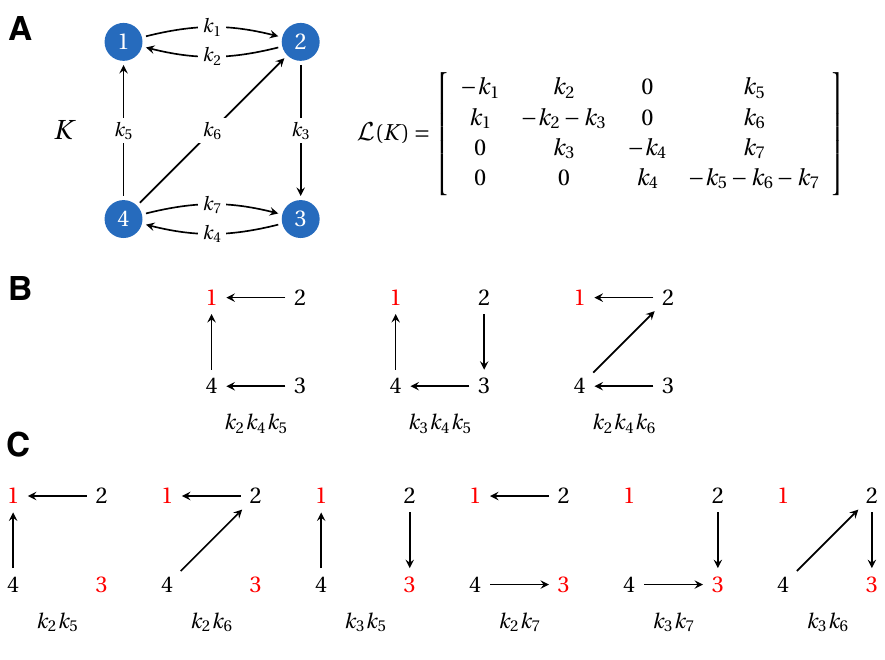}
\caption{Graphs, Laplacian matrices and spanning forests. (\textbf{A}) An example graph, $K$, on the left with its Laplacian matrix on the right. 
(\textbf{B}) All spanning trees in $\Phi_{\{1\}}(K)$, with the roots shown in red. Each tree is annotated below with its weight.
(\textbf{C}) All spanning forests in $\Phi_{\{1,3\}}(K)$, with the roots shown in red. Each forest is annotated below with its weight.}
\label{fig:intro}
\end{figure}

When it is necessary to make clear which graph is being discussed, we may use subscripts, as in $i \ra_G j$, or brackets, as in $\Vs(G)$, for disambiguation, but will otherwise keep the notation light and allow the context to clarify the meaning. 

A graph $G$ gives rise to a dynamics that is most easily described by imagining each edge to be a chemical reaction under mass-action kinetics with the edge label as the rate constant. Since each edge has only one source vertex, this dynamics is necessarily linear and can therefore be described by a matrix equation, 
\begin{equation}
    \frac{d\xb}{dt} = \Ls(G) \, \xb(t) \,,
\label{eq:master}
\end{equation}
where $\xb(t) = \left( x_1(t), \dotsc, x_N(t) \right)^\T$ is the $N \times 1$ column vector of the ``concentrations'' of the vertices at time $t$, and $\Ls(G)$ is the $N \times N$ \emph{Laplacian matrix} of $G$. The linear framework acquires its name from the linearity of Eqn.~\ref{eq:master} \cite{gunawardena:12}. It may seem puzzling that the framework can deal with nonlinear biochemistry; this capability arises through the labels and is reviewed further in \cite{gunawardena:14,nam:22}.

By working through the definition of $\Ls(G)$, it is not difficult to see by that its entries are given by,
\begin{equation}
    \Ls(G)_{i,j} = \begin{cases}
        \ell(j \to i) & \text{if $i \neq j$ and $j \to_G i$} \\
        0 & \text{if $i \neq j$ and $j \not\to_G i$} \\
        -\lambda_i & \text{if $i = j$} \,,
    \end{cases}
\label{eq:lap}
\end{equation}
where it is convenient to introduce the quantity $\lambda_i$ for the sum of all outgoing edge labels from $i$, 
\begin{equation}
    \lambda_i = \sum_{j \in \Vs(G) : i \to j}{\ell(i \to j)} \,.
\label{eq:lambda}
\end{equation}
The Laplacian matrix of the graph in Fig.~\ref{fig:intro}A is shown on the right of the panel. 

Laplacian matrices of undirected graphs, under appropriate scalings, may be considered as discrete approximations of the Laplacian operator \cite{chung}, so that Eqn.~\ref{eq:master} may be thought of as a directed analogue of a discretised diffusion equation over the graph $G$. It is evident that the dynamics must conserve the total concentration, so that Eqn.~\ref{eq:master} satisfies the conservation law, 
\begin{equation}
    x_1(t) + \dotsb + x_N(t) = x_{\text{tot}} \,.
\label{eq:conserve}
\end{equation}
Eqn.~\ref{eq:conserve} corresponds to the column sums of $\Ls(G)$ being zero, so that $\1^\T \Ls(G) = \0^\T$, where $\1$ and $\0$ are the all-ones and zero vectors of the appropriate dimension, respectively.

A graph $G$ defines the infinitesimal generator of a Markov process, denoted $\Xmp$, whose state space is the set of vertices, $\Vs(G)$, and whose transitions are given by the edges. If $i \ra j$, then the infinitesimal transition rate from $i$ to $j$ is given by the corresponding edge label,
\begin{equation}
    \lim_{h \to 0}{\frac{\Pr{\left[ X(t+h) = j \mid X(t) = i \right]}}{h}} = \ell(i \to j) \,.
\label{e-mpd}
\end{equation}
If there is no edge, $i \not\ra j$, then the infinitesimal transition rate is zero. The \emph{master equation}, or Kolmogorov forward equation, of $\Xmp$ is then given by the Laplacian dynamics in Eqn.~\ref{eq:master} \cite[Theorem~4]{mirzaev:13}, with the dynamical variable $x_i(t)$ being the probability that $\Xmp$ is in state $i$ at time $t$. In this case, the conservation law in Eqn.~\ref{eq:conserve} corresponds to the conservation of total probability, with $x_{\text{tot}} = 1$. We will make these assumptions from now on to reflect the stochastic context in which we are working. The linearity of Eqn.~\ref{eq:master} may be less puzzling when it is seen to be a master equation. 

It is not hard to show that any Markov process which has an infinitesimal generator, so that the limits in Eqn.~\ref{e-mpd} exist, gives rise to a corresponding graph and that these transformations are inverse to each other \cite[Theorem~4]{mirzaev:13}. It is in this sense that the linear framework, in the context considered here, offers a graph-theoretic approach to Markov processes. This viewpoint is rarely seen in Markov process theory but has earlier roots in biophysics, in the work, once again, of the inimitable Terrell Hill \cite{hill:66} and also subsequently of J{\"u}rgen Schnakenberg \cite{schnakenberg:76}. It then seems to have faded from view within biophysics until more recent times, when it resurfaced primarily within stochastic thermodynamics. References may be found in our prequel paper \cite{nam:25}.

\subsubsection{Steady-state and transient quantities}
\label{sec:ss-trans}

Given a graph $G$, the Laplacian dynamics in Eqn.~\ref{eq:master} always relaxes to a vector \cite[Theorem~2]{mirzaev:13}, 
\begin{equation}
    \xb^* = \lim_{t \ra \infty} \xb(t) \,,
\label{e-itl}
\end{equation}
which generally depends on the initial conditions, $\xb(0)$, as described by \cite[Theorem~3]{mirzaev:13}. It is clear that this vector must be a s.s. of the dynamics, so it follows from Eqn.~\ref{eq:master} that $\xb^*$ lies in the kernel of the Laplacian matrix, $\xb^* \in \ker\Ls(G)$. 

It will be helpful to understand for which vertices $i$ this s.s.~probability can be nonzero, $x_i^* \not= 0$. We will summarise what we need for the present paper but further details are provided in \cite{mirzaev:13, nam:25}. Given $i, j \in \Vs(G)$, we say that $i$ \emph{leads to} $j$, denoted $i \tos j$, if either $i = j$ or there is a directed path of edges, $i = i_1 \ra i_2 \ra \dotsb \ra i_k = j$, in $G$. We say that $i$ is \emph{strongly connected} to $j$ if $i \tos j$ and $j \tos i$. This defines an equivalence relation on the vertices, whose equivalence classes are the \emph{strongly connected components} (SCCs) of $G$. The ``leads to'' relation on vertices gives rise to a partial order on the SCCs, and the terminal elements in this partial order are the \emph{terminal SCCs} of $G$. A vertex $i$ can have nonzero s.s.~probability, so that $x^*_i \not= 0$ if, and only if, $i$ belongs to some terminal SCC. As the Laplacian dynamics in Eqn.~\ref{eq:master} evolves in time, the probability mass increasingly concentrates on the terminal SCCs, irrespective of the initial conditions.

An important special case arises when $G$ itself is strongly connected, as is the case for the example in Fig.~\ref{fig:intro}A. Strongly connected graphs have been frequently invoked in the literature \cite{owen:23,biddle:21,estrada:16,biddle:19,martinez-corral:24}. A more general situation, which will be important for us here, arises when $G$ has a single terminal SCC. The kernel of the Laplacian is then one-dimensional, $\dim\ker\Ls(G) = 1$ \cite{gunawardena:12}, so that there is a unique s.s.~probability distribution over the vertices. This s.s.~is independent of the initial conditions and occurs irrespective of the size and structure of $G$. In the strongly connected case, this collapse in the degrees of freedom at s.s.~has significant biochemical implications \cite{nam:22}.

We can also think about probabilities from the viewpoint of the Markov process, $\Xmp$, associated to the graph $G$. We will describe this informally, following the treatment in our prequel paper \cite{nam:25}, which gives further details and references. We can imagine generating an ensemble of stochastic trajectories of $\Xmp$, with the initial state $X(0)$ being repeatedly drawn randomly from the initial probability distribution $\xb(0)$. Some trajectories may endlessly cycle without making progress but these trajectories are non-generic and constitute a set of measure zero within the ensemble. The proportion of trajectories in the ensemble for which $X(t) = i$, in the limit of an infinite ensemble, is the probability $x_i(t)$ that appears in the master equation in Eqn.~\ref{eq:master}. The s.s.~probability $x^*_i$ arises in the infinite time limit (Eqn.~\ref{e-itl}). In the case where $G$ is strongly connected, so that $x^*_i$ is unique and does not depend on $\xb(0)$, the ergodic theorem for Markov processes \cite{norris} tells us that $x^*_i$ is also given by the asymptotic proportion of time that $\Xmp$ spends in state $i$ along any individual generic trajectory. 

The Markov process viewpoint is necessary to define the transient quantities mentioned in the Introduction, with which we will be concerned here. In the case of Hill's work, it is convenient to do this under the assumption that the terminal SCCs of the graph are all singletons. The unique vertices in each SCC are called \emph{terminal vertices} or \emph{absorbing states}. It follows immediately from the definitions given above that there are no outgoing edges from any terminal vertex. In this setting, s.s.~probability is concentrated on the terminal vertices and any generic trajectory will eventually reach some terminal vertex and remain there indefinitely. The prequel paper describes how a general graph can be transformed so that it satisfies the assumptions made here \cite{nam:25}. 

Hill's work concerned splitting probabilities and FPTs, as described in the Introduction. These can be defined more formally as follows. Given a graph $G$, let $Z \subseteq \Vs(G)$ denote the set of $T \geq 1$ terminal vertices, $Z = \{z_1, \dotsc, z_T\}$, so that $\{z_1\}, \dotsc, \{z_T\}$ form the terminal SCCs. Choose a source vertex, $u \in \Vs(G) \setminus Z$, and generate an ensemble of stochastic trajectories, as above, starting from $X(0) = u$. In the limit of an infinite ensemble, the proportion of trajectories that terminate at $z \in Z$ defines the \emph{splitting probability}, $\pi_{u,z}$. The time taken for a trajectory to first reach some terminal vertex in $Z$ is the \emph{unconditional FPT} from $u$ to $Z$ and, for $r \geq 1$, the $r$-th moment of the corresponding probability distribution is the quantity $\mu^{(r)}_{u,Z}$. Finally, if we take the sub-ensemble of trajectories that terminate at a particular target vertex, $z \in Z$, then the time taken to reach $z$ along these trajectories is the \emph{conditional FPT} from $u$ to the target vertex $z$, and the $r$-th moment of the corresponding probability distribution is the quantity $\mu^{(r)}_{c,u,z}$. 

The following two sections describe how s.s.~probabilities and the three types of transient quantities defined above can be calculated as rational algebraic functions of the edge labels \cite{nam:25}. 

For the work of Mark Kac, we will consider a strongly connected graph and any vertex, $u \in \Vs(G)$. If we generate an ensemble of stochastic trajectories starting from $X(0) = u$, then each generic trajectory will eventually return to $u$. The time taken for this first return is the \emph{recurrence time} to $u$ and the $r$-th moment of the corresponding probability distribution is the quantity $\tau^{(r)}_u$. We will also calculate these moments as rational algebraic functions of the edge labels using the ideas introduced below. 

\subsubsection{Spanning forests and Matrix-Tree theorems}
\label{sec:forests}

The determinant of a matrix is generally an alternating sum of monomials in the matrix entries. One of the remarkable properties of the Laplacian matrix, $\Ls(G)$, of the graph $G$ is that its \emph{minors}---the determinants of square submatrices obtained by removing rows and columns of $\Ls(G)$---have a combinatorial origin in the \emph{spanning forests} of $G$. It follows, furthermore, that certain minors---especially the principal minors, which arise from removing rows and columns with the same indices---exhibit extensive cancellations that yield manifestly integrally positive polynomials, with an accompanying sign. The classical Matrix-Tree theorem deals with the first minors, in which a single row and single column are removed from $\Ls(G)$, while the more recent All-Minors Matrix-Tree theorem, as its name suggests, deals with any minor. We do not need these theorems in full here---they are discussed further in \cite{nam:25}---but we will rely on some of their consequences, for which we need further concepts and notation that are described in this section. 

A subgraph, $F$, of $G$ is a \emph{spanning forest} of $G$ if it contains all the vertices of $G$, lacks cycles when edge directions are ignored, and contains at most one outgoing edge from each vertex. The subset of vertices in $F$ with no outgoing edges are the \emph{roots} of $F$. If there is only one root, then $F$ is a \emph{spanning tree}. If $A$ is a non-empty subset of vertices of $G$, $\emptyset \neq A \subseteq \Vs(G)$, then the set of spanning forests of $G$ rooted at $A$ is denoted $\Phi_A(G)$. Fig.~\ref{fig:intro}B shows the spanning trees in $\Phi_{\{1\}}(K)$ for the graph $K$ in Fig.~\ref{fig:intro}A, and Fig.~\ref{fig:intro}C shows the spanning forests in $\Phi_{\{1,3\}}(K)$. 

Given $F \in \Phi_A(G)$, it is not difficult to see that for any $i \in \Vs(G)$, there exists a unique root, $a \in A$, such that $i \tos_F a$. Let $\# S$ denote the size of the subset $S \subseteq \Vs(G)$. Given $B \subseteq \Vs(G)$ with $\# B = \# A$, let $\Phi_{B \ra A}(G) \subseteq \Phi_A(G)$ denote the set of those forests rooted at $A$ in which each vertex in $B$ leads to a distinct root in $A$. For example, if $F \in \Phi_{\{1,2\} \to \{3,4\}}(G)$, then either $1 \tos 3$ and $2 \tos 4$ in $F$, or $1 \tos 4$ and $2 \tos 3$ in $F$. On the other hand, if $F \in \Phi_{\{1,3\} \to \{3,4\}}(G)$, then it must be the case that $1 \tos 4$ in $F$, since, by definition, every vertex leads to itself, so that $3 \tos 3$. For the same reason, it follows that $\Phi_{A \to A}(G) = \Phi_A(G)$. Since every vertex leads to the root of a spanning tree, $\Phi_{\{i\} \ra \{j\}}(G) = \Phi_{\{j\}}(G)$ for any $i,j \in \Vs(G)$. 

Spanning forests and spanning trees are defined only in terms of the vertices and edges of the graph. They are independent of the edge labels, which enter the picture through the \emph{weight} function. Given a subgraph $F$ of a graph $G$, the weight of $F$, denoted $w(F)$, is the product of the edge labels over the edges of $F$, 
\[
    w(F) = \prod_{i \ra j \in \Es(F)} \ell(i \ra j) \,.
\]
$w(F)$ is a monomial in the edge labels. Figs.~\ref{fig:intro}B and \ref{fig:intro}C give the weights of the trees and forests shown there.

More generally, if $U$ is a set of subgraphs of $G$, then the weight of $U$, denoted $w(U)$, is the sum of the weights of the constituent subgraphs,
\[
    w(U) = \sum_{F \in U} w(F) \,.
\]
By the usual convention for empty products and empty sums, the weight of an edgeless subgraph is one, while the weight of an empty set of subgraphs is zero \cite{nam:25}. It is clear that $w(U)$ is a manifestly integrally positive polynomial in the edge labels. For example, it follows from Fig.~\ref{fig:intro}B that 
\[
    w(\Phi_{\{1\}}(K)) = k_2k_4k_5 + k_3k_4k_5 + k_2k_4k_6 \,,
\]
and from Fig.~\ref{fig:intro}C that
\[
    w(\Phi_{\{1,3\}}(K)) = k_2k_5 + k_2k_6 + k_3k_5 + k_2k_7 + k_3k_7 \,.
\]
Recall that the total degree of the monomial $(k_1)^{e_1} \cdots (k_m)^{e_m}$ is $e_1 + \cdots + e_m$. A spanning forest with $r$ roots of a graph with $N$ vertices has $N - r$ edges, so that the corresponding weight is a monomial of total degree $N-r$ in the edge labels. Hence, if $\# A = r$, then $w(\Phi_A(G))$ is a polynomial in which every monomial has the same total degree, $N - r$. 

Finally, we need notation for submatrices. Let $A \subseteq \Vs(G)$ be a vertex subset of size $k$, so that $\# A = k$, and let $\Abar$ denote its complement, $\Abar = \Vs(G) \setminus A$. Recall that $\#\Vs(G) = N$, so that $\#\Abar = N - k$. We denote the $k \times k$ submatrix of $\Ls(G)$ consisting of the rows and columns in $A$ by $\Ls(G)_{[A,A]}$. Accordingly, $\Ls(G)_{[\Abar,\Abar]}$ denotes the $(N-k) \times (N-k)$ submatrix in which the rows and columns in $A$ have been removed.

\subsubsection{Previous results}
\label{sec:prev}

Having introduced the notation that we need, we can now summarise the results from our previous work on which we will rely for the present paper.

The first result is a special case of the All-Minors Matrix-Tree theorem. Let $A$ be a non-empty, proper subset of vertices, $\emptyset \not= A \subsetneq \Vs(G)$, with $\# A = k$. Then the principal minor obtained by removing the rows and columns in $A$ is given by, 
\begin{equation}
    \det{\Ls(G)_{[\Abar,\Abar]}} = (-1)^{N - k} \, w(\Phi_A(G)) \,.
\label{eq:ammtt}
\end{equation}
Eqn.~\ref{eq:ammtt} follows directly from the All-Minors Matrix-Tree theorem as stated in \cite[Theorem~1]{nam:25}. Notice how the minor is a manifestly integrally positive polynomial in the edge labels, with an accompanying sign. When $k = 1$, Eqn.~\ref{eq:ammtt} corresponds to the classical Matrix-Tree theorem \cite{gunawardena:12,mirzaev:13}.

Now consider the case when $G$ has a single terminal SCC. As noted above, $\dim\ker\Ls(G) = 1$. Using the adjugate matrix of $\Ls(G)$, which is expressed in terms of the first minors of $\Ls(G)$, it is not difficult to determine a canonical basis vector, $\rhob \in \ker\Ls(G)$, whose entries are given by \cite{gunawardena:12}, 
\begin{equation}
    \rho_i = w(\Phi_{\{i\}}(G)) \,.
\label{eq:rho}
\end{equation}
It is easy to see that a vertex that is not in the terminal SCC cannot be the root of any spanning tree, while any vertex in the terminal SCC is the root of some spanning tree. It follows from Eqn.~\ref{eq:rho}, recalling the conventions mentioned above for the weight function, that $\rho_i \not= 0$ if, and only if, $i$ is in the terminal SCC. Since $\xb^* \in \ker\Ls(G)$, it follows by normalising to the total probability that the s.s.~probability of the corresponding Markov process is given by, 
\begin{equation}
    x^*_i = \frac{\rho_i}{\rho_1 + \dotsb + \rho_N} \,.
\label{eq:ss}
\end{equation}
As expected, $x^*_i > 0$ if, and only if, $i$ is in the terminal SCC. In this case, Eqn.~\ref{eq:rho} expresses $x^*_i$ as a manifestly integrally positive rational algebraic function of the edge labels \cite{gunawardena:12}. 

It follows from Eqn.~\ref{eq:ammtt} that the submatrix $\Ls(G)_{[\Abar,\Abar]}$ is invertible if, and only if, $\Phi_A(G)$ is non-empty. Furthermore, since an inverse matrix can be calculated in terms of minors through the adjugate matrix, it is not hard to see that an expression can be found for the inverse matrix, $\left(\Ls(G)_{[\Abar,\Abar]}\right)^{-1}$. However, stating this result requires some care because of the change of indexing that arises when passing to a submatrix. Given any finite set, $S = \{s_1, \dotsc, s_k\}$, whose indexing reflects the order of its elements, so that $s_1 < \dotsb < s_k$, define the ordering bijection, $\theta_S: \{ 1, \dotsc, k \} \to S$, by $\theta_S(i) = s_i$. Then, for any $1 \leq i, j \leq N-k$, 
\begin{equation}
    \left( \Ls(G)_{[\Abar,\Abar]} \right)^{-1}_{i,j}
    = -\frac{w(\Phi_{A \cup \{\theta_{\Abar}(j)\} \to A \cup \{\theta_{\Abar}(i)\}}(G))}{w(\Phi_A(G))} \,.
\label{eq:inv}
\end{equation}
Eqn.~\ref{eq:inv} is \cite[Proposition~2]{nam:25}. 

We can now move on to the transient quantities described above. Let $G$ be a graph whose terminal SCCs are singletons and let $Z \subseteq \Vs(G)$ be the set of terminal vertices. Choose $z \in Z$ and $u \in \Zbar$ such that $u \tos z$. Then the splitting probability from $u$ to $z$ is given by,
\begin{equation}
    \pi_{u,z}(G) = \frac{w(\Phi_{(Z \setminus \{z\}) \cup \{u\} \to Z}(G))}{w(\Phi_Z(G))} \,.
\label{eq:split}
\end{equation}
Eqn.~\ref{eq:split} is \cite[Theorem 4]{nam:25}. The $r$-th moment of the unconditional FPT from $u$ to any terminal vertex in $Z$ is given by, 
\begin{equation}
    \mu_{u,Z}^{(r)}(G) = r! \sum_{(i_1, \dotsc, i_r) \in \Zbar^r}{\left( \prod_{j=0}^{r-1}{\frac{w(\Phi_{Z \cup \{i_j\} \to Z \cup \{i_{j+1}\}}(G))}{w(\Phi_Z(G))}} \right)} \,,
\label{eq:moment-Z}
\end{equation}
where we have set $i_0 = u$. Eqn.~\ref{eq:moment-Z} is \cite[Theorem 6]{nam:25}. The $r$-th moment of the conditional FPT from $u$ to $z$ is given by, 
\begin{equation}
\begin{aligned}
    & \mu^{(r)}_{c,u,z}(G)  \\
    & \, = r! \sum_{(i_1, \dotsc, i_r) \in \Zbar^r}\left( \prod_{j=0}^{r-1}{\frac{w(\Phi_{Z \cup \{i_j\} \to Z \cup \{i_{j+1}\}}(G))}{w(\Phi_Z(G))}} \right) \left( \frac{w(\Phi_{(Z \setminus \{z\}) \cup \{i_r\} \to Z}(G))}{w(\Phi_{(Z \setminus \{z\}) \cup \{u\} \to Z}(G))} \right) \,.
\end{aligned}
\label{eq:moment}
\end{equation}
Here, as before, $i_0 = u$. Eqn.~\ref{eq:moment} is  \cite[Theorem 5]{nam:25}. We see from these results that all three transient quantities can be expressed as manifestly integrally positive rational algebraic functions of the edge labels. 

Finally, we shall also need a couple of additional elementary results. The first result shows how spanning forest weights of a graph $G$ may be built up recursively. Given $B \subseteq \Vs(G)$, $z \in B$ and $i \in \Bbar$, we have that,
\begin{equation}
    \sum_{j \in \Bbar : j \ra_G z}w(\Phi_{B \cup \{i\} \to B \cup \{j\}}(G)) \, \ell(j \to z) = w(\Phi_{(B \setminus \{z\}) \cup \{i\} \to B}(G)) \,.
\label{eq:span}
\end{equation}
Eqn.~\ref{eq:span} is \cite[Eqn.~24]{nam:25}. 

The second result is the matrix determinant lemma, which gives the determinant of a rank-one update of an invertible matrix \cite{harville}. Let $\Mb$ be an $n \times n$ invertible matrix and $\ub, \vb \in \R^n$ be $n \times 1$ column vectors. Then, 
\begin{equation}
    \det{\left( \Mb + \ub \vb^\T \right)} = \left( 1 + \vb^\T \Mb^{-1} \ub \right) \det{\Mb} \,.
\label{eq:det}
\end{equation}

With these preliminaries out of the way, we can now embark on the new results of the present paper. 

\subsection{Two graph operators}
\label{sec:transform}

As described in the Introduction, our reformulation and generalisation of the work of Hill and Kac rely on two new operators on graphs, which we define in this section. 

\subsubsection{The Hill operator}
\label{sec:hill-def}

Let $G$ be a graph with $T \geq 1$ singleton terminal SCCs, and let $Z$ be the set of terminal vertices. We assume that $N > T$, so that $\Zbar \neq \emptyset$ and there is some non-terminal vertex. Let $u \in \Zbar$ be any such vertex. The \emph{Hill construction} on $G$ with $u$ as source vertex, denoted $\Hs_u[G]$, which formalises the procedure outlined in \cite{hill:88}, is defined as follows. 

\vspace{0.5em}
\begin{definition}
The vertices of $\Hs_u[G]$ are the non-terminal vertices of $G$, $\Vs(\Hs_u[G]) = \Zbar$. The edges and labels of $\Hs_u[G]$ are acquired from those of $G$ through the following rules.
\begin{enumerate}[topsep=1em]
    \item If $j, k \in \Zbar$ are non-terminal vertices with $k \neq u$ ($j = u$ is allowed) and there is an edge $j \to_G k$, then there is an edge $j \to_{\Hs_u[G]} k$ with, 
    \[
        \ell(j \to_{\Hs_u[G]} k) = \ell(j \to_G k) \,.
    \]
    \item If $j \in \Zbar$ is a non-terminal vertex and there is an edge $j \ra_G u$, then there is an edge $j \to_{\Hs_u[G]} u$ with, 
    \[
        \ell(j \to_{\Hs_u[G]} u) = \ell(j \to_G u) + \sum_{j \ra_G z ,\, z \in Z} \ell(j \ra_G z) \,.
    \]
    \item If $j \in \Zbar$ is a non-terminal vertex other than $u$ ($j \neq u$) without an edge to $u$ ($j \not\to_G u$) but with at least one edge, $j \to_G z$, to some $z \in Z$, then there is an edge $j \ra_{\Hs_u[G]} u$ with, 
    \[
        \ell(j \ra_{\Hs_u[G]} u) = \sum_{j \ra_G z ,\, z \in Z} \ell(j \ra_G z) \,.
    \]
\end{enumerate}
There are no other edges in $\Hs_u[G]$. 
\label{def:hill}
\end{definition}
\vspace{0.5em}

Definition~\ref{def:hill} removes the terminal vertices in $Z$ and reroutes the edges entering $Z$ so that they enter the source vertex $u$ in $\Hs_u[G]$. Most of the other edges in $G$ remain in $\Hs_u[G]$ with the same labels (point 1). Two subtleties arise. First, a non-terminal vertex, $j \in \Zbar$, may have multiple edges entering $Z$, as well as, potentially, a pre-existing edge $j \ra_G u$. The labels on these edges must be summed to yield the label on the edge $j \ra u$ in $\Hs_u[G]$ (points 2 and 3). Second, the source vertex $u$ may also have one or more edges that enter $Z$ in $G$. These edges are ignored in $\Hs_u[G]$ but other edges leaving $u$ that do not enter $Z$ are preserved, along with their labels. We note that Hill did not discuss these subtleties in \cite{hill:88}.

Fig.~\ref{fig:hill} shows the Hill construction on the six-vertex graph $K$ on the left, with two terminal vertices, $2$ and $3$, where the source vertex is $1$. The graph $\Hs_1[K]$ on the right has four vertices. Note that $\Hs_1[K]$ inherits its vertex indexing from $K$. The label on the edge $4 \ra 1$ in $\Hs_1[K]$ illustrates the summation of the labels coming from the edge $4 \ra_K 1$ to the source vertex and from the edges $4 \ra_K 2$ and $4 \ra_K 3$ to the two terminal vertices (point 2). The edge $1 \ra_K 2$ disappears in $\Hs_1[K]$, along with its label $k_1$, but the edge $1 \ra_K 6$ remains in $\Hs_1[K]$ with its label $k_2$ (point 1). There are no edges $5 \ra 1$ and $6 \ra 1$ in $K$ but $\Hs_1[K]$ acquires these edges, by virtue of the edges $5 \ra 2$ and $6 \ra 3$ in $K$, and does so with the same labels, $k_7$ and $k_{10}$, respectively (point 3).

\begin{figure}
\centering
\includegraphics[trim={1.3in 4.7in 1.5in 4.3in}, clip, width=0.9\textwidth]{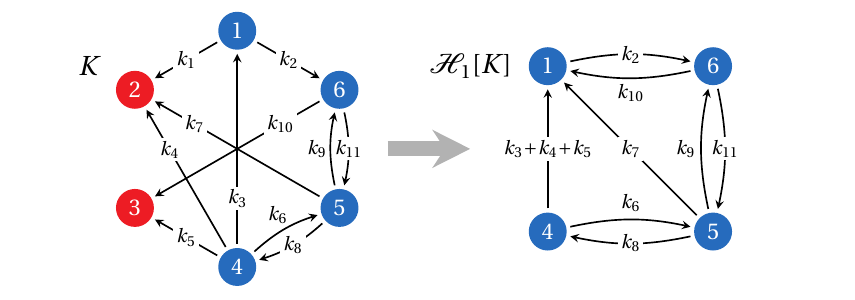}
\caption{The Hill construction. The graph $K$ on the left has source vertex $u = 1$ and two terminal vertices, $2$ and $3$, in red. The Hill construction on $K$ leads to the graph $\Hs_1[K]$ on the right, with $4$ vertices, which inherit their indexing from $K$. The edges and edge labels of $\Hs_1[K]$ illustrate the rules in Definition~\ref{def:hill}. See the text for more details.}
\label{fig:hill}
\end{figure}

It is easy to infer from Definition \ref{def:hill} the following two properties of $\Hs_u[G]$ that we will exploit below. Note that the properties are \emph{structural}, in that they do not concern the edge labels.

\vspace{0.5em}
\begin{lemma}
Under the conditions stated previously, if $i \in \Zbar$, then $i \tos_{\Hs_u[G]} u$, so that $\Hs_u[G]$ has one terminal SCC, which must include $u$.

\proof If $i = u$, then, by definition, $i \tos_{\Hs_u[G]} u$. If $i \not= u$, then $i \tos_G z$ for some $z \in Z$. By Definition~\ref{def:hill}.2 and 3, the last edge on any path from $i$ to $z$ in $G$ is rerouted to $u$ in $\Hs_u[G]$, so that $i \tos_{\Hs_u[G]} u$, as claimed. The final assertion follows in consequence. \hfill \qed
\label{lem:hill-scc}
\end{lemma}
\vspace{0.5em}

Recalling the discussion in Section \ref{sec:ss-trans}, Lemma \ref{lem:hill-scc} tells us that $\dim{\ker{\Ls(\Hs_u[G])}} = 1$, so that the s.s.~probabilities of the Markov process associated with $\Hs_u[G]$ can be obtained in terms of the canonical basis vector, $\rhob \in \ker\Ls(\Hs_u[G])$, as specified by Eqns.~\ref{eq:rho} and \ref{eq:ss}.

\vspace{0.5em}
\begin{lemma}
Under the conditions stated previously, if $k \in \Zbar$, then $u \tos_{\Hs_u[G]} k$ if, and only if, $u \tos_G k$.

\proof If $u = k$, the conclusion is trivial, so assume that $u \neq k$. If $u \tos_{\Hs_u[G]} k$, then there is a path, $u = j_1 \to \dotsb \to j_n = k$, of edges in $\Hs_u[G]$. We may assume that this path does not revisit $u$, so that $u \neq j_2, \dotsc, j_{n-1}$, for otherwise we can simply consider the portion of this path starting from the last visit to $u$. Since none of these edges now enter $u$, it follows from Definition~\ref{def:hill}.1 that they must have arisen from edges in $G$, so that $u \tos_G k$. Now suppose that $u \tos_G k$, so that there exists a directed path, $u = i_1 \to \dotsb \to i_m = k$, of edges in $G$. As before, we may assume that $u \neq i_2, \dotsc, i_m$, so that none of the edges enter $u$. Since $k$ is non-terminal, each vertex along the path must also be non-terminal. It then follows from Definition~\ref{def:hill}.1 that each edge must also exist in $\Hs_u[G]$, so that $u \tos_{\Hs_u[G]} k$.
\hfill \qed
\label{lem:hill-paths}
\end{lemma}
\vspace{0.5em}

At this point, we need to understand the relationship between $\Ls(G)$ and $\Ls(\Hs_u[G])$ that arises from the Hill construction. The former is an $N \times N$ matrix, while the latter is an $(N-T) \times (N-T)$ matrix. The vertex indices in $\Hs_u[G]$ are also derived from those in $G$ and therefore may not reflect the typical indexing, $1, \dotsc, N-T$. For example, the vertices of $\Hs_1[K]$ in Fig.~\ref{fig:hill} are $1, 4, 5, 6$. To move between the indices of $\Ls(\Hs_u[G])$ and the vertices of $\Hs_u[G]$, we have to use the ordering bijection, $\theta_{\Zbar}: \{1, \dotsc, N-T\} \ra \Zbar$, introduced previously. If we take a vertex $j \in \Zbar = \Vs(\Hs_u[G])$, then $1 \leq \theta^{-1}_{\Zbar}(j) \leq N-T$ is the corresponding row or column index in $\Ls(\Hs_u[G])$. Given $j,k \in \Zbar$ and interpreting the entries of the Laplacian matrix from Eqn.~\ref{eq:lap}, we see that, 
\[
    \Ls(\Hs_u[G])_{\theta_{\Zbar}^{-1}(k),\theta_{\Zbar}^{-1}(j)} = \begin{cases}
        \ell(j \to_{\Hs_u[G]} k) & \text{if $k \neq j$ and $j \to_{\Hs_u[G]} k$} \\
        0 & \text{if $k \neq j$ and $j \not\to_{\Hs_u[G]} k$} \\
        -\lambda_k(\Hs_u[G]) & \text{if $k = j$}\,.
    \end{cases}
\]
We can now use Definition~\ref{def:hill} to translate these terms, which are defined in $\Hs_u[G]$, into the corresponding terms defined in $G$. It is helpful to introduce a similar quantity to $\lambda_j$ in Eqn.~\ref{eq:lambda}, which arises from the labelling rules in Definition~\ref{def:hill}. Let $\lambda_{j,Z}$ be the sum of the outgoing edge labels from $j$ to $Z$ in $G$,
\begin{equation} 
    \lambda_{j,Z} = \sum_{j \ra_G z,\, z \in Z} \ell(j \ra z) \,.
\label{e-luz}
\end{equation}
If there are no edges from $j$ to $Z$, then, by the usual convention for empty sums, $\lambda_{j,Z} = 0$. It is worth noting that, unlike the quantity $\lambda_j$, which is defined for any graph, the quantity $\lambda_{j,Z}$ only makes sense for the graph $G$ with the set $Z$ of terminal vertices. 

Working through Definition \ref{def:hill}, we see that, 
\begin{equation}
\begin{split}
    \Ls(\Hs_u[G])_{\theta_{\Zbar}^{-1}(k),\theta_{\Zbar}^{-1}(j)}
    =
    \begin{cases}
        \ell(j \to_G k) & \text{if $k \not= j$, $k \neq u$, $j \to_G k$} \\
        0 & \text{if $k \not= j$, $k \neq u$, $j \not\to_G k$} \\
        \ell(j \to_G u) + \lambda_{j,Z} & \text{if $k \not=j$, $k = u$, $j \to_G u$} \\
        \lambda_{j,Z} & \text{if $k \not= j$, $k = u$, $j \not\to_G u$} \\
        -\lambda_u(G) + \lambda_{u,Z} & \text{if $k = j$, $k = u$} \\
        -\lambda_k(G) & \text{if $k = j$, $k \not= u$} \,.
    \end{cases}
\end{split}
\label{eq:lap-hill-off}
\end{equation}
Note how the penultimate case in Eqn.~\ref{eq:lap-hill-off}, for $k = j = u$,undertakes the compensation required by the loss in $\Hs_u[G]$ of the edges $u \ra_G z$ for $z \in Z$.

We can see from Eqn.~\ref{eq:lap-hill-off} that the entries of $\Ls(\Hs_u[G])$ are obtained from those of $\Ls(G)$ by the addition of terms $\lambda_{j,Z}$ to the row $\theta_{\Zbar}^{-1}(u)$, where $j$ varies with the column index. This suggests a way to construct $\Ls(\Hs_u[G])$ from $\Ls(G)$, as follows. Let us form the $(N-T) \times 1$ column vector, $\lambdab$, to carry the $\lambda_{j,Z}$ terms, keeping in mind that the ordering bijection is required to move between the non-terminal vertices and their indices, 
\begin{equation}
    \lambdab^\T = \left( \lambda_{\theta_{\Zbar}(1),Z}, \dotsc, \lambda_{\theta_{\Zbar}(N-T),Z} \right) \,;
\label{e-lmbt}
\end{equation}
and let us form the $(N-T) \times 1$ column vector, $\deltab_u$, with $1$ at index $\theta_{\Zbar}^{-1}(u)$ and $0$ elsewhere, 
\begin{equation}
    \deltab^\T_u = (0, \dotsc, 0, 1, 0, \dotsc, 0) \,,
\label{e-dltb}
\end{equation}
so as to single out the appropriate row; and, finally, let us remove the rows and columns of $\Ls(G)$ that come from $Z$, to see that, 
\begin{equation}
    \Ls(\Hs_u[G]) = \Ls(G)_{[\Zbar,\Zbar]} + \deltab_u \lambdab^\T \,.
\label{eq:lap-hill-rank1}
\end{equation}
The last term in Eqn.~\ref{eq:lap-hill-rank1} is the product of an $(N-T) \times 1$ matrix and a $1 \times (N-T)$ matrix, to give an $(N-T) \times (N-T)$ matrix with only a single nonzero row at index $\theta_{\Zbar}^{-1}(u)$. Eqn.~\ref{eq:lap-hill-rank1} therefore shows that $\Ls(\Hs_u[G])$ is a rank-one update of $\Ls(G)_{[\Zbar,\Zbar]}$.

Let us see how Eqn.~\ref{eq:lap-hill-rank1} works for the example in Fig.~\ref{fig:hill}. Here, $Z = \{2, 3\}$, $\Zbar = \{ 1, 4, 5, 6 \}$, and $u = 1$. As such, the ordering bijection is given by,
\[
    \theta_{\Zbar}(1) = 1, \qquad
    \theta_{\Zbar}(2) = 4, \qquad
    \theta_{\Zbar}(3) = 5, \qquad
    \theta_{\Zbar}(4) = 6.
\]
We can use Eqn.~\ref{eq:lap} to read off the Laplacian matrix from the graph $K$ in Fig.~\ref{fig:hill}, 
\[
    \Ls(K) = \left[ \begin{array}{cccccc}
        -k_1 - k_2 & 0 & 0 & k_3 & 0 & 0 \\
        k_1 & 0 & 0 & k_4 & k_7 & 0 \\
        0 & 0 & 0 & k_5 & 0 & k_{10} \\
        0 & 0 & 0 & -k_3 - k_4 - k_5 - k_6 & k_8 & 0 \\
        0 & 0 & 0 & k_6 & -k_7 - k_8 - k_9 & k_{11} \\
        k_2 & 0 & 0 & 0 & k_9 & -k_{10} - k_{11}
    \end{array} \right] \,.
\]
If we remove the rows and columns corresponding to the terminal vertices, this leaves the $4 \times 4$ matrix,
\[
    \Ls(K)_{[\Zbar,\Zbar]} = 
    \left[ \begin{array}{cccc}
        -k_1 - k_2 & k_3 & 0 & 0 \\
        0 & -k_3 - k_4 - k_5 - k_6 & k_8 & 0 \\
        0 & k_6 & -k_7 - k_8 - k_9 & k_{11} \\
        k_2 & 0 & k_9 & -k_{10} - k_{11}
    \end{array} \right] \,.
\]
Meanwhile, the vectors $\lambdab$ and $\deltab_u = \deltab_1$ are given by 
\[
    \lambdab = \left[ \begin{array}{c}
        k_1 \\
        k_4 + k_5 \\
        k_7 \\
        k_{10}
    \end{array} \right]
    \qquad \text{and} \qquad
    \deltab_1 = \left[ \begin{array}{c}
        1 \\
        0 \\
        0 \\
        0
    \end{array} \right].
\]
Therefore, the corresponding rank-one update is the $4 
\times 4$ matrix,
\[
\deltab_1 \lambdab^\T = \left[ \begin{array}{cccc}
        k_1 & k_4 + k_5 & k_7 & k_{10} \\
        0 & 0 & 0 & 0 \\
        0 & 0 & 0 & 0 \\
        0 & 0 & 0 & 0
    \end{array} \right] \,.
\]
If we add this to the submatrix $\Ls(K)_{[\Zbar,\Zbar]}$, following the prescription in Eqn.~\ref{eq:lap-hill-rank1}, we get the $4 \times 4$ matrix, 
\[
    \left[ \begin{array}{cccc}
        -k_2 & k_3 + k_4 + k_5 & k_7 & k_{10} \\
        0 & -k_3 - k_4 - k_5 - k_6 & k_8 & 0 \\
        0 & k_6 & -k_7 - k_8 - k_9 & k_{11} \\
        k_2 & 0 & k_9 & -k_{10} - k_{11}
    \end{array} \right] \,,
\]
which is precisely $\Ls(\Hs_1[K])$, as is easy to see by using Eqn.~\ref{eq:lap} on the graph $\Hs_1[K]$ in Fig.~\ref{fig:hill}.

The Hill operator that we have defined here differs in one detail from the construction that Hill originally sketched in \cite{hill:88}. Hill kept track of each ``terminal'' edge, $j \to_G z$, where $j \in \Zbar$ and $z \in Z$. Had we done the same, we would have needed a directed multigraph, in which there may be more than one directed edge from one vertex to another. Multigraphs are widely used but they are not a natural setting for the kinds of biochemical networks that inspired the linear framework \cite{gunawardena:12}, which relies on graphs. We have therefore adapted Hill's construction for this setting by, in effect, combining his parallel edges into a single edge, whose label is the sum of the labels of the individual edges, as specified in Definition~\ref{def:hill}.

\subsubsection{The unravelling operator}
\label{sec:unravel}

We now come to the unravelling construction. Unlike the Hill construction, which starts from a graph with terminal vertices and reduces its size by removing the terminal vertices, the unravelling construction starts from a strongly connected graph and increases its size by introducing a new terminal vertex. 
\vspace{0.5em}
\begin{definition}
Let $G$ be a strongly connected graph, and let $v \in \Vs(G)$. Then the unravelling of $G$ at vertex $v$, denoted $\Us_v[G]$, has vertices $\Vs(\Us_v[G]) = \Vs(G) \cup \{v^+\}$. The edges and labels of $\Us_v[G]$ are acquired from those of $G$ through the following rules. Choose $j, k \in \Vs(G)$. 
\begin{enumerate}
    \item If $j \to_G k$ with $k \neq v$, then $j \to_{\Us_v[G]} k$ and $\ell(j \to_{\Us_v[G]} k) = \ell(j \to_G k)$.
    \item If $j \to_G v$, then $j \to_{\Us_v[G]} v^+$ and $\ell(j \to_{\Us_v[G]} v^+) = \ell(j \to_G v)$.
\end{enumerate}
There are no other edges in $\Us_v[G]$. 
\label{def:unravel}
\end{definition}
\vspace{0.5em}

\begin{figure}
\centering
\includegraphics[trim={1.3in 4.7in 1.5in 4.3in}, clip, width=0.9\textwidth]{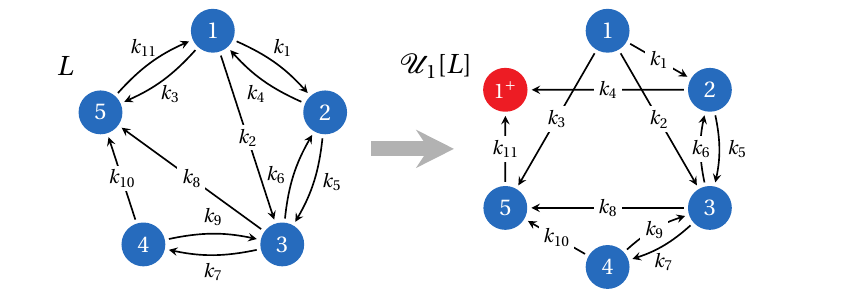}
\caption{The unravelling construction. The strongly connected graph $L$ on the left is unravelled at the vertex $v = 1$ to yield the graph, $\Us_1[L]$, on the right, with the new vertex $1^+$ in red. A stochastic trajectory starting from $1$ in $\Us_1[L]$ always reaches $1^+$ and remains there, mimicking a trajectory from $1$ in $L$ that returns to $1$ for the first time. See the text for more details.}
\label{fig:unravel}
\end{figure}

Definition~\ref{def:unravel} reroutes the edges entering $v$ in $G$ so that they now enter $v^+$ in $\Us_v[G]$ with the same label. It is clear that $v$ cannot be reached in $\Us_v[G]$ from any other vertex, so that $\{v\}$ is a SCC that is initial in the partial order. It is also clear that no edges leave $v^+$, so that $\{v^+\}$ is a terminal SCC. Moreover, because $G$ is strongly connected, if $k \in \Vs(G) \setminus \{v\}$, then $v \tos_{\Us_v[G]} k$ and $k \tos_{\Us_v[G]} v^+$. It follows that $\{v^+\}$ is the unique terminal SCC of $\Us_v[G]$. Note that, in general, it may not be the case that the vertices $\Vs(G) \setminus \{v\}$ form a single SCC in $\Us_v[G]$.

Now, let us compare stochastic trajectories of the Markov processes on $G$ and on $\Us_v[G]$, each starting at the vertex $v$. It is not difficult to see, using similar reasoning to that in \cite{nam:25}, that the distribution of the recurrence time to $v$ in $G$ is identical to that of the FPT from $v$ to $v^+$ in $\Us_v[G]$. It follows that the $r$-th moments are equal, 
\begin{equation}
    \tau_v^{(r)}(G) = \mu_{v,v^+}^{(r)}(\Us_v[G]) \,, 
\label{eq:recur-unravel}
\end{equation}
so that we can use the results on FPTs from our prequel paper \cite{nam:25}, as summarised above, to calculate recurrence time moments.

It is also not difficult to see how unravelling affects the Laplacian. If $G$ is indexed in the usual way, so that $\Vs(G) = \{1, \dotsc, N\}$, then we can assign the index $\np$ to the new vertex $v^+$. To form the $(\np) \times (\np)$ matrix $\Ls(\Us_v[G])$, following Definition~\ref{def:unravel}, we enlarge the $N \times N$ matrix $\Ls(G)$ by adjoining a new row with index $\np$ and a new column with index $\np$, both with zero entries, and then interchange the off-diagonal elements of row $v$ with those of row $\np$. We can see this at work in the example graph $L$ in Fig.~\ref{fig:unravel}. The Laplacian matrix, $\Ls(L)$, can be read off from Eqn.~\ref{eq:lap} to be, 
\begin{equation}
    \left[ \begin{array}{ccccc}
        -(k_1+k_2+k_3) & k_4   & 0   & 0   & k_{11} \\
        k_1 & -(k_4+k_5) & k_6 & 0 & 0 \\
        k_2 & k_5 & -(k_6+k_7+k_8) & k_9 & 0 \\
        0   & 0   & k_7 & -(k_9+k_{10})  & 0 \\
        k_3 & 0   & k_8 & k_{10} & -k_{11} \\ 
    \end{array} \right] \,.
\label{e-ll}
\end{equation}
If we follow the procedure described above (with $v = 1$), we get the matrix,
\begin{equation}
    \left[ \begin{array}{cccccc}
        -(k_1+k_2+k_3) & 0   & 0   & 0   & 0  & 0 \\
        k_1 & -(k_4+k_5) & k_6 & 0 & 0   & 0 \\
        k_2 & k_5 & -(k_6+k_7+k_8) & k_9 & 0  & 0 \\
        0   & 0   & k_7 & -(k_9+k_{10})  & 0  & 0 \\
        k_3 & 0   & k_8 & k_{10} & -k_{11} & 0 \\
        0   & k_4 & 0   & 0  & k_{11} & 0 \\ 
    \end{array} \right] \,.
\label{e-lusl}
\end{equation}
Reading off from Eqn.~\ref{eq:lap}, this is exactly $\Ls(\Us_1[L])$. 

We noted above that $\{v^+\}$ is the unique terminal SCC of $\Us_v[G]$ and is also a singleton. Accordingly, $\Us_v[G]$ satisfies the conditions required above for the Hill construction and is in the domain of the Hill operator. What happens if we compose the Hill operator with the unravelling operator? If $G$ is indexed in the usual way, so that $\Vs(G) = \{1, \dotsc, N\}$, and the new terminal vertex in $\Us_v[G]$ is $v^+ \equiv \np$, then $Z = \{ \np \}$, $\Zbar = \{1, \dotsc, N\}$ and the ordering bijection, $\theta_{\Zbar}$, is the identity. Here, the bar notation in $\overline{Z}$ denotes the set complement in $\Vs(\Us_v[G])$. We can compute the Laplacian of $\Hs_v[\,\Us_v[G]]$ using the rank-one update formula in Eqn.~\ref{eq:lap-hill-rank1}, with $u = v$. The vectors $\deltab_v$ and $\lambdab$ can be determined as follows. Eqn.~\ref{e-luz} tells us that if $j \ra_G v$, so that $j \ra_{\Us_v[G]} \np$, then 
\[
    \lambda_{j,Z} = \ell(j \ra_{\Us_v[G]} \np) = \ell(j \ra_G v) \,,
\]
while $\lambda_{j,Z} = 0$ if $j \not\ra_G v$. Hence, by virtue of Eqn.~\ref{eq:lap}, the vector $\lambdab$
is simply the $v$-th row of $\Ls(G)$ but with $0$ at index $v$ in place of the usual negative diagonal element of $\Ls(G)$. As for $\deltab_v$, it has $1$ at index $v$ and $0$ elsewhere. 

Now, consider what happens when we first unravel $G$ from $v$, and change $\Ls(G)$ as described above to form $\Ls(\Us_v[G])$, and then undertake the rank-one update on $\Ls(\Us_v[G])$, as described in Eqn.~\ref{eq:lap-hill-rank1} to form $\Ls(\Hs_v[\, \Us_v[G]])$. We first remove the $(\np)$-th row and column of $\Ls(\Us_v[G])$, the former of which contained the off-diagonal elements of row $v$ of $\Ls(G)$ that were interchanged into row $\np$ when forming $\Ls(\Us_v[G])$. The rank-one update then adds back these entries into row $v$ of $\Ls(\Hs_v[\, \Us_v[G]])$, omitting the $(\np)$-th entry, which is zero. We therefore recover $\Ls(G)$,
\[
    \Ls(\Hs_v[\, \Us_v[G]]) = \Ls(G) \,.
\]
Since the Laplacian matrix defines the graph, we see that the Hill operator is a left-inverse of the unravelling operator.

\vspace{0.5em}
\begin{lemma}
If $G$ is a strongly connected graph and $v \in \Vs(G)$, then $\Hs_v[\, \Us_v[G]] = G$. 
\label{t-hug}
\end{lemma}
\vspace{0.5em}

We can see Lemma~\ref{t-hug} at work on the graph $L$ in Fig.~\ref{fig:unravel}, with $v = 1$. The rank-one update for $\Ls(\Us_1[L])$, shown in Eqn.~\ref{e-lusl}, can be easily calculated from Eqns.~\ref{e-lmbt} and \ref{e-dltb} to be,
\[
\deltab_1 \lambdab^\T = \left[ \begin{array}{ccccc}
        0 & k_4 & 0 & 0 & k_{11} \\
        0 & 0 & 0 & 0 & 0 \\
        0 & 0 & 0 & 0 & 0 \\
        0 & 0 & 0 & 0 & 0 \\
    \end{array} \right] \,.
\]
Adding this matrix to the submatrix of $\Ls(\Us_1[L])$ that is defined by the vertices in $\Zbar = \{1, \dotsc, 5\}$ recovers $\Ls(L)$, shown in Eqn.~\ref{e-ll}, as expected from Lemma~\ref{t-hug}. 

\subsection{The forest-tree lemma}
\label{sec:hill-lemma}

We need two further preliminary results for the Hill operator, of which the first relates the weight of spanning trees in $\Hs_u[G]$ to the weight of a corresponding collection of spanning forests in $G$. We call it the \emph{forest-tree lemma} in consequence. Let $G$ be a graph whose terminal SCCs are singletons. Let $\emptyset \neq Z \subsetneq \Vs(G)$ be the set of terminal vertices, with $T = \# Z$, and let $u \in \Zbar = \Vs(G) \setminus Z$ be a source vertex.

\vspace{0.5em}
\begin{lemma}[Forest-tree lemma]
Given $j \in \Zbar$, we have 
\[
    w(\Phi_{\{j\}}(\Hs_u[G])) = w(\Phi_{Z \cup \{u\} \to Z \cup \{j\}}(G)) \,.
\]
\label{lem:hill}
\end{lemma}
\vspace{-1.8em}
\begin{proof}
Suppose, without loss of generality, that $u = 1$ and $Z = \left\{ N-T+1, \dotsc, N \right\}$, so that $\Zbar = \Vs(\Hs_1[G]) = \left\{ 1, \dotsc, N-T \right\}$. Let us also define $Y = \Zbar \setminus \{j\} = \overline{Z \cup \{j\}}$. It will be convenient for dealing with signs to use the negative transpose, 
\begin{equation}
    \Lb(G) = -\Ls(G)^\T \,,
\label{eq:lap-t}
\end{equation}
in lieu of the usual Laplacian matrix $\Ls(G)$. Noting that $\Hs_1[G]$ has $N - T$ vertices, the AMMTT in the form of Eqn.~\ref{eq:ammtt} then implies that, 
\[
    w(\Phi_{\{j\}}(\Hs_1[G])) = \left( -1 \right)^{N-T-1} \det{\Ls(\Hs_1[G])_{[\Zbar \setminus \{j\}, \Zbar \setminus \{j\}]}} = \det{\Lb(\Hs_1[G])_{[Y,Y]}}\,.
\]
Therefore, we want to show that 
\[
    \det{\Lb(\Hs_1[G])_{[Y,Y]}} = w(\Phi_{Z \cup \{1\} \to Z \cup \{j\}}(G)) \,.
\]
If we take the negative transpose of Eqn.~\ref{eq:lap-hill-rank1} and set $u = 1$, we see that, 
\begin{equation}
    \Lb(\Hs_1[G]) = \Lb(G)_{[\Zbar,\Zbar]} - \lambdab \, \deltab_1^\T \,.
\label{e-lh1}
\end{equation}
Since $\Zbar = \{1, \dotsc, N-T\}$, the ordering bijection $\theta_{\Zbar}$ is the identity and $\theta_{\Zbar}^{-1}(u) = 1$, so it follows from Eqn.~\ref{e-dltb} that $\deltab_1^\T = (1, 0, \dotsc, 0)$. Taking the $[Y,Y]$ submatrix of Eqn.~\ref{e-lh1} and recalling that $Y \subset \Zbar$, we see that, 
\begin{equation}
    \Lb(\Hs_1[G])_{[Y,Y]} = \Lb(G)_{[Y,Y]} - \lambdab_{[Y]} \, (\deltab_1)_{[Y]}^\T \,.
\label{eq:lap-t-hill-rank1}
\end{equation}
Let us consider first the case when $j = u = 1$. Then $Y = \{2, \dotsc, N-T\}$, so that,
\[ (\deltab_1)_{[Y]}^\T = \mathbf{0} \,.\]
Taking determinants, it follows from Eqn.~\ref{eq:lap-t-hill-rank1} that, 
\[
    \det{\Lb(\Hs_1[G])_{[Y,Y]}} = \det{\Lb(G)_{[Y,Y]}} \,.
\]
Since $\overline{Y} = Z \cup \{1\}$, the AMMTT in the form of Eqn.~\ref{eq:ammtt} tells us that, 
\[
    \det{\Lb(G)_{[Y,Y]}}
    = \left( -1 \right)^{N-T-1} \det{\Ls(G)_{[Y,Y]}}
    = w(\Phi_{Z \cup \{1\}}(G)) \,.
\]
Putting the pieces together, we see that,
\[ w(\Phi_{\{1\}}(\Hs_1[G])) = w(\Phi_{Z \cup \{1\}}(G)) \,. \]
Recalling that $\Phi_{A \ra A}(G) = \Phi_A(G)$, this proves the case when $j = u = 1$. 

We now consider the case when $j > u = 1$. If we apply the matrix determinant lemma (Eqn.~\ref{eq:det}) to Eqn.~\ref{eq:lap-t-hill-rank1} with,
\[
    \Mb = \Lb(G)_{[Y,Y]}\,, \qquad
    \ub = -\lambdab_{[Y]}\,, \qquad
    \vb = \left( \deltab_1 \right)_{[Y]} \,,
\]
and note that, since $j > 1$, $(\deltab_1)_{[Y]}^\T = (1, 0, \dotsc, 0)$, and we obtain, 
\[
    \det{\Lb(\Hs_1[G])}_{[Y,Y]} = \left( 1 - \sum_{k=1}^{N-T-1}{\left( \Lb(G)_{[Y,Y]} \right)_{1,k}^{-1} \, \left( \lambdab_{[Y]} \right)_k} \right) \det{\Lb(G)_{[Y,Y]}} \,. 
\]
The AMMTT in the form of Eqn.~\ref{eq:ammtt} tells us that the right-hand minor is given by 
\[
    \det{\Lb(G)_{[Y,Y]}}
    = \left( -1 \right)^{N-T-1} \det{\Ls(G)_{[Y,Y]}}
    = w(\Phi_{Z \cup \{j\}}(G)),
\]
whereas Eqn.~\ref{eq:inv} tells us that 
\[
    \left( \Lb(G)_{[Y,Y]} \right)_{1,k}^{-1}
    = -\left( \Ls(G)_{[Y,Y]} \right)_{k,1}^{-1}
    = \frac{w(\Phi_{Z \cup \{\theta_Y(1), j\} \to Z \cup \{\theta_Y(k), j\}}(G))}{w(\Phi_{Z \cup \{j\}}(G))}, 
\]
where the ordering bijection, $\theta_Y : \left\{ 1, \dotsc, N-T-1 \right\} \to Y = \Zbar \setminus \{j\}$, is given by  
\[
    \theta_Y(k) = \begin{cases}
        k & \text{if $k < j$} \\
        k + 1 & \text{if $k \geq j$.}
    \end{cases}
\]
Meanwhile, the $k$-th entry of $\lambdab_{[Y]}$ is simply
\[
    \left( \lambdab_{[Y]} \right)_k = \lambda_{\theta_Y(k), Z}.
\]
Putting these pieces together and noting that, since $j > 1$, $\theta_Y(1) = 1$, we get 
\[
    \det{\Lb(\Hs_1[G])}_{[Y,Y]}
    = w(\Phi_{Z \cup \{j\}}(G)) - \sum_{k=1}^{N-T-1}{\lambda_{\theta_Y(k),Z} \, w(\Phi_{Z \cup \{1,j\} \to Z \cup \{\theta_Y(k),j\}}(G))} \,,
\]
which can clearly be re-indexed to yield 
\[
    \det{\Lb(\Hs_1[G])}_{[Y,Y]} = w(\Phi_{Z \cup \{j\}}(G)) - \sum_{k \in Y}{\lambda_{k,Z} \, w(\Phi_{Z \cup \{1,j\} \to Z \cup \{k,j\}}(G))} \,.
\]
We now rearrange the right-hand sum, as follows: 
\begin{align*}
    & \det{\Lb(\Hs_1[G])}_{[Y,Y]} \\
    & \qquad = w(\Phi_{Z \cup \{j\}}(G)) - \sum_{k \in Y}{\left( \sum_{z \in Z : k \to_G z}{\ell(k \to_G z)} \right) w(\Phi_{Z \cup \{1,j\} \to Z \cup \{k,j\}}(G))} \\
    & \qquad = w(\Phi_{Z \cup \{j\}}(G)) - \sum_{z \in Z}{\,\sum_{k \in Y : k \to_G z}{\ell(k \to_G z) \, w(\Phi_{Z \cup \{1,j\} \to Z \cup \{k,j\}}(G))}}\,,
\end{align*}
at which point we can apply Eqn.~\ref{eq:span} to the inner sum, to get  
\[
    \det{\Lb(\Hs_1[G])}_{[Y,Y]}
    = w(\Phi_{Z \cup \{j\}}(G)) - \sum_{z \in Z}{w(\Phi_{((Z \cup \{j\}) \setminus \{z\}) \cup \{1\} \to Z \cup \{j\}}(G))} \,.
\]
Each term in the right-hand sum is the weight of all spanning forests of $G$ rooted at $Z \cup \{j\}$ with a path from $1$ to $z$. Therefore, the right-hand sum is the weight of all spanning forests of $G$ rooted at $Z \cup \{j\}$ with a path from $1$ to any of the terminal vertices. Subtracting this from the weight of all spanning forests rooted at $Z \cup \{j\}$, we obtain the weight of the subset of such spanning forests with a path from $1$ to $j$: 
\[
    \det{\Lb(\Hs_1[G])}_{[Y,Y]}
    = w(\Phi_{Z \cup \{1\} \to Z \cup \{j\}}(G)) \,.
\]
This yields the desired result. 
\end{proof}

\subsection{The exchange formula}
\label{sec:exchange}

We now turn to the second of our two fundamental lemmas, which we call the \emph{exchange formula}. It will allow us to move between different Hill constructions, by converting weights of spanning trees in one construction to weights of spanning trees in the other construction, at the expense of introducing \emph{exchange factors}. We again let $G$ be a graph whose terminal SCCs are singletons. Let $\emptyset \neq Z \subsetneq \Vs(G)$ be the set of terminal vertices, with $T = \#Z$. Let $u \in \Zbar$ be a source vertex. We first need to clarify when a Hill construction has spanning trees rooted at a given vertex.

\vspace{0.5em}
\begin{lemma}
Let $k \in \Zbar$. $\Phi_{\{k\}}(\Hs_u[G])$ is non-empty if, and only if, $u \tos_G k$.
\label{l-rst0}
\end{lemma}

\begin{proof}
Suppose that $u \tos_G k$. It follows from Lemma~\ref{lem:hill-paths} that $u \tos_{\Hs_u[G]} k$. Now choose any vertex $w \in \Vs(\Hs_u[G]) = \Zbar$. Since $Z$ is the set of terminal vertices of $G$, there must be a path in $G$ from $w$ to some terminal vertex in $Z$. Again as a consequence of Definition~\ref{def:hill}, this path in $G$ gives rise to a corresponding path to $u$ in $\Hs_u[G]$, so that $w \tos_{\Hs_u[G]} u$. By concatenation, $w \tos_{\Hs_u[G]} k$. Hence, every vertex in $\Zbar = \Vs(\Hs_u[G])$ has a path to $k$, from which it readily follows that there is a spanning tree of $\Hs_u[G]$ rooted at $k$. Hence, $\Phi_{\{k\}}(\Hs_u[G]) \not= \emptyset$. Conversely, if there exists a spanning tree of $\Hs_u[G]$ that is rooted at $k$, then, evidently, $u \tos_{\Hs_u[G]} k$. It then follows from Lemma~\ref{lem:hill-paths} that $u \tos_G k$, as claimed. 
\end{proof}

The exchange factors can now be defined as the obvious ratios of weights. 

\vspace{0.5em}
\begin{definition}
    Let $u, v \in \Zbar$ be source vertices, which need not be distinct. Let $k \in \Zbar$ be such that $u \tos_G k$. The exchange factor $\psi_{k|v,u}$ for the graph $G$ is defined by,
    \[ \psi_{k|v,u} = \frac{w(\Phi_{\{k\}}(\Hs_v[G]))}{w(\Phi_{\{k\}}(\Hs_u[G]))} \,.\]
\label{lem:exchange-def}
\end{definition}

Exchange factors are well-defined by Lemma~\ref{l-rst0}, and it should be remembered that they are defined for the graph $G$. By definition, exchange factors are manifestly integrally positive rational algebraic functions of the edge labels. Let us see what an exchange factor looks like for the graph, $K$, in Fig.~\ref{fig:exchange}A. This five-vertex graph has two terminal vertices, $4$ and $5$, so that the Hill constructions $\Hs_1[K]$ and $\Hs_2[K]$ both have three vertices, as shown. We see that $1 \tos_K 3$, so that the exchange factor $\psi_{3|2,1}$ is defined. Fig.~\ref{fig:exchange}B shows the spanning trees rooted at $3$ of both $\Hs_1[K]$ and $\Hs_2[K]$, from which we conclude that, 
\begin{equation}
    \psi_{3|2,1} = \frac{w(\Phi_{\{3\}}(\Hs_2[K]))}{w(\Phi_{\{3\}}(\Hs_1[K]))} = \frac{k_1k_2}{k_1k_3 + k_1k_2} = \frac{k_2}{k_2 + k_3}.
\label{eq:psi-ex1}
\end{equation}
On the other hand, $1 \not\tos_K 2$ and there are no spanning trees of $\Hs_1[K]$ that are rooted at $2$. There is, therefore, no exchange factor $\psi_{2|v,1}$ for any $v$.

Note that, in Eqn.~\ref{eq:psi-ex1}, $\psi_{3|2,1}$ is expressed solely in terms of the labels in $\Hs_1[K]$, and makes no use of the labels in $\Hs_2[K]$. Definition~\ref{lem:exchange-def} would have little point, except that we will show more generally that $\psi_{k|v,u}$ may be calculated solely in terms of spanning forests of $\Hs_u[G]$, together with some labels of $G$, but with no involvement of $\Hs_v[G]$. We now embark on that calculation.

\begin{figure}
\centering
\includegraphics[trim={1.7in 3.9in 1.8in 3.2in}, clip, width=0.9\textwidth]{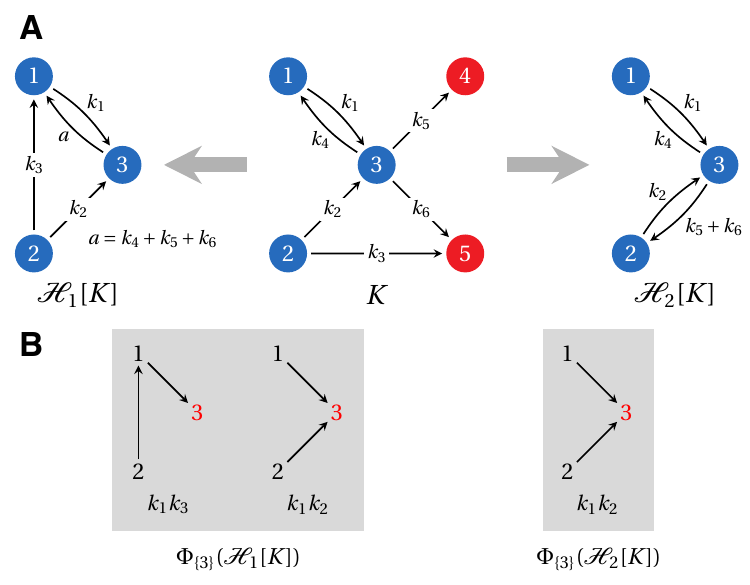}
\caption{Exchange factors. (\textbf{A}) \emph{Center:} A five-vertex graph, $K$, with two terminal vertices, $4$ and $5$ (red). \emph{Left and right:} The Hill constructions, $\Hs_1[K]$ and $\Hs_2[K]$. (\textbf{B}) The two spanning trees of $\Hs_1[K]$ rooted at $3$ (\emph{left}) and the one spanning tree of $\Hs_2[K]$ rooted at $3$ (\emph{right}).
}
\label{fig:exchange}
\end{figure}

We will make the same assumptions as in Definition~\ref{lem:exchange-def}. As in the proof of Lemma \ref{lem:hill}, we will assume, without any loss of generality, that $Z = \left\{ N-T+1, \dotsc, N \right\}$. Taking the negative transpose, as before, the rank-one update formula (Eqn.~\ref{eq:lap-hill-rank1}) tells us that 
\[
    \Lb(\Hs_u[G]) = \Lb(G)_{[\Zbar,\Zbar]} - \lambdab \, \deltab_u^\T
    \quad \text{and} \quad
    \Lb(\Hs_v[G]) = \Lb(G)_{[\Zbar,\Zbar]} - \lambdab \, \deltab_v^\T \,.
\]
Therefore, 
\[
    \Lb(\Hs_v[G]) = \Lb(\Hs_u[G]) + \lambdab \left( \deltab_u - \deltab_v \right)^\T \,.
\]
Now choose $k \in \Zbar = \left\{ 1, \dotsc, N-T \right\}$ such that $u \tos_G k$ and let $Y = \Zbar \setminus \{k\}$. Taking the determinant of the $[Y,Y]$ submatrix, we see that, 
\[
    \det{\Lb(\Hs_v[G])_{[Y,Y]}} = \det{\left( \Lb(\Hs_u[G])_{[Y,Y]} + \lambdab_{[Y]} (\deltab_u - \deltab_v)_{[Y]}^\T \right)} \,.
\]
By assumption, $u \tos_G k$, so, according to Lemma~\ref{l-rst0}, $\Phi_{\{k\}}(\Hs_u[G]) \not= \emptyset$. If we consider $Y$ as a subset of $\Vs(\Hs_u[G]) = \Zbar$, so that $\overline{Y} = \Zbar \setminus Y = \{k\}$ and $\overline{\{k\}} = Y$, it follows from the AMMTT (Eqn.~\ref{eq:ammtt}) that $\Lb(\Hs_u[G])_{[Y,Y]}$ is invertible. Applying the matrix determinant lemma (Eqn.~\ref{eq:det}) and reorganising,
\[
    \frac{\det\Lb(\Hs_v[G])_{[Y,Y]}}{\det\Lb(\Hs_u[G])_{[Y,Y]}} = 1 + \left( \deltab_u - \deltab_v \right)_{[Y]}^\T \left( \Lb(\Hs_u[G])_{[Y,Y]} \right)^{-1} \lambdab_{[Y]} \,.
\]
It then follows directly from the AMMTT (Eqn.~\ref{eq:ammtt}) that our exchange factor is given by, 
\[
    \psi_{k|v,u} = 1 + (\deltab_u - \deltab_v)_{[Y]}^\T \left( \Lb(\Hs_u[G])_{[Y,Y]} \right)^{-1} \lambdab_{[Y]} \,.
\]
It remains only to evaluate the matrix product on the right-hand side. At this point, we have to consider three rather annoying corner cases: $u = v$, $k = u \not= v$ and $k = v \not= u$. To avoid needlessly complicating the exposition, we leave it as a straightforward exercise to the reader to confirm that the formula that we derive here for the general case---$u \not= v$, $k \not= u$ and $k \not= v$---also gives the correct results in these corner cases. We will return to these cases below.

Recall that $\theta_Y: \{1, \dotsc, N-T-1\} \ra Y$ is the ordering bijection, introduced before Eqn.~\ref{eq:inv}. Accordingly, in the general case, the column vector $(\deltab_u - \deltab_v)_{[Y]}$ has $1$ for its $\theta_Y^{-1}(u)$-th entry and $-1$ for its $\theta_Y^{-1}(v)$-th entry, while the column vector $\lambdab_{[Y]}$ has entries $\lambda_{\theta_Y(m),Z}$ for $1 \leq m \leq N-T-1$. Here, $\lambda_{\theta_Y(m),Z}$ is the sum of the edge labels in $G$ from $\theta_Y(m)$ to all the terminal vertices, as defined in Eqn.~\ref{e-luz}. It follows that,
\begin{equation}
\begin{aligned}
    & (\deltab_u - \deltab_v)_{[Y]}^\T \left( \Lb(\Hs_u[G])_{[Y,Y]} \right)^{-1} \lambdab_{[Y]} \\
    & \quad = \sum_{m=1}^{N-T-1} \left(\left(\Lb(\Hs_u[G])_{[Y,Y]} \right)^{-1}_{\theta_Y^{-1}(u),m} - \left(\Lb(\Hs_u[G])_{[Y,Y]}\right)^{-1}_{\theta_Y^{-1}(v),m}\right) \lambda_{\theta_Y(m),Z} \,.
\end{aligned}
\label{e-dudv}
\end{equation}
The formula for the inverse of a Laplacian submatrix (Eqn.~\ref{eq:inv}) tells us that, using $w$ for either $u$ or $v$ and recalling again that $\overline{Y} = \{k\}$,
\[
    \left( \Lb(\Hs_u[G])_{[Y,Y]} \right)^{-1}_{\theta_Y^{-1}(w),m}
    = \frac{w(\Phi_{\{k,w\} \to \{k,\theta_Y(m)\}}(\Hs_u[G]))}{w(\Phi_{\{k\}}(\Hs_u[G]))} \,.
\]
This formula takes us from matrices on the right to forests on the right, so, after substituting, we can re-index with $m \in Y$ rather than $1 \leq m \leq N-T-1$, and rewrite the matrix product in Eqn.~\ref{e-dudv} as, 
\begin{align*}
    & (\deltab_u - \deltab_v)_{[Y]}^\T \left( \Lb(\Hs_u[G])_{[Y,Y]} \right)^{-1} \lambdab_{[Y]} \\
    & \quad = \sum_{m \in Y}{\left( \frac{w(\Phi_{\{k,u\} \to \{k,m\}}(\Hs_u[G])) - w(\Phi_{\{k,v\} \to \{k,m\}}(\Hs_u[G]))}{w(\Phi_{\{k\}}(\Hs_u[G]))} \right) \lambda_{m,Z}} \,.
\end{align*}
Putting the pieces together, we arrive at the following conclusion. 

\vspace{0.5em}
\begin{lemma}[Exchange formula]
Under the assumptions made above, the exchange factor $\psi_{k|v,u}$ of the graph $G$ is given by, 
\begin{equation}
1 + \sum_{m \in \Zbar \setminus \{k\}}{\left( \frac{w(\Phi_{\{k,u\} \to \{k,m\}}(\Hs_u[G])) - w(\Phi_{\{k,v\} \to \{k,m\}}(\Hs_u[G]))}{w(\Phi_{\{k\}}(\Hs_u[G]))} \right) \lambda_{m,Z}} \,.
\label{e-excf}
\end{equation}
\label{lem:exchange}
\end{lemma}

The salient point here, as mentioned above, is that the expression in Eqn~\ref{e-excf} depends solely on the spanning forests of $\Hs_u[G]$, together with the aggregated labels, $\lambda_{m,Z}$ of $G$, and not on any quantities derived from $\Hs_v[G]$.

The expression in Eqn.~\ref{e-excf} is not manifestly integrally positive, although $\psi_{k|v,u}$ does have a manifestly integrally positive expression through Definition~\ref{lem:exchange-def}. This tells us that there must be cancellations embedded within the expression for $\psi_{k|v,u}$ in Eqn~\ref{e-excf}. These cancellations are easy to find in the corner cases mentioned above, which allows us also to see what these cases look like when derived from Eqn.~\ref{e-excf}. If $u = v$, then the numerator within the summation in Eqn.~\ref{e-excf} becomes zero and so $\psi_{k|u,u} = 1$, as expected from Definition~\ref{lem:exchange-def}. If $k = v \not= u$, then the right-hand set of spanning forests within the summation, $\Phi_{\{k,v\} \ra \{k,m\}}(\Hs_u[G])$, is not properly defined because $\{k,v\} = \{v\}$ has size $1$ and $\{k,m\} = \{v,m\}$ has size $2$. In particular, if $v$ is a root, it cannot have a path to the other root $m$. We can hence treat $\Phi_{\{k,v\} \ra \{k,m\}}(\Hs_u[G])$ as empty, so that $w(\Phi_{\{k,v\} \ra \{k,m\}}(\Hs_u[G])) = 0$. The expression for the exchange factor then simplifies to,
\[
    \psi_{v|v,u} = 1 + \sum_{m \in Y}{\left( \frac{w(\Phi_{\{v,u\} \to \{v,m\}}(\Hs_u[G]))}{w(\Phi_{\{v\}}(\Hs_u[G]))} \right) \lambda_{m,Z}} \,,
\]
which is manifestly integrally positive. The more interesting situation arises when $k = u \not= v$, for which, by the same reasoning, 
\begin{equation} 
\psi_{u|v,u} = 1 - \sum_{m \in Y} {\left( \frac{w(\Phi_{\{u,v\} \to \{u,m\}}(\Hs_u[G]))}{w(\Phi_{\{u\}}(\Hs_u[G]))} \right) \lambda_{m,Z}} \,.
\label{e-puvu}
\end{equation}
How do the cancellations manifest themselves in this other special case?

We can appeal to the elementary result from our prequel paper \cite[Eqn.~24]{nam:25}, which is repeated in Eqn.~\ref{eq:span}. Applying this result to $\Hs_u[G]$ and taking $z = u$, $B = \{u\}$ and $i = v \in \overline{B} = \Vs(\Hs_u[G]) \setminus \{u\}$, Eqn~\ref{eq:span} tells us that, 
\[
    w(\Phi_{\{u\}}(\Hs_u[G])) = \sum_{j \in \overline{B}, j \ra_{\Hs_u[G]} u} w(\Phi_{\{u,v\} \ra \{u,j\}}(\Hs_u[G])) \, \ell(j \ra_{\Hs_u[G]} u) \,.
\]
However, because $k = u$, we have $\overline{B} = Y$, so we may rewrite the right-hand side of the formula above using the same index $m \in Y$ as we used in Eqn.~\ref{e-puvu} to get, 
\[
    \sum_{m \in Y, m \ra_{\Hs_u[G]} u} w(\Phi_{\{u,v\} \ra \{u,m\}}(\Hs_u[G])) \, \ell(m \ra_{\Hs_u[G]} u) \,.
\]
If we now divide by $w(\Phi_{\{u\}}(\Hs_u[G]))$ on both sides, we get, 
\begin{equation} 
    1 = \sum_{m \in Y, m \ra_{\Hs_u[G]} u} \left( \frac{w(\Phi_{\{u,v\} \ra \{u,m\}}(\Hs_u[G]))}{w(\Phi_{\{u\}}(\Hs_u[G]))}\right)\ell(m \ra_{\Hs_u[G]} u) \,.
\label{e-mym}
\end{equation}
The summation term in Eqn.~\ref{e-mym} is suspiciously close to that in Eqn.~\ref{e-puvu}, but Eqn.~\ref{e-puvu} has a seemingly larger range than Eqn.~\ref{e-mym}. However, referring back to Definition~\ref{def:hill}, if $m \not\ra_{\Hs_u[G]} u$, then $m$ evidently cannot have an edge to a terminal vertex in $G$, so that $\lambda_{m,Z} = 0$. Hence, the range of summation can be taken in both cases to be the same: $m \in Y, \, m \ra_{\Hs_u[G]} u$. Substituting Eqn.~\ref{e-mym} into Eqn.~\ref{e-puvu}, we see that $\psi_{u|v,u}$ is given by, 
\[
    \sum_{m \in Y, m \ra_{\Hs_u[G]} u} \left( \frac{w(\Phi_{\{u,v\} \ra \{u,m\}}(\Hs_u[G]))}{w(\Phi_{\{u\}}(\Hs_u[G]))}\right)(\ell(m \ra_{\Hs_u[G]} u) - \lambda_{m,Z}) \,.
\]
The departure from manifest positivity now lies solely in the right-hand factor. We see from Definition~\ref{def:hill} that, if $m \ra_G u$, then this factor is simply $\ell(m \ra_G u)$, while, if $m \not\ra_G u$, then this factor is zero. Since $m \ra_G u$ implies that $m \ra_{\Hs_u[G]} u$, we can index the sum over a smaller subset and recover manifest positivity with the formula,
\[
    \psi_{u|v,u} = \sum_{\substack{m \in Y, \, m \ra_G u}} \left( \frac{w(\Phi_{\{u,v\} \ra \{u,m\}}(\Hs_u[G]))}{w(\Phi_{\{u\}}(\Hs_u[G]))} \right) \ell(m \ra_G u) \,.
\]

Aside from these special cases, where $u = v$, $k = u \not= v$ or $k = v \not= u$, the source of the cancellations in Eqn.~\ref{e-excf} that yield manifest positivity are not easy to identify. We leave this as a problem for future work. 

Let us see what Lemma~\ref{lem:exchange} tells us for the exchange factor $\psi_{3|2,1}$, that we previously calculated directly (Eqn.~\ref{eq:psi-ex1}) for the graph $K$ in Fig.~\ref{fig:exchange}A. In this case, $u = 1$, $v = 2$ and $Z = \{4,5\}$, so that $\Zbar = \{1,2,3\}$ and $Y = \Zbar \setminus \{3\} = \{1,2\}$. Hence, Eqn.~\ref{e-excf} tells us that $\psi_{3|2,1}$ is given by 
\[
    1 + \sum_{m \in \{1,2\}}{\left( \frac{w(\Phi_{\{1,3\} \to \{m,3\}}(\Hs_1[K])) - w(\Phi_{\{2,3\} \to \{m,3\}}(\Hs_1[K]))}{w(\Phi_{\{3\}}(\Hs_1[K]))} \right) \lambda_{m,\{4,5\}}} \,.
\]
We can read off from Fig.~\ref{fig:exchange}A that,
\[
    \lambda_{1,\{4,5\}} = 0 \quad \text{and} \quad \lambda_{2,\{4,5\}} = k_3 \,,
\]
and it follows from Fig.~\ref{fig:exchange}B that,
\[
    w(\Phi_{\{3\}}(\Hs_1[K])) = k_1 k_2 + k_1 k_3 \,.
\]
We leave it as an exercise for the reader to also confirm that, 
\begin{align*}
    w(\Phi_{\{1,3\} \to \{1,3\}}(\Hs_1[K]))
    & = k_2 + k_3 \\
    w(\Phi_{\{1,3\} \to \{2,3\}}(\Hs_1[K]))
    & = 0 \\
    w(\Phi_{\{2,3\} \to \{1,3\}}(\Hs_1[K]))
    & = k_3 \\
    w(\Phi_{\{2,3\} \to \{2,3\}}(\Hs_1[K]))
    & = k_1.
\end{align*}
Substituting these expressions into the formula above, we get,
\[
    \psi_{3|2,1} = 1 + \left( \frac{k_2 + k_3 - k_3}{k_1 k_2 + k_1 k_3} \right) \cdot 0 + \left( \frac{0 - k_1}{k_1 k_2 + k_1 k_3} \right) k_3
    = 1 - \frac{k_3}{k_2 + k_3} = \frac{k_2}{k_2 + k_3} \,.
\]
as expected from Eqn.~\ref{eq:psi-ex1}.

\subsection{Generalising Hill's procedure}
\label{sec:hill-general}

With the forest-tree lemma (Lemma \ref{lem:hill}) and the exchange formula (Lemma \ref{lem:exchange}), we now have the resources in hand to give straightforward proofs of Hill-like versions of the main results in our prequel paper \cite{nam:25}, stated previously in Eqns.~\ref{eq:split}, \ref{eq:moment-Z}, and \ref{eq:moment}. As before, let $G$ be a graph with singleton terminal SCCs; let $\emptyset \neq Z \subsetneq \Vs(G)$ be the set of terminal vertices, with $T = \# Z$; and let $u \in \Zbar$ be a non-terminal source vertex. We need to recall some notation introduced in our prequel paper \cite[\S2.1.1]{nam:25}. Given a graph $G$ and vertex $v \in \Vs(G)$, the subset of vertices leading out from $v$, denoted $\Pp(v) \subseteq \Vs(G)$, is given by $\Pp(v) = \{ j \in \Vs(G) : v \tos_G j \}$. Note that $v \in \Pp(v)$.

We begin with a simple deduction that combines Eqn.~\ref{eq:span} with the forest-tree lemma. 

\vspace{0.5em}
\begin{lemma}
With the assumptions above, choose $z \in Z$. Let $W \subseteq \Zbar = \Vs(\Hs_u[G])$ be the vertex subset of $\Hs_u[G]$ given by $W = \Pp(u)$. Then,  
\begin{align*}
    \sum_{k \in W : k \to_G z}{\ell(k \to_G z) \, w(\Phi_{\{k\}}(\Hs_u[G]))}
    & = w(\Phi_{(Z \setminus \{z\}) \cup \{u\} \to Z}(G))\,, \\
    \sum_{z \in Z}{\sum_{k \in W : k \to_G z}{\ell(k \to_G z) \, w(\Phi_{\{k\}}(\Hs_u[G]))}}
    & = w(\Phi_Z(G))\,.
\end{align*}
\label{cor:hill-2}
\end{lemma}
\vspace{-0.8em}

\begin{proof}
Choose $k \in \Zbar = \Vs(\Hs_u[G])$. If $k \not\in W$, then $u \not\tos_{\Hs_u[G]} k$, so that $\Phi_{\{k\}}(\Hs_u[G]) = \emptyset$ and $w(\Phi_{\{k\}}(\Hs_u[G])) = 0$. Hence, the left-hand side of the first equation may be rewritten with $\Zbar$ as the index set in place of $W$,
\[
    \sum_{k \in W : k \to_G z}{\ell(k \to_G z) \, w(\Phi_{\{k\}}(\Hs_u[G]))}
    = \sum_{k \in \Zbar : k \to_G z}{\ell(k \to_G z) \, w(\Phi_{\{k\}}(\Hs_u[G]))} \,.
\]
Lemma \ref{lem:hill} now allows us to exchange spanning trees of $\Hs_u[G]$ for spanning forests of $G$, to get, 
\[
    \sum_{k \in \Zbar : k \to_G z}{\ell(k \to_G z) \, w(\Phi_{Z \cup \{u\} \to Z \cup \{k\}}(G))} \,.
\]
We can then apply Eqn.~\ref{eq:span}, with $B = Z$ and $i = u$, to get, 
\[
    w(\Phi_{(Z \setminus \{z\}) \cup \{u\} \to Z}(G)) \,,
\]
which gives the first equation. 

For the second equation in the statement of Lemma~\ref{cor:hill-2}, summing the right-hand side of the first equation over $z \in Z$ corresponds to forming the set of spanning forests rooted at $Z$, in which $u$ has a path to any terminal vertex, $z \in Z$. Since $u$ has a path to some terminal vertex in each spanning forest rooted at $Z$, this set must be $\Phi_Z(G)$ itself. This yields the second equation.
\end{proof}

We can now easily prove our formalisation of Hill's result for splitting probabilities \cite{hill:88}.

\vspace{0.5em}
\begin{theorem}[Eqn.~\ref{eq:split} revisited]
With the same assumptions as at the start of this section, choose $z \in Z$ and let $W$ be as defined in the statement of Lemma~\ref{cor:hill-2}. The splitting probability from $u$ to $z$ is given by,
\[
    \pi_{u,z}(G) = \frac{\sum_{k \in W : k \to_G z}{\ell(k \to_G z) \, x_k^*(\Hs_u[G])}}{\sum_{z' \in Z}{\sum_{k \in W : k \to_G z'}{\ell(k \to_G z') \, x_k^*(\Hs_u[G])}}} \,.
\]
\label{thm:hill-split}
\end{theorem}
\vspace{-0.8em}

\begin{proof}
If we consider our formula for the splitting probabilities in Eqn.~\ref{eq:split}, we see that we can replace the numerator with the first equation in Lemma~\ref{cor:hill-2} and the denominator with the second equation in Lemma~\ref{cor:hill-2}, to get, 
\begin{equation}
    \pi_{u,z}(G) = \frac{\sum_{k \in W : k \to_G z}{\ell(k \to_G z) \, w(\Phi_{\{k\}}(\Hs_u[G]))}}{\sum_{z' \in Z}{\sum_{k \in W : k \to_G z'}{\ell(k \to_G z') \, w(\Phi_{\{k\}}(\Hs_u[G]))}}} \,.
\label{eq:hill-split-1}
\end{equation}
Lemma \ref{lem:hill-scc} tells us that $\Hs_u[G]$ has a single terminal SCC, so that we can use Eqns.~\ref{eq:rho} and \ref{eq:ss} to determine s.s.~probabilities. Noting that the vertices of $\Hs_u[G]$ are given by $\Zbar$, we see that, 
\[
    x_k^*(\Hs_u[G])
    = \frac{w(\Phi_{\{k\}}(\Hs_u[G]))}{\sum_{j \in \Zbar}{w(\Phi_{\{j\}}(\Hs_u[G]))}} \,.
\]
Dividing above and below in Eqn.~\ref{eq:hill-split-1} by the denominator of this s.s. probability yields the required result.
\end{proof}

Theorem~\ref{thm:hill-split} expresses splitting probabilities in terms of s.s.~probabilities of the Hill construction with source vertex $u$, $\Hs_u[G]$, together with the labels arising from those edges of $G$ that enter $Z$. It is sufficiently similar to Hill's own formula in \cite[Eqn.~7]{hill:88} that it seems reasonable to regard $\Hs_u[G]$ as an appropriate formalisation of Hill's intuitions, subject to the caveat mentioned at the end of \S\ref{sec:hill-def}. 

We now turn to the FPT to any terminal vertex, which Hill considered only for the mean \cite{hill:88}. We use our previous result (Eqn.~\ref{eq:moment-Z}) to derive a Hill-like formula for all moments of this FPT.

\vspace{0.5em}
\begin{theorem}[Eqn.~\ref{eq:moment-Z} revisited] 
With the same assumptions as at the start of this section, let $W$ be as defined in the statement of Lemma~\ref{cor:hill-2}. Let $\Omega \subseteq \Zbar^r$ be the subset of $r$-tuples, $(i_1, \dotsc, i_r)$ with $i_j \in \Zbar$, such that, 
\[
    u = i_0 \tos_G i_1 \tos_G \dotsb \tos_G i_{r-1} \tos_G i_r \,.
\]
Then, the $r$-th moment of the FPT from $u$ to $Z$ is given by,
\begin{equation}
    \mu_{u,Z}^{(r)}(G) = r! \sum_{(i_1, \dotsc, i_r) \in \Omega}{\left( \prod_{j=0}^{r-1}{\frac{\psi_{i_{j+1} | i_j,u} \, x_{i_{j+1}}^*(\Hs_u[G])}{\sum_{z \in Z}{\sum_{k \in W : k \to_G z}{\ell(k \to_G z) \, x_k^*(\Hs_u[G])}}}} \right)} \,.
\label{e-urZ}
\end{equation}
\label{thm:hill-moment-Z}
\end{theorem}
\vspace{-1em}

\begin{proof}
In Eqn.~\ref{eq:moment-Z}, the denominator term, $w(\Phi_Z(G))$, is evidently independent of both the inner product and the outer summation. Let us extricate it, along with the factor $r!$, for now, so as to focus on what remains,
\begin{equation} 
    \sum_{(i_1, \dotsc, i_r) \in \Zbar^r} \left( \prod_{j=0}^{r-1} w(\Phi_{Z \cup \{i_j\} \to Z \cup \{i_{j+1}\}}(G)) \right)\,.
\label{e-rsi1}
\end{equation}
If the $r$-tuple $(i_1, \dotsc, i_r)$ does not satisfy the condition for being in $\Omega$, then $i_p \not\tos_G i_{p+1}$ for some $0 \leq p < r$. But then $\Phi_{Z \cup \{i_j\} \to Z \cup \{i_{j+1}\}}(G) = \emptyset$ and $w(\Phi_{Z \cup \{i_j\} \to Z \cup \{i_{j+1}\}}(G)) = 0$. It is therefore clear that Eqn.~\ref{eq:moment-Z} holds just as well with the outer summation taken over $(i_1, \dotsc, i_r) \in \Omega$, instead of over $(i_1, \dotsc, i_r) \in \Zbar^r$ (see also the comment after this proof). We can now apply the forest-tree lemma (Lemma~\ref{lem:hill}) so that Eqn.~\ref{e-rsi1} becomes, 
\begin{equation}
   \sum_{(i_1, \dotsc, i_r) \in \Omega} \left( \prod_{j=0}^{r-1} w(\Phi_{i_{j+1}}(\Hs_{i_j}[G])) \right) \,.
\label{e-rsi2}
\end{equation}
By assumption and the transitivity of the $\tos$ relation, $u = i_0 \tos_G i_{j+1}$ for $0 \leq j < r$. Therefore, the exchange factor  $\psi_{i_{j+1}|i_j,u}$ is defined (Definition~\ref{lem:exchange-def}), so that we can rewrite Eqn.~\ref{e-rsi2} as, 
\begin{equation}
   \sum_{(i_1, \dotsc, i_r) \in \Omega} \left( \prod_{j=0}^{r-1} \psi_{i_{j+1}|i_j,u} \, w(\Phi_{i_{j+1}}(\Hs_{u}[G])) \right) \,.
\label{e-rsi3}
\end{equation}
By virtue of the exchange formula (Lemma~\ref{lem:exchange}), $\Hs_u[G]$ becomes the only Hill construction that is required for Eqn.~\ref{e-rsi3}, in contrast to Eqn.~\ref{e-rsi2}.

We can now bring back into consideration the factor $r!$ and the denominator term, $w(\Phi_Z(G))$, from Eqn.~\ref{eq:moment-Z}. If we apply the second part of Lemma~\ref{cor:hill-2} to the denominator term, we see that, 
\[
    \mu^{(r)}_{u,Z}(G) = r! \sum_{(i_1, \dotsc, i_r) \in \Omega}{\left( \prod_{j=0}^{r-1}{\frac{\psi_{i_{j+1}|i_j,u} \, w(\Phi_{\{i_{j+1}\}}(\Hs_u[G]))}{\sum_{z \in Z}{\sum_{k \in W : k \to_G z}{\ell(k \to_G z) \, w(\Phi_{\{k\}}(\Hs_u[G]))}}}} \right)} \,.
\]
We can now divide above and below by 
\[
    \left( \sum_{m \in \Zbar}{w(\Phi_{\{m\}}(\Hs_u[G]))} \right)^r \,,
\]
and use Eqns.~\ref{eq:rho} and \ref{eq:ss} to convert weights to s.s.~probabilities, as we did in the proof of Theorem \ref{thm:hill-split}, from which the required result follows.
\end{proof}

The reader may have noticed that we could have introduced $\Omega$ into Eqn.~\ref{eq:moment-Z} and into our previous result in \cite[Theorem 6]{nam:25}, of which Eqn.~\ref{eq:moment-Z} is a restatement. We chose not to do so in \cite{nam:25} because, on balance, it would have complicated the statement of the result without much gain of insight. In contrast, $\Omega$ becomes a necessity for the statement of Theorem~\ref{thm:hill-moment-Z}, for otherwise the exchange factors would not be well-defined.

Let us take $r = 1$ in Theorem~\ref{thm:hill-moment-Z} and extricate once again the denominator term. For the numerator term, 
\[
    \sum_{i_1 \in \Omega} \psi_{i_1|i_0,u} \, x_{i_1}^*(\Hs_u[G])
    = \sum_{i_1 \in \Omega} x_{i_1}^*(\Hs_u[G])
    = \sum_{i_1 \in \Zbar} x_{i_1}^*(\Hs_u[G]) = 1 \,.
\]
The first equality arises because $i_0 = u$, so it follows from Definition~\ref{lem:exchange-def} that $\psi_{i_1|i_0,u} = 1$. The second equality arises because, if $i_1 \in \Zbar \setminus \Omega$, then $u \not\tos_G i_1$, which implies that $u \not\tos_{\Hs_u[G]} i_1$. Hence, $\Hs_u[G]$ cannot have a spanning tree rooted at $i_1$, which implies through Eqns.~\ref{eq:rho} and \ref{eq:ss} that $x_{i_1}^*(\Hs_u[G]) = 0$. We can therefore write the summation over $i_1 \in \Zbar$ instead of over $i_1 \in \Omega$. Since $\Zbar = \Vs(\Hs_u[G])$, the resulting expression then gives the total s.s.~probability, which must be $1$. Accordingly, we are only left with the denominator term, so that
\[
    \mu_{u,Z}^{(1)}(G) = \left( \sum_{z \in Z}{\sum_{k \in W : k \to_G z}{\ell(k \to_G z) \, x_k^*(\Hs_u[G])}} \right)^{-1} \,.
\]
This is, with the introduction of the Hill construction, essentially identical to Hill's formula for the mean FPT to $Z$ \cite[Eqn.~8]{hill:88}. 

We now turn to the moments of the conditional FPT from $u$ to a specific terminal vertex, $z \in Z$. Hill did not consider these quantities but we can follow the examples of Theorems~\ref{thm:hill-split} and \ref{thm:hill-moment-Z} and formulate the Hill-like result that follows. 

\vspace{0.5em}
\begin{theorem}[Eqn.~\ref{eq:moment} revisited]
With the same assumptions as at the start of this section, choose $z \in Z$. Let $W$ be as defined in the statement of Lemma~\ref{cor:hill-2} and let $\Omega$ be as defined in the statement of Theorem~\ref{thm:hill-moment-Z}. Then the $r$-th moment of the conditional FPT from $u$ to $z$ is given by,
\[
\begin{split}
    \mu^{(r)}_{c,u,z}(G) = r! \sum_{(i_1, \dotsc, i_r) \in \Omega}{\left( \prod_{j=0}^{r-1}{\frac{\psi_{i_{j+1}|i_j,u} \, x_{i_{j+1}}^*(\Hs_u[G])}{\sum_{z' \in Z}{\sum_{k \in W , k \to_G z'}{\ell(k \to_G z') \, x_k^*(\Hs_u[G])}}}} \right)} \\
    \times \left( \frac{\sum_{k \in W: k \to_G z}{\ell(k \to_G z) \, \psi_{k|i_r,u} \, x_k^*(\Hs_u[G])}}{\sum_{k \in W: k \to_G z}{\ell(k \to_G z) \, x_k^*(\Hs_u[G])}} \right) \,.\qquad\qquad\quad
\end{split}
\]
\label{thm:hill-moment}
\end{theorem}
\vspace{-1em}

\begin{proof} 
Eqns.~\ref{eq:moment-Z} and \ref{eq:moment} are both sums over the same index set of $r$-tuples, $\Omega \subseteq \Zbar^r$. Each summand in Eqn.~\ref{eq:moment} has two factors, of which the first is identical to the single factor that appears in the sum in Eqn.~\ref{eq:moment-Z}. We can, therefore, undertake the same transformation of this first factor as we did in the proof of Theorem~\ref{thm:hill-moment-Z}, which yields the corresponding factor in the statement of the present theorem. As for the second factor in Eqn.~\ref{eq:moment}, we can apply the first equation in Lemma~\ref{cor:hill-2} to the numerator and denominator to get,
\begin{equation}
    \frac{\sum_{k \in W : k \to_G z}{\ell(k \to_G z) \, w(\Phi_{\{k\}}(\Hs_{i_r}[G]))}}{\sum_{k \in W : k \to_G z}{\ell(k \to_G z) \, w(\Phi_{\{k\}}(\Hs_u[G]))}} \,.
\label{e-kgz}
\end{equation}
Because $k \in W$, we know that $u \tos_{\Hs_u[G]} k$. Hence, by Lemma \ref{lem:hill-paths}, $u \tos_G k$. So we can use Definition~\ref{lem:exchange-def} to replace $\Hs_{i_r}[G]$ by $\Hs_u[G]$ in the numerator of Eqn.~\ref{e-kgz}, to get  
\[
    \sum_{k \in W : k \to_G z}{\ell(k \to_G z) \, \psi_{k|i_r,u} \, w(\Phi_{\{k\}}(\Hs_{u}[G]))} \,.
\]
As previously, we can now replace weights by s.s.~probabilities in the rewritten second factor in Eqn.~\ref{e-kgz}, by dividing above and below by, 
\[
    \sum_{m \in \Zbar}{w(\Phi_{\{m\}}(\Hs_u[G]))} \,,
\]
and appealing to Eqns.~\ref{eq:rho} and \ref{eq:ss}. The second factor then becomes,  
\[
\left( \frac{\sum_{k \in W : k \to_G z}{\ell(k \to_G z) \,\psi_{k|i_r,u} \, x^*_k(\Hs_{u}[G])}}{\sum_{k \in W : k \to_G z}{\ell(k \to_G z) \, x^*_k(\Hs_u[G])}} \right) \,,
\]
as required in the statement of the present theorem. The result follows. 
\end{proof}

\subsection{Generalising Kac's lemma}
\label{sec:kac}

The results proved in the preceding section, particularly Theorem~\ref{thm:hill-moment-Z}, suggest a way to revisit and generalise Kac's lemma \cite{kac:47,serfozo}, which gives the mean recurrence time in terms of s.s.~probabilities. Let $G$ be a strongly connected graph, and let $u \in \Vs(G)$. Recall the unravelling operator that we introduced in \S\ref{sec:unravel}, which yields the unravelled graph, $\Us_u[G]$, with the new vertex $u^+$. (Note the change of symbol from $v$ in \S\ref{sec:unravel} to $u$ here.) In the graph $\Us_u[G]$, $\{u\}$ is the unique initial SCC and $\{u^+\}$  is the unique terminal SCC. Because its single terminal SCC is a singleton, $\Us_u[G]$ satisfies the conditions described in \S\ref{sec:hill-def} for defining FPTs, with the set of terminal vertices given by $Z = \{u^+\}$. As noted in \S\ref{sec:unravel}, the distribution of recurrence times to $u$ in $G$ is identical to the distribution of FPTs from $u$ to $u^+$ in $\Us_u[G]$. It follows that the moments of the recurrence time distribution are given by (Eqn.~\ref{eq:recur-unravel}), 
\[
    \tau_u^{(r)}(G) = \mu_{u,\{u^+\}}^{(r)}(\Us_u[G]) \,.
\]
This provides the starting point for applying Theorem \ref{thm:hill-moment-Z} to $\Us_u[G]$. 

Let us work through the different features of the expression for $\mu_{u,\{u^+\}}^{(r)}(\Us_u[G])$ given in Eqn.~\ref{e-urZ} in the statement of Theorem~\ref{thm:hill-moment-Z}. The graph that appears on the right-hand side of Eqn.~\ref{e-urZ} is now $\Hs_u[\,\Us_u[G]]$. Because Lemma~\ref{t-hug} tells us that the Hill operator is a left-inverse to the unravelling operator, this graph is just $G$ itself,  $\Hs_u[\,\Us_u[G]] = G$. As we have observed previously, the denominator in Eqn.~\ref{e-urZ} can be extricated from both the inner product and the outer sum. This denominator expression is a sum over the index sets $Z$ and $W$, where $Z = \{u^+\} \subseteq \Vs(\Us_u[G])$ and $W$, as defined in Lemma~\ref{cor:hill-2}, is a subset of $\Vs(\Hs_u[\,\Us_u[G]]) = \Vs(G)$. Since $G$ is strongly connected, $W = \Pp(u) = \Vs(G)$. The denominator expression in Eqn.~\ref{e-urZ} is therefore, 
\[
    \sum_{k \in \Vs(G) : k \to_{\Us_u[G]} u^+}{\ell(k \to_{\Us_u[G]} u^+) \, x_k^*(G)} \,.
\]
It follows from Definition~\ref{def:unravel} that $k \to_{\Us_u[G]} u^+$ if, and only if, $k \to_G u$, and that the corresponding edge labels are identical. The expression above may therefore be rewritten as, 
\[
    \sum_{k \in \Vs(G) : k \to_G u}{\ell(k \to_G u) \, x_k^*(G)} \,.
\]
A glance at Eqn.~\ref{eq:lap} shows that this latter expression is just,
\begin{equation}
    \left( \lap(G) \, \xb^*(G) \right)_u + \lambda_u(G) \, x^*_u(G) \,.
\label{e-lgxg}
\end{equation}
Since $\xb^*(G)$ is a steady state, it follows from the Laplacian master equation (Eqn.~\ref{eq:master}) that the first term in Eqn.~\ref{e-lgxg} must be zero, so that,  
\[
    \sum_{k \in \Vs(G) : k \to_G u}{\ell(k \to_G u) \, x_k^*(G)} = \lambda_u(G) \, x^*_u(G) \,.
\]
We can therefore extricate the prefactor,
\[ 
    \frac{1}{\left( \lambda_u(G) \, x^*_u(G) \right)^r} \,,
\]
from the sum of products in Eqn.~\ref{e-urZ}.

Let us now consider the remaining sum, which, by Lemma~\ref{t-hug}, we can write as
\begin{equation}
    \sum_{(i_1, \dotsc, i_r) \in \Omega}\left( \prod_{j=0}^{r-1}{\psi_{i_{j+1}|i_j,u} \, x_{i_{j+1}}^*(G)} \right) \,.
\label{eq:sum-remaining}
\end{equation}
Here, $\Omega$ is a subset of $\Zbar^r$, where $\Zbar \subseteq \Vs(\Us_u[G])$ may be identified as a set with $\Vs(G)$. We now show that the sum in Eqn.~\ref{eq:sum-remaining} may be extended to run over $\Vs(G)^r$. To see this, choose any $(i_1, \dotsc, i_r) \in \Vs(G)^r$ and consider the corresponding product that appears in Eqn.~\ref{eq:sum-remaining}, 
\begin{equation}
    \prod_{j=0}^{r-1}{\psi_{i_{j+1}|i_j,u} \, x_{i_{j+1}}^*(G)} \,.
\label{eq:prod-remaining}
\end{equation}
The exchange factors $\psi_{i_{j+1}|i_j,u}$ in Eqn.~\ref{eq:prod-remaining} arise from the graph $\Us_u[G]$ and are well-defined according to Definition~\ref{lem:exchange-def} because, since $G$ is strongly connected, $u \tos_G i_{j+1}$, and so, by virtue of Definition~\ref{def:unravel}, $u \tos_{\Us_u[G]} i_{j+1}$.  Using Definition \ref{lem:exchange-def} and Lemma \ref{t-hug}, we see that, 
\[
    \psi_{i_{j+1} \mid i_j,u} = \frac{w(\Phi_{\{i_{j+1}\}}(\Hs_{i_j}[\,\Us_u[G]]))}{w(\Phi_{\{i_{j+1}\}}(\Hs_u[\,\Us_u[G]]))}
    = \frac{w(\Phi_{\{i_{j+1}\}}(\Hs_{i_j}[\,\Us_u[G]]))}{w(\Phi_{\{i_{j+1}\}}(G))}\,.
\]
According to Lemma \ref{l-rst0}, $\Phi_{\{i_{j+1}\}}(\Hs_{i_j}[\,\Us_u[G]])$ is non-empty, and so $w(\Phi_{\{i_{j+1}\}}(\Hs_{i_j}[\,\Us_u[G]])) \not= 0$, if, and only if, $i_j \tos_{\Us_u[G]} i_{j+1}$. But $(i_1, \dotsc, i_r) \in \Omega$ precisely when, 
\[ i_1 \tos_{\Us_u[G]} i_2 \tos_{\Us_u[G]} \dotsb \tos_{\Us_u[G]} i_{r-1} \tos_{\Us_u[G]} i_r \,.\]
It follows that the product in Eqn.~\ref{eq:prod-remaining} is nonzero if, and only if, $(i_1, \dotsc, i_r) \in \Omega$. Accordingly, we are at liberty in Eqn.~\ref{eq:sum-remaining} to replace the outer sum indexed over $\Omega$ by the same sum indexed over $\Vs(G)^r$. 

Let us now evaluate the exchange factors themselves. Since $\Zbar = \Vs(G)$ and $\Hs_u[\,\Us_u[G]] = G$, the exchange formula in Eqn.~\ref{e-excf} becomes in this case, 
\[
1 + \sum_{m \in \Vs(G) \setminus \{i_{j+1}\}}\left( \frac{w(\Phi_{\{u,i_{j+1}\} \to \{m,i_{j+1}\}}(G)) - w(\Phi_{\{i_j,i_{j+1}\} \to \{m,i_{j+1}\}}(G))}{w(\Phi_{\{i_{j+1}\}}(G))} \right) \lambda_{m,\{u^+\}} \,.
\]
The corner cases that annoyed us in the derivation of Lemma~\ref{lem:exchange} can rear their heads again here, if $i_j = u$, $i_{j+1} = u$ or $i_j = i_{j+1}$. However, if the reader has exercised due diligence with the exercise that was suggested during that derivation, it will be found that these cases can be absorbed into the exchange formula as written above. The last factor in the exchange formula is given by
\[ \lambda_{m,\{u^+\}} = \begin{cases}
    \ell(m \to_{\Us_u[G]} u^+) & \text{if $m \to_{{\cal U}_u[G]} u^+$} \\
    0 & \text{otherwise} \,.
\end{cases}
\]
But, once again, $m \ra_{\Us_u[G]} u^+$ if, and only if, $m \ra_G u$, and the corresponding labels are identical. We can therefore rewrite the sum in the exchange factor as,
\[
\sum_{\substack{m \in \Vs(G) \setminus \{i_{j+1}\} \\ m \to_G u}}\left( \frac{w(\Phi_{\{u,i_{j+1}\} \to \{m,i_{j+1}\}}(G)) - w(\Phi_{\{i_j,i_{j+1}\} \to \{m,i_{j+1}\}}(G))}{w(\Phi_{\{i_{j+1}\}}(G))} \right) \ell(m \to_G u) \,.
\]
Putting the pieces together leads to the following result.

\vspace{0.5em}
\begin{theorem}[Generalised Kac lemma]
Let $G$ be a strongly connected graph and choose $u \in \Vs(G)$. The $r$-th moment of the recurrence time to $u$ is given by,
\[
    \tau^{(r)}_u(G) = \frac{r!}{\left( \lambda_u(G) \, x_u^*(G) \right)^r} \sum_{(i_1, \dotsc, i_r) \in \Vs(G)^r}{\left( \prod_{j=0}^{r-1}{\psi_{i_{j+1}|i_j,u} \, x_{i_{j+1}}^*(G)} \right)} \,,
\]
where the exchange factor takes the form,
\begin{align*}
    & \psi_{i_{j+1}|i_j,u} = 1 \,\, + \\
    & \sum_{\substack{m \in \Vs(G) \setminus \{i_{j+1}\} \\ m \to_G u}}\left( \frac{w(\Phi_{\{u,i_{j+1}\} \to \{m,i_{j+1}\}}(G)) - w(\Phi_{\{i_j,i_{j+1}\} \to \{m,i_{j+1}\}}(G))}{w(\Phi_{\{i_{j+1}\}}(G))} \right) \ell(m \to_G u) \,.
\end{align*}
\label{thm:kac-moment}
\end{theorem}

The unravelling operator and the Hill operator were essential for the proof of Theorem~\ref{thm:kac-moment} but do not appear in the statement of the theorem, in which all quantities arise solely from the graph $G$. 

If we set $r = 1$ in Theorem~\ref{thm:kac-moment} and recall that $i_0 = u$, the exchange factor in the formula for $\tau^{(1)}_u(G)$ becomes $\psi_{i_1|u,u} = 1$ (Definition~\ref{lem:exchange-def}). The sum therefore reduces to the total s.s. probability, $\sum_{i_1 \in \Vs(G)}{x_{i_1}^*(G)} = 1$. It follows that, 
\begin{equation}
    \tau^{(1)}_u(G) = \frac{1}{\lambda_u(G) \, x_u^*(G)} \,,
\label{eq:kac-mean}
\end{equation}
which is Kac's formula for the mean recurrence time \cite{kac:47}, in the version appropriate for continuous-time Markov processes \cite{serfozo}. 

It is helpful to note one corner case in which the exchange factor, $\psi_{i_{j+1} \mid i_{j}, u}$, in Theorem~\ref{thm:kac-moment} simplifies dramatically. If $i_{j+1} = u \not= i_j$, then,
\begin{equation}
    \psi_{i_{j+1}|i_j,u} = \psi_{u \mid i_j, u} = 0 \,.
\label{eq-uiju}
\end{equation}
To see why, note that $\Phi_{\{u,i_{j+1}\} \to \{m,i_{j+1}\}}(G) = \Phi_{\{u,u\} \to \{m,u\}}(G)$ is improperly defined since $\{u,u\}$ has size $1$ and $\{m,u\}$ has size 2.
Hence, $w(\Phi_{\{u,i_{j+1}\} \to \{m,i_{j+1}\}}(G)) = 0$. The formula for $\psi_{i_{j+1} \mid i_{j}, u}$ then simplifies to,
\[
    \psi_{u|i_j,u} = 1 - \sum_{\substack{m \in \Vs(G) \setminus \{u\} \\ m \to_G u}}\left( \frac{w(\Phi_{\{i_j,u\} \to \{m,u\}}(G))}{w(\Phi_{\{u\}}(G))} \right) \ell(m \to u) \,.
\]
Applying Eqn.~\ref{eq:span} with $B = \{u\}$, $z = u$, and $i = i_j$, for which we need $u \not= i_j$, we see that, 
\[
    \sum_{\substack{m \in \Vs(G) \setminus \{u\} \\ m \to_G u}}{w(\Phi_{\{i_j,u\} \to \{m,u\}}(G)) \, \ell(m \to u)}
    = w(\Phi_{\{i_j\} \to \{u\}}(G)) = w(\Phi_{\{u\}}(G)) \,,
\]
from which it follows that, as claimed,
\[
    \psi_{u \mid i_j, u} = 1 - \frac{w(\Phi_{\{u\}}(G))}{w(\Phi_{\{u\}}(G))} = 0 \,.
\]

One feature of the new results presented in this paper, including Theorem~\ref{thm:kac-moment}, is that the rational algebraic structure of the calculated quantities is not immediately visible. This is not an issue for the quantities calculated in Theorems~\ref{thm:hill-split}, \ref{thm:hill-moment-Z} and \ref{thm:hill-moment} because the original versions of these results from our prequel paper, which are reproduced here, respectively, as Eqns.~\ref{eq:split}, \ref{eq:moment-Z} and \ref{eq:moment}, make it straightforward to work out several aspects of the rational structure of the corresponding quantities. But Theorem~\ref{thm:kac-moment} is new and does not have an original version, so it may be helpful to note some easy deductions from the formulas in Theorem~\ref{thm:kac-moment}. First, it follows from Eqns.~\ref{eq:rho} and \ref{eq:ss} that the common denominator of the rational functions for $x^*_u(G)$ and $x^*_{i_{j+1}}(G)$ will cancel out in the formula for $\tau^{(r)}_u(G)$, leaving the polynomial $w(\Phi_{\{u\}}(G))^r$ in the denominator of the prefactor and the polynomial $w(\Phi_{\{i_{j+1}\}}(G))$ in the numerator of the inner term. The latter can then be cancelled with the denominator of the exchange factor, $\psi_{i_{j+1} \mid i_j,u}$. This leaves a polynomial in the inner term, whose total degree is $N-1$, assuming that $\#\Vs(G) = N$. It follows that,
\begin{equation}
    \tau^{(r)}_u(G) = \left( \frac{r!}{\left( \lambda_u(G) \, w(\Phi_u(G)) \right)^r} \right) P \,,
\label{e-trug}
\end{equation}
where $P$ is a polynomial in the edge labels whose monomials all have total degree $r \left( N-1 \right)$. $\tau^{(r)}_u(G)$ is therefore a rational algebraic function with total degree $r \left( N-1 \right)$ in the numerator and $rN$ in the denominator.

\subsection{An example}
\label{sec:example}

In our prequel paper \cite{nam:25}, we considered examples of \textit{pipeline graphs} \cite[Fig.~4]{nam:25}, one of which was also treated by Hill \cite{hill:88}, and also a \textit{butterfly graph} \cite{nam:25}, for which we calculated various transient quantities---mean unconditional FPTs, mean conditional FPTs, splitting probabilities---using the formulas introduced in our prequel paper. The reader may like to try out the formulas in the present paper on these same examples. To offer something different here, we consider the \textit{cycle graph}, $K$, of five vertices with a cycle of reversible edges, shown in Fig.~\ref{fig:example}A, for which we calculate the first and second moments, $\tau_1^{(1)}(K)$ and $\tau_1^{(2)}(K)$, of the recurrence time to vertex $u = 1$, using the generalised Kac lemma in Theorem~\ref{thm:kac-moment}. The recurrence time formulas for the higher moments are novel results, to the best of our knowledge, so it seems appropriate to see in more detail how they can be calculated on an example. 

Cycle graphs are reversible pipeline graphs whose ends have been brought together. These kinds of Markov processes have been widely used to study single-molecule enzyme kinetics \cite{kou:05, shaevitz:05, moffitt:10a, moffitt:14}. Cycle and pipeline graphs are particularly simple in structure and the combinatorial explosion that arises in enumerating spanning forests can be handled more readily in them \cite{nam:23}. Nevertheless, the algebraic calculations needed for the generalised Kac lemma rapidly become very tedious. This provides an opportunity to examine the Chebotarev--Agaev recurrence \cite{chebotarev:02}, which offers a systematic way to undertake the algebra without enumerating the forests. We have mentioned this approach previously, as one strategy for confronting the algebraic complexity associated with spanning forests \cite{nam:23}, but we have not previously had an opportunity to show it at work, which we do here. 

\begin{figure}
\centering
\includegraphics[trim={1.7in 4.2in 1.7in 3.8in}, clip, width=0.9\textwidth]{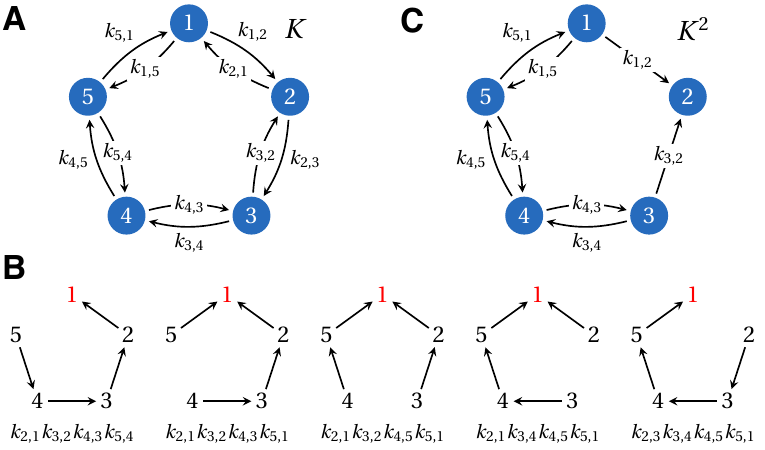}
\caption{(\textbf{A}) A cycle graph, $K$, on five vertices.
(\textbf{B}) The spanning trees of $K$ rooted at $1$ (red) and their weights.
(\textbf{C}) The graph $K^2$ discussed in the text, in which the edges leaving vertex $2$ in $K$ have been removed.
}
\label{fig:example}
\end{figure}

As discussed after Theorem \ref{thm:kac-moment}, the mean recurrence time is given by Kac's lemma (Eqn.~\ref{eq:kac-mean}), which in this case leads to, 
\begin{equation}
    \tau_1^{(1)} = \frac{1}{\left( k_{1,2} + k_{1,5} \right) x_1^*} \,.
\label{e-t1K}
\end{equation}
We will typically omit mention of the graph $K$ in such expressions, for brevity. To calculate the s.s.~probability, $x_1^*$, we can use Eqn.~\ref{eq:ss}, for which we require the spanning trees of $K$. In this case, $K$ is strongly connected and it is easy to see that, for any vertex $i$, there are five spanning trees rooted at $i$, as shown in Fig.~\ref{fig:example}B for $i = 1$. The corresponding weight is then given by, 
\begin{equation}
\begin{aligned}
    w(\Phi_{\{1\}}) & = k_{2,1} k_{3,2} k_{4,3} k_{5,4} + k_{2,1} k_{3,2} k_{4,3} k_{5,1}
    + k_{2,1} k_{3,2} k_{4,5} k_{5,1} \\
    & \qquad + k_{2,1} k_{3,4} k_{4,5} k_{5,1} + k_{2,3} k_{3,4} k_{4,5} k_{5,1} \,.
\end{aligned}
\label{eq:ex-w1}
\end{equation}
It is straightforward to enumerate the remaining spanning trees of $K$, whose weights are given by, 
\begin{equation}
\begin{aligned}
    w(\Phi_{\{2\}}) & = k_{1,5} k_{3,2} k_{4,3} k_{5,4} + k_{1,2} k_{3,2} k_{4,3} k_{5,4}
    + k_{1,2} k_{3,2} k_{4,3} k_{5,1} \\
    & \qquad + k_{1,2} k_{3,2} k_{4,5} k_{5,1} + k_{1,2} k_{3,4} k_{4,5} k_{5,1} \\
    w(\Phi_{\{3\}}) & = k_{1,5} k_{2,1} k_{4,3} k_{5,4} + k_{1,5} k_{2,3} k_{4,3} k_{5,4}
    + k_{1,2} k_{2,3} k_{4,3} k_{5,4} \\
    & \qquad + k_{1,2} k_{2,3} k_{4,3} k_{5,1} + k_{1,2} k_{2,3} k_{4,5} k_{5,1} \\
    w(\Phi_{\{4\}}) & = k_{1,5} k_{2,1} k_{3,2} k_{5,4} + k_{1,5} k_{2,1} k_{3,4} k_{5,4}
    + k_{1,5} k_{2,3} k_{3,4} k_{5,4} \\
    & \qquad + k_{1,2} k_{2,3} k_{3,4} k_{5,4} + k_{1,2} k_{2,3} k_{3,4} k_{5,1} \\
    w(\Phi_{\{5\}}) & = k_{1,5} k_{2,1} k_{3,2} k_{4,3} + k_{1,5} k_{2,1} k_{3,2} k_{4,5}
    + k_{1,5} k_{2,1} k_{3,4} k_{4,5} \\
    & \qquad + k_{1,5} k_{2,3} k_{3,4} k_{4,5} + k_{1,2} k_{2,3} k_{3,4} k_{4,5} \,.
\end{aligned}
\label{eq:ex-trees}
\end{equation}
We can now use Eqns.~\ref{eq:rho} and \ref{eq:ss} to calculate $x_1^*$ as, 
\[
    x_1^* = \frac{w(\Phi_{\{1\}})}{w(\Phi_{\{1\}}) + w(\Phi_{\{2\}}) + w(\Phi_{\{3\}}) + w(\Phi_{\{4\}}) + w(\Phi_{\{5\}})} \,,
\]
which can then be used to calculate $\tau_1^{(1)}$ in Eqn.~\ref{e-t1K}. The complete expression is shown in Appendix~\ref{app:example}. 

For the second moment, Theorem \ref{thm:kac-moment} tells us that,
\[
    \tau_1^{(2)} = \frac{2}{\left( k_{1,2} + k_{1,5} \right)^2 \left( x_1^* \right)^2}
    \sum_{(i_1,i_2) \in \Vs(K)^2}{\left( \psi_{i_1 \mid i_0,1} \, x_{i_1}^* \right) \cdot \left( \psi_{i_2 \mid i_1,1} \, x_{i_2}^*\right)} \,.
\]
This expression depends not only on $x_1^*$, but also on the s.s.~probability of every other vertex in $K$. We can use the spanning tree weights in Eqns.~\ref{eq:ex-w1} and \ref{eq:ex-trees} to calculate these probabilities, as we did above for $x_1^*$, using Eqns.~\ref{eq:rho} and \ref{eq:ss}. What remains to be calculated are the exchange factors, which require more work. Since $i_0 = u = 1$, Definition \ref{lem:exchange-def} tells us that $\psi_{i_1 \mid i_0,1} = \psi_{i_1 \mid 1,1} = 1$ for each $i_1 \in \Vs(K)$. Therefore, we have,
\begin{equation}
    \tau_1^{(2)} = \frac{2}{\left( k_{1,2} + k_{1,5} \right)^2 \left( x_1^* \right)^2}
    \sum_{(i_1,i_2) \in \Vs(K)^2}{\psi_{i_2 \mid i_1,1} \, x_{i_1}^* \, x_{i_2}^*} \,.
\label{eq:ex-tau2}
\end{equation}
Theorem \ref{thm:kac-moment} tells us that the exchange factor, $\psi_{i_2 \mid i_1,1}$ is given by, 
\begin{equation}
    \psi_{i_2|i_1,1} = 1 + \sum_{\substack{m \in \Vs(K) \setminus \{i_2\} \\ m \to_K 1}}\left( \frac{w(\Phi_{\{1,i_2\} \to \{m,i_2\}}) - w(\Phi_{\{i_1,i_2\} \to \{m,i_2\}})}{w(\Phi_{\{i_2\}})} \right) \ell(m \to 1) \,.
\label{eq:ex-psi12}
\end{equation}
If $i_1 = 1$, then it follows from Definition~\ref{lem:exchange-def} that, for any $i_2$, $\psi_{i_2 \mid 1,1} = 1$. If $i_2 = 1$, it follows from Eqn.~\ref{eq-uiju} that, provided $i_1 \not= 1$, then $\psi_{1 \mid i_1, 1} = 0$. 

To calculate the remaining 16 exchange factors, for $i_1, i_2 \neq 1$, we must calculate the spanning forest weights in the numerator of Eqn.~\ref{eq:ex-psi12}. To do so, we use the recursive procedure developed in \cite{chebotarev:02}. Interested readers should note that the ``Laplacian matrix'' defined in \cite{chebotarev:02} is the negative transpose, $\Lb(G) = -\lap(G)^{\T}$, of the Laplacian matrix used here, as in Eqn.~\ref{eq:lap-t}.

Let $G$ be a graph with vertices $\Vs(G) = \left\{ 1, \dotsc, N \right\}$. For $k = 0, 1, 2, \dotsc$, let $\Qb^{(k)}(G)$ denote the $N \times N$ matrix whose $(i,j)$-th entry is the total weight of those spanning forests, $F$, of $G$ with exactly $k$ edges, in which $j$ is a root and $i \leadsto_F j$. More formally, recalling that a forest with $k$ edges on $N$ vertices must have $N-k$ roots, if $i \not= j$, then, 
\begin{equation}
    \Qb_{i,j}^{(k)}(G) = \sum_{\substack{A \subseteq \Vs(G),\, \#A = N-k \\ j \in A, \, i \not\in A}} w(\Phi_{(A \setminus \{j\}) \cup \{i\} \to A} (G)) \,,
\label{eq:qij}
\end{equation}
whereas, if $i = j$, then every forest $F$ with $N-k$ roots, of which $j$ is one, trivially satisfies the requirement that $i \leadsto_F j$, so that,
\begin{equation}
    \Qb_{j,j}^{(k)}(G) = \sum_{\substack{A \subseteq \Vs(G),\, \#A = N-k \\ j \in A}} w(\Phi_{A} (G)) \,.
\label{eq:qij2}
\end{equation}
The $\Qb^{(k)}$ notation used here is slightly different from that in \cite{chebotarev:02} but the relationship should be clear. 

To get some intuition for the $\Qb^{(k)}$ matrices, let us consider some particular values of $k$. For any graph $G$, there is a maximum, $d$, to the number of edges in any spanning forest of $G$. If $G$ is strongly connected, then $d = N-1$, with the maximum number of edges being found in any spanning tree of $G$. Since there can be no forests with more edges than $d$, it follows that $\Qb^{(k)}(G) = \mathbf{0}$ when $k > d$. At the other extreme, if $k = 0$, there is only one forest, $F$, in which all vertices are roots and there are no edges. In this case, $i \leadsto_F j$ if, and only if, $i = j$. The convention for empty products tells us that the weight of this forest is $1$. Hence, 
\begin{equation}
    \Qb^{(0)}(G) = \Ib \,.
\label{eq:recur-init}
\end{equation}
Now consider the case $k = 1$, which is more informative. The forests with $N - 1$ roots are in one-to-one correspondence with the edges, $i \to j$, of $G$, with $i$ being the only non-root in the forest. Hence, if $i \not= j$, it follows from Eqn.~\ref{eq:lap} that,  
\begin{equation}
    \Qb_{i,j}^{(1)}(G) = 
    \begin{cases}
        \ell(i \to j) = \lap(G)_{j,i} & \text{if $i \to_G j$} \\
        0 = \lap(G)_{j,i} & \text{if $i \not\to_G j$} \,, \\
    \end{cases}
\label{e-qbij}
\end{equation}
so that, off the diagonal, $\Qb^{(1)}(G) = \lap(G)^{\T}$. On the diagonal, when $i = j$, any forest with $N - 1$ roots in which $j$ is a root contributes to the sum in Eqn.~\ref{eq:qij2}. The only forests that do not contribute are those in which $j$ is the source vertex for some edge, $j \to_G k$, and is therefore a non-root. The weight of these excluded forests is $\lambda_j = -\lap(G)_{j,j}$, using the notation introduced in Eqn.~\ref{eq:lambda}. Meanwhile, the weight of all the forests with $N - 1$ roots, or, equivalently, the sum of all the edge labels in $G$, is just $\lambda_1 + \cdots + \lambda_N = -\tr{\lap(G)}$. Putting this together with Eqn.~\ref{e-qbij}, we see that,
\begin{equation} 
\Qb^{(1)}(G) = \lap(G)^{\T} - (\tr{\lap(G)^{\T}}) \, \Ib \,.
\label{e-qb1}
\end{equation}

With these preliminary calculations in hand, it should not be too surprising that a recurrence emerges for the $\Qb^{(k)}$ matrices \cite[Proposition 4]{chebotarev:02},
\begin{equation}
    \Qb^{(k+1)}(G) = \lap(G)^{\T} \, \Qb^{(k)}(G) - \left( \frac{\tr{\left( \lap(G)^{\T} \, \Qb^{(k)}(G) \right)}}{k + 1} \right) \Ib \,,
\label{eq:recur}
\end{equation}
whose logic should be reasonably clear. Using Eqn.~\ref{eq:recur} with Eqn.~\ref{eq:recur-init} as the initial condition yields Eqn.~\ref{e-qb1} after one recursion. 

The $\Qb^{(k)}(G)$ matrices are shown in \cite{chebotarev:02} to determine several properties of the Laplacian matrix, $\lap(G)$, such as its Moore--Penrose inverse. For our purposes here, the salient point is that the entries of $\Qb^{(k)}(G)$ give the weights of spanning forests of $G$ without enumerating the forests, which rapidly becomes intractable. 

Let us now return to our formula for the exchange factor $\psi_{i_2 \mid i_1, 1}$ of $K$ (Eqn.~\ref{eq:ex-psi12}), for which we need the weights of spanning forests with two roots, of which one root is $i_2$. We can impose the condition that $i_2$ is a root by using the following trick. For a graph $G$, with $v \in \Vs(G)$, let $G^v$ be the graph obtained from $G$ by removing all the edges outgoing from $v$. Any spanning forest of $G^v$ must have $v$ as a root and if $A, B \subseteq \Vs(G)$ with $\# A = \# B$ and $v \in A$, then, 
\[
    \Phi_{B \to A}(G) = \Phi_{B \to A}(G^v) \,.
\]
It follows that we can calculate the spanning forest weights of $K$ by calculating the $\Qb^{(k)}$ matrices for $K^{i_2}$, instead of for $K$. Since $K$ has five vertices, the required forests have three edges and so their weights will need to be extracted from the entries of $\Qb^{(3)}(K^{i_2})$. In particular, the weights needed in Eqn.~\ref{eq:ex-psi12}, with $m, u \in \Vs(K) \setminus \{i_2\}$, are given by, 
\[
    w(\Phi_{\{u,i_2\} \to \{m,i_2\}}) = \Qb_{u,m}^{(3)}(K^{i_2}) \,.
\]
Determining $\Qb_{1,m}^{(3)}(K^{i_2})$ and $\Qb_{i_1,m}^{(3)}(K^{i_2})$, for $i_1, i_2 \in \left\{ 2, \dotsc, 5 \right\}$ and $m \in \Vs(K) \setminus \{i_2\}$, will therefore give us the quantities needed for the exchange factors $\psi_{i_2 \mid i_1, 1}$ in Eqn.~\ref{eq:ex-psi12}.

To undertake this calculation, we first require the Laplacian matrices of $K^{i_2}$ for $i_2 = 2, \dotsc, 5$, which can be easily obtained from the Laplacian matrix of $K$, 
\[
    \Ls(K) = \left[ \begin{array}{ccccc}
        -k_{1,2} - k_{1,5} & k_{2,1} & 0 & 0 & k_{5,1} \\
        k_{1,2} & -k_{2,1} - k_{2,3} & k_{3,2} & 0 & 0 \\
        0 & k_{2,3} & -k_{3,2} - k_{3,4} & k_{4,3} & 0 \\
        0 & 0 & k_{3,4} & -k_{4,3} - k_{4,5} & k_{5,4} \\
        k_{1,5} & 0 & 0 & k_{4,5} & -k_{5,1} - k_{5,4}
    \end{array} \right] \,,
\]
by zeroing out column $i_2$. We can now recursively apply Eqn.~\ref{eq:recur}, starting from $\Qb^{(0)}(K^{i_2}) = \Ib$. Although the calculations can be done by hand in this case, it is easier, and less error-prone, to implement it with a symbolic mathematics package, like SymPy, as we did for the results shown below. We outline the calculation for $i_2 = 2$, for which the graph $K^2$ is shown in Fig.~\ref{fig:example}C. The transpose of the Laplacian matrix, as needed for Eqn.~\ref{eq:recur}, is then given by,
\begin{equation}
    \Ls(K^2)^{\T} = \begin{bmatrix}
        -k_{1,2} - k_{1,5} & k_{1,2} & 0 & 0 & k_{1,5} \\
        0 & 0 & 0 & 0 & 0 \\
        0 & k_{3,2} & -k_{3,2} - k_{3,4} & k_{3,4} & 0 \\
        0 & 0 & k_{4,3} & -k_{4,3} - k_{4,5} & k_{4,5} \\
        k_{5,1} & 0 & 0 & k_{5,4} & -k_{5,1} - k_{5,4} \\
    \end{bmatrix} \,.
\label{e-lsk2}
\end{equation}
Applying Eqn.~\ref{eq:recur} once, or Eqn.~\ref{e-qb1} directly, gives $\Qb^{(1)}(K^2)$, 
\[ \begin{bmatrix}
        -k_{1,2} - k_{1,5} + U & k_{1,2} & 0 & 0 & k_{1,5} \\
        0 & U & 0 & 0 & 0 \\
        0 & k_{3,2} & -k_{3,2} - k_{3,4} + U & k_{3,4} & 0 \\
        0 & 0 & k_{4,3} & -k_{4,3} - k_{4,5} + U & k_{4,5} \\
        k_{5,1} & 0 & 0 & k_{5,4} & -k_{5,1} - k_{5,4} + U 
\end{bmatrix} \,, \]
where $U = -\tr{\Ls(K^2)^{\T}}$ is the sum of all edge labels in $K^2$,
\begin{align*}
    U & = \lambda_1 + \lambda_3 + \lambda_4 + \lambda_5 \\
    & = k_{1,2} + k_{1,5} + k_{3,2} + k_{3,4} + k_{4,3} + k_{4,5} + k_{5,1} + k_{5,4} \,.
\end{align*}
Applying Eqn.~\ref{eq:recur} again yields a more complicated expression for $\Qb^{(2)}(K^2)$, 
\begin{equation}
    \begin{bmatrix}
         V_1 & k_{1,2} \left( \lambda_3 + \lambda_4 + \lambda_5 \right) & 0 & k_{1,5} k_{5,4} & k_{1,5} \left( \lambda_3 + \lambda_4 \right) \\
         0 & V_2 & 0 & 0 & 0 \\
         0 & k_{3,2} \left( \lambda_1 + \lambda_4 + \lambda_5 \right) & V_3 & k_{3,4} \left( \lambda_1 + \lambda_5 \right) & k_{3,4} k_{4,5} \\
         k_{4,5} k_{5,1} & k_{3,2} k_{4,3} & k_{4,3} \left( \lambda_1 + \lambda_5 \right) & V_4 & k_{4,5} \left( \lambda_1 + \lambda_3 \right) \\
         k_{5,1} \left( \lambda_3 + \lambda_4 \right) & k_{1,2} k_{5,1} & k_{4,3} k_{5,4} & k_{5,4} \left( \lambda_1 + \lambda_3 \right) & V_5
    \end{bmatrix} 
\label{e-v1k}
\end{equation}
where the diagonal entries are given by, 
\begin{align*}
    V_1 & = k_{3,2} k_{4,3} + k_{3,2} k_{4,5} + k_{3,2} k_{5,1} + k_{3,2} k_{5,4} + k_{3,4} k_{4,5} + k_{3,4} k_{5,1} + k_{3,4} k_{5,4} + k_{4,3} k_{5,1} \\
    & \qquad + k_{4,3} k_{5,4} + k_{4,5} k_{5,1} \\
    V_2 & = k_{1,2} k_{3,2} + k_{1,2} k_{3,4} + k_{1,2} k_{4,3} + k_{1,2} k_{4,5} + k_{1,2} k_{5,1} + k_{1,2} k_{5,4} + k_{1,5} k_{3,2} + k_{1,5} k_{3,4} \\
    & \qquad + k_{1,5} k_{4,3} + k_{1,5} k_{4,5} + k_{1,5} k_{5,4} + k_{3,2} k_{4,3} + k_{3,2} k_{4,5} + k_{3,2} k_{5,1} + k_{3,2} k_{5,4}  \\
    & \qquad + k_{3,4} k_{4,5} + k_{3,4} k_{5,1} + k_{3,4} k_{5,4} + k_{4,3} k_{5,1} + k_{4,3} k_{5,4} + k_{4,5} k_{5,1} \\
    V_3 & = k_{1,2} k_{4,3} + k_{1,2} k_{4,5} + k_{1,2} k_{5,1} + k_{1,2} k_{5,4} + k_{1,5} k_{4,3} + k_{1,5} k_{4,5} + k_{1,5} k_{5,4} + k_{4,3} k_{5,1} \\
    & \qquad + k_{4,3} k_{5,4} + k_{4,5} k_{5,1} \\
    V_4 & = k_{1,2} k_{3,2} + k_{1,2} k_{3,4} + k_{1,2} k_{5,1} + k_{1,2} k_{5,4} + k_{1,5} k_{3,2} + k_{1,5} k_{3,4} + k_{1,5} k_{5,4} + k_{3,2} k_{5,1} \\
    & \qquad + k_{3,2} k_{5,4} + k_{3,4} k_{5,1} + k_{3,4} k_{5,4} \\
    V_5 & = k_{1,2} k_{3,2} + k_{1,2} k_{3,4} + k_{1,2} k_{4,3} + k_{1,2} k_{4,5} + k_{1,5} k_{3,2} + k_{1,5} k_{3,4} + k_{1,5} k_{4,3} + k_{1,5} k_{4,5} \\
    & \qquad + k_{3,2} k_{4,3} + k_{3,2} k_{4,5} + k_{3,4} k_{4,5}.
\end{align*}
It is instructive to check one of the entries in Eqn.~\ref{e-v1k} against Eqns.~\ref{eq:qij} and \ref{eq:qij2}. According to Eqn.~\ref{eq:qij}, the $(1,5)$-th entry is the weight of all triply-rooted spanning forests of $K^2$ in which $5$ is a root, $1$ is not a root, and $1 \leadsto 5$. Since $2$ is always a root in a spanning forest of $K^2$, the possible sets of roots are $\{ 2, 3, 5 \}$ and $\{ 2, 4, 5 \}$. A glance at Fig.~\ref{fig:example}C shows that, in any spanning forest rooted at $\{2, 3, 5\}$ in which $1 \tos 5$, the edge $1 \to 5$ must always be present, and there is then a choice between including either $4 \to 3$ or $4 \to 5$. Since these are the outgoing edges from vertex $4$, the contributions of these forests to the weight is $k_{1,5} \lambda_4$. Similarly, in any forest rooted at $\{2, 4, 5\}$ in which $1 \tos 5$, the edge $1 \to 5$ must always be present, and the remaining edge must be chosen from those outgoing from vertex $3$. The contribution of these forests to the weight is then $k_{1,5} \lambda_3$. Putting these two contributions together gives the $(1,5)$-th entry in Eqn.~\ref{e-v1k}.

Applying Eqn.~\ref{eq:recur} again yields a yet more complicated expression for $\Qb^{(3)}(K^2)$. However, we do not need the full matrix but just those entries required to calculate $\psi_{2 \mid i_1,1}$ in Eqn.~\ref{eq:ex-psi12}. For this, the only relevant value of $m$ in Eqn.~\ref{eq:ex-psi12} is $m = 5$, for which we need $\Qb^{(3)}_{1,5}(K^2)$ and $\Qb^{(3)}_{i_1,5}(K^2)$, for $i_1 = 2, \dotsc, 5$. That is, we need the last column of $\Qb^{(3)}(K^2)$. We can work this out by hand from Eqns.~\ref{e-lsk2} and \ref{e-v1k} using Eqn.~\ref{eq:recur} and, after some simplification, we find that, 
\begin{align*}
    \Qb^{(3)}_{1,5}(K^2)
    & = k_{1,5} k_{3,2} k_{4,3} + k_{1,5} k_{3,2} k_{4,5} + k_{1,5} k_{3,4} k_{4,5} \\
    \Qb^{(3)}_{2,5}(K^2)
    & = 0 \\
    \Qb^{(3)}_{3,5}(K^2)
    & = k_{1,2} k_{3,4} k_{4,5} + k_{1,5} k_{3,4} k_{4,5} \\
    \Qb^{(3)}_{4,5}(K^2) 
    & = k_{1,2} k_{3,2} k_{4,5} + k_{1,2} k_{3,4} k_{4,5} + k_{1,5} k_{3,2} k_{4,5} + k_{1,5} k_{3,4} k_{4,5} \\
    \Qb^{(3)}_{5,5}(K^2)
    & = k_{1,2} k_{3,2} k_{4,3} + k_{1,2} k_{3,2} k_{4,5} + k_{1,2} k_{3,4} k_{4,5} + k_{1,5} k_{3,2} k_{4,3} + k_{1,5} k_{3,2} k_{4,5} \\
    & \qquad + k_{1,5} k_{3,4} k_{4,5} \,.
\end{align*}
It is again instructive to check one of these entries against Eqns.~\ref{eq:qij} and \ref{eq:qij2}. For instance, $\Qb^{(3)}_{3,5}(K^2)$ should be the weight of all doubly-rooted spanning forests of $K^2$ in which $5$ is a root, $3$ is not a root, and $3 \leadsto 5$. Since $2$ is always a root in any spanning forest of $K^2$, the only possible set of roots for such a forest is $\{ 2, 5 \}$, and the only way to have $3 \leadsto 5$ is if the forest includes both the edges $3 \to 4$ and $4 \to 5$. The remaining edge must then be either one of the two outgoing edges from $1$, so that the weight is $k_{3,4}k_{4,5}\lambda_1$, as shown above. 

We can now calculate the exchange factors using Eqn.~\ref{eq:ex-psi12}. For $i_2 = 2$, we can substitute the expressions obtained above to get,
\begin{align*}
    \psi_{2 \mid 2,1} & = 1 + \left( \frac{k_{1,5} k_{3,2} k_{4,3} + k_{1,5} k_{3,2} k_{4,5} + k_{1,5} k_{3,4} k_{4,5}}{w(\Phi_{\{2\}})} \right) k_{5,1} \\
    \psi_{2 \mid 3,1} & = 1 + \left( \frac{k_{1,5} k_{3,2} k_{4,3} + k_{1,5} k_{3,2} k_{4,5} - k_{1,2} k_{3,4} k_{4,5}}{w(\Phi_{\{2\}})} \right) k_{5,1} \\
    \psi_{2 \mid 4,1} & = 1 + \left( \frac{k_{1,5} k_{3,2} k_{4,3} - k_{1,2} k_{3,2} k_{4,5} - k_{1,2} k_{3,4} k_{4,5}}{w(\Phi_{\{2\}})} \right) k_{5,1} \\
    \psi_{2 \mid 5,1} & = 1 - \left( \frac{k_{1,2} k_{3,2} k_{4,3} + k_{1,2} k_{3,2} k_{4,5} + k_{1,2} k_{3,4} k_{4,5}}{w(\Phi_{\{2\}})} \right) k_{5,1} \,.
\end{align*}
For the remaining 12 exchange factors for $i_2 = 3, 4, 5$, we can follow the same procedure described above, to get, 
\begin{align*}
    \psi_{3 \mid 2,1} & = 1 + \frac{k_{5,1} \left( k_{1,5} k_{2,3} k_{4,3} + k_{1,5} k_{2,3} k_{4,5} \right) - k_{1,5} k_{2,1} k_{4,3} k_{5,4}}{w(\Phi_{\{3\}})} \\
    \psi_{3 \mid 3,1} & = 1 + \frac{k_{2,1} \left( k_{1,2} k_{4,3} k_{5,1} + k_{1,2} k_{4,3} k_{5,4} + k_{1,2} k_{4,5} k_{5,1} \right)}{w(\Phi_{\{3\}})} \\
    & \qquad + \frac{k_{5,1} \left( k_{1,5} k_{2,1} k_{4,3} + k_{1,5} k_{2,1} k_{4,5} + k_{1,5} k_{2,3} k_{4,3} + k_{1,5} k_{2,3} k_{4,5} \right)}{w(\Phi_{\{3\}})} \\
    \psi_{3 \mid 4,1} & = 1 + \frac{k_{2,1} \left( k_{1,2} k_{4,3} k_{5,1} + k_{1,2} k_{4,3} k_{5,4} \right)}{w(\Phi_{\{3\}})} \\
    & \qquad + \frac{k_{5,1} \left(k_{1,5} k_{2,1} k_{4,3} + k_{1,5} k_{2,3} k_{4,3} - k_{1,2} k_{2,3} k_{4,5} \right)}{w(\Phi_{\{3\}})} \\
    \psi_{3 \mid 5,1} & = 1 + \frac{k_{1,2} k_{2,1} k_{4,3} k_{5,4} - k_{5,1} \left( k_{1,2} k_{2,3} k_{4,3} + k_{1,2} k_{2,3} k_{4,5} \right)}{w(\Phi_{\{3\}})} \\
    \psi_{4 \mid 2,1} & = 1 + \frac{k_{1,5} k_{2,3} k_{3,4} k_{5,1} - k_{2,1} \left( k_{1,5} k_{3,2} k_{5,4} + k_{1,5} k_{3,4} k_{5,4} \right)}{w(\Phi_{\{4\}})} \\
    \psi_{4 \mid 3,1} & = 1 + \frac{k_{2,1} \left( k_{1,2} k_{3,4} k_{5,1} + k_{1,2} k_{3,4} k_{5,4} - k_{1,5} k_{3,2} k_{5,4} \right)}{w(\Phi_{\{4\}})} \\
    & \qquad + \frac{k_{5,1} \left( k_{1,5} k_{2,1} k_{3,4} + k_{1,5} k_{2,3} k_{3,4} \right)}{w(\Phi_{\{4\}})} \\
    \psi_{4 \mid 4,1} & = 1 + \frac{k_{2,1} \left( k_{1,2} k_{3,2} k_{5,1} + k_{1,2} k_{3,2} k_{5,4} + k_{1,2} k_{3,4} k_{5,1} + k_{1,2} k_{3,4} k_{5,4} \right)}{w(\Phi_{\{4\}})} \\
    & \qquad + \frac{k_{5,1} \left( k_{1,5} k_{2,1} k_{3,2} + k_{1,5} k_{2,1} k_{3,4} + k_{1,5} k_{2,3} k_{3,4} \right)}{w(\Phi_{\{4\}})} \\
    \psi_{4 \mid 5,1} & = 1 + \frac{k_{2,1} \left( k_{1,2} k_{3,2} k_{5,4} + k_{1,2} k_{3,4} k_{5,4} \right) - k_{1,2} k_{2,3} k_{3,4} k_{5,1}}{w(\Phi_{\{4\}})} \\
    \psi_{5 \mid 2,1} & = 1 - \frac{k_{2,1} \left( k_{1,5} k_{3,2} k_{4,3} + k_{1,5} k_{3,2} k_{4,5} + k_{1,5} k_{3,4} k_{4,5} \right)}{w(\Phi_{\{5\}})} \\
    \psi_{5 \mid 3,1} & = 1 + \frac{k_{2,1} \left( k_{1,2} k_{3,4} k_{4,5} - k_{1,5} k_{3,2} k_{4,3} - k_{1,5} k_{3,2} k_{4,5} \right)}{w(\Phi_{\{5\}})} \\
    \psi_{5 \mid 4,1} & = 1 + \frac{k_{2,1} \left( k_{1,2} k_{3,2} k_{4,5} + k_{1,2} k_{3,4} k_{4,5} - k_{1,5} k_{3,2} k_{4,3} \right)}{w(\Phi_{\{5\}})} \\
    \psi_{5 \mid 5,1} & = 1 + \frac{k_{2,1} \left( k_{1,2} k_{3,2} k_{4,3} + k_{1,2} k_{3,2} k_{4,5} + k_{1,2} k_{3,4} k_{4,5} \right)}{w(\Phi_{\{5\}})} \,.
\end{align*}
We can now determine $\tau_1^{(2)}$ by using Eqn.~\ref{eq:ex-tau2} and bringing together all the calculations above. The complete rational algebraic function is shown in Appendix \ref{app:example}.

One way to cross-check our calculations of $\tau_1^{(1)}$ and $\tau_1^{(2)}$ is to compare them against the formulas we derived in our prequel paper \cite{nam:25}. Specifically, we can apply Eqn.~\ref{eq:moment-Z} to the unravelled graph, $\Us_1[K]$, as specified in Eqn.~\ref{eq:recur-unravel}. Using this approach, $\tau_1^{(1)}$ can be written in terms of the spanning trees and forests of $\Us_1[K]$, 
\begin{equation}
\begin{aligned}
    \tau_1^{(1)} & = \mu_{1,\{1^+\}}^{(1)}(\Us_1[K]) \\
    & = \sum_{i_1=1}^5{\frac{w(\Phi_{\{1,1^+\} \to \{i_1,1^+\}}(\Us_1[K]))}{w(\Phi_{\{1^+\}}(\Us_1[K]))}} \,,
\end{aligned}
\label{eq:ex-tau1-alt}
\end{equation}
where $1^+$ is the new vertex in $\Us_1[K]$. Similarly, $\tau_1^{(2)}$ can be written as, 
\begin{equation}
\begin{aligned}
    \tau_1^{(2)} & = \mu_{1,\{1^+\}}^{(2)}(\Us_1[K]) \\
    & = 2 \sum_{i_1=1}^5{\sum_{i_2=1}^5{\frac{w(\Phi_{\{1,1^+\} \to \{i_1,1^+\}}(\Us_1[K])) \, w(\Phi_{\{i_1,1^+\} \to \{i_2,1^+\}}(\Us_1[K]))}{w(\Phi_{\{1^+\}}(\Us_1[K]))^2}}} \,.
\end{aligned}
\label{eq:ex-tau2-alt}
\end{equation}
We can undertake these calculations as we did above using the Chebotarev--Agaev recurrence in Eqn.~\ref{eq:recur} with $G = \Us_1[K]$. Identifying the new vertex, $1^+$, with $6$ in the vertex ordering, the corresponding Laplacian matrix is given by 
\[
    \Ls(\Us_1[K]) = \left[ \begin{array}{cccccc}
        -k_{1,2} - k_{1,5} & 0 & 0 & 0 & 0 & 0 \\
        k_{1,2} & -k_{2,1} - k_{2,3} & k_{3,2} & 0 & 0 & 0 \\
        0 & k_{2,3} & -k_{3,2} - k_{3,4} & k_{4,3} & 0 & 0 \\
        0 & 0 & k_{3,4} & -k_{4,3} - k_{4,5} & k_{5,4} & 0 \\
        k_{1,5} & 0 & 0 & k_{4,5} & -k_{5,1} - k_{5,4} & 0 \\
        0 & k_{2,1} & 0 & 0 & k_{5,1} & 0
    \end{array} \right]\,.
\]
It is easy to see that the required spanning forests in Eqns.~\ref{eq:ex-tau1-alt} and \ref{eq:ex-tau2-alt} are given by the following matrix entries, 
\begin{align*}
    w(\Phi_{\{1^+\}}(\Us_1[K])) & = \Qb_{6,6}^{(5)}(\Us_1[K]) \\
    w(\Phi_{\{i,1^+\} \to \{j,1^+\}}(\Us_1[K])) & = \Qb_{i,j}^{(4)}(\Us_1[K]) \,.
\end{align*}
We do not provide the resulting expressions here but have verified using a SymPy implementation of Eqn.~\ref{eq:recur} that they are indeed equal to the formulas for $\tau_1^{(1)}$ and $\tau_1^{(2)}$ shown in Appendix \ref{app:example}.

\section{Discussion}
\label{sec:discussion}

In the prequel to the present paper \cite{nam:25}, we showed how transient properties of a Markov process, such as splitting probabilities and the moments of conditional and unconditional FPTs, could be expressed as manifestly integrally positive rational algebraic functions of the transition rates, using the All-Minors Matrix-Tree theorem and spanning forests of the underlying linear framework graph. In the present paper, we have sought to integrate this work with two other approaches to determining transient properties, due to Terrell Hill \cite{hill:88} and to Mark Kac \cite{kac:47}. Hill sketched a modification to a Markov process through which the splitting probabilities and the mean FPT could be calculated in terms of s.s.~probabilities in the modified process. Similarly, Kac showed how the mean recurrence times of a Markov process could be calculated in terms of s.s.~probabilities. Hill's argument amounted to the calculation of an example, and both Hill and Kac considered only the first moment. 

We have introduced three concepts to integrate their work with ours. First, we have formalised Hill's method in terms of a \textit{Hill operator}, $\Hs_u[G]$ on a graph, $G$, where $u$ is the source vertex for the FPTs (Definition~\ref{def:hill}). Second, we have defined the \textit{exchange factors}, $\psi_{k \mid v, u}$ (Definition~\ref{lem:exchange-def}), which allow us to calculate spanning tree weights in one Hill construction, $\Hs_v[G]$, in terms of spanning tree weights in another Hill construction, $\Hs_u[G]$, using the \textit{exchange formula} (Lemma~\ref{lem:exchange}). Third, to generalise Kac's result, we have introduced the \textit{unravelling operator}, $\Us_v[G]$, on a graph $G$, where $v$ is the vertex whose recurrence time is to be calculated (Definition~\ref{def:unravel}). These three concepts have allowed us to give a rigorous proof of Hill's claim about splitting probabilities, $\pi_{u,z}(G)$ (Theorem~\ref{thm:hill-split}); to prove Hill-like results for the $r$-th moment of the unconditional FPT to any target state, $\mu^{(r)}_{u,Z}(G)$ (Theorem~\ref{thm:hill-moment-Z}), and the $r$-th moment of the conditional FPT to a given target state, $\mu^{(r)}_{c,u,z}(G)$ (Theorem~\ref{thm:hill-moment}); and, finally, to prove Kac-like results for the $r$-th moment of the recurrence time, $\tau^{(r)}_u(G)$ (Theorem~\ref{thm:kac-moment}). 

Theorems~\ref{thm:hill-split}, \ref{thm:hill-moment-Z} and \ref{thm:hill-moment} are all stated in terms of $\Hs_u[G]$, and appropriate exchange factors, so that the Hill operator is central to generalising Hill's idea to a broader range of transient quantities and, especially, to higher moments of FPTs. Theorem~\ref{thm:kac-moment} is stated solely in terms of $G$ but the proof relies on exploiting Theorem~\ref{thm:hill-moment-Z} and using the remarkable algebraic relationship between our two operators (Lemma~\ref{t-hug}), which we may write informally as,
\begin{equation} 
    \Hs_u[\,\Us_u] = \Ib \,,
\label{e-hsuu}
\end{equation}
to ``disappear'' both operators from the statement of the theorem.

Our results show that, as soon as one ventures beyond the first moment, it is no longer possible to rely solely on s.s.~probabilities, and therefore solely on spanning trees, of the linear framework graph. In particular, given a source vertex $u$, our formulas for the higher moments of FPTs from $u$ show that, in addition to the s.s.~probabilities of $\Hs_u[G]$, which require the spanning trees of $\Hs_u[G]$, we also need the exchange factors, $\psi_{k \mid v,u}$, which require both spanning trees and spanning forests of $\Hs_u[G]$. Similarly, in our formulas for the higher moments of the recurrence time, it is the exchange factors that require the spanning forests, as well as spanning trees, of $G$, while the s.s.~probabilities of $G$ require only the spanning trees. The exchange factors thereby appear to compartmentalise the requirement for spanning forests in these calculations. 

The separation between forests and trees comes with a difference, which is that the exchange formula (Lemma~\ref{lem:exchange}) for the exchange factors is no longer manifestly integrally positive. Definition~\ref{lem:exchange-def} shows that the exchange factors do have a representation as manifestly integrally positive rational algebraic functions, which means that the exchange formula must conceal cancellations. We have shown how these come about for some simple corner cases but the general problem of uncovering the cancellations does not seem straightforward and we have had to leave it open for future work. 

The results of the present paper, when compared against those of our prequel paper, give two rather different formulas for splitting probabilities (Theorem~\ref{thm:hill-split}) and moments of FPTs of Markov processes (Theorems~\ref{thm:hill-moment-Z} and \ref{thm:hill-moment}), together with new formulas for recurrence times (Theorem~\ref{thm:kac-moment}). We have just noted one distinction between the old and the new: the new formulas no longer show manifest integral positivity. The old formulas also reveal more clearly the rational algebraic structure. We can see, for instance, directly from Eqn.~\ref{eq:moment-Z} that the denominator polynomial for the $r$-th moment of the unconditional FPT, $\mu^{(r)}_{u,Z}$, must be $w(\Phi_{Z}(G))^r$, while this is harder to discern from Theorem~\ref{thm:hill-moment-Z}. However, Theorems~\ref{thm:hill-split}, \ref{thm:hill-moment-Z} and \ref{thm:hill-moment} are formulated, through the Hill operator, in terms of the graph $\Hs_u[G]$, which is always smaller than the graph $G$ itself. If the latter has $N$ vertices and the set of target states, $Z$, has size $\# Z = T$, then $\Hs_u[G]$ has $N-T$ vertices. Calculations are usually easier in a smaller graph and the altered structure of $\Hs_u[G]$ in comparison to $G$ may also be informative. As for the new Theorem~\ref{thm:kac-moment}, we showed in Eqn.~\ref{e-trug} how details of the rational algebraic structure of $\tau^{(r)}_u(G)$ can be readily derived.

Another feature of the new formulas is the appearance of s.s.~probabilities. One biophysical problem that has been of considerable interest in recent years is the existence of \textit{Hopfield barriers}, which are the fundamental limits to information processing for systems at thermodynamic equilibrium \cite{estrada:16}. In previous work, we have determined the Hopfield barrier for sharpness of input-output functions \cite{martinez-corral:24}. An important but challenging open question is whether such Hopfield barriers also exist for information processing in the transient regime, when systems are relaxing to a s.s.~of thermodynamic equilibrium. Steady-state probabilities simplify dramatically at thermodynamic equilibrium, as specified by equilibrium statistical mechanics \cite{nam:22}, but, up to now, it has been difficult to see how the eventual relaxation to thermodynamic equilibrium can influence transient properties. All the new formulas presented here include s.s.~probabilities. This is particularly interesting for the recurrence times in Theorem~\ref{thm:kac-moment} because they are formulated for strongly connected graphs, in which the conditions for thermodynamic equilibrium may be specified. Of course, the impact of thermodynamic equilibrium on the exchange factors still has to be addressed, but there is now at least an opening through which the question of transient Hopfield barriers may potentially be tackled.

A serious problem with undertaking spanning forest calculations is that forest enumeration rapidly becomes intractable and is already tedious even for quite small graphs, such as the five-vertex cycle graph, $K$, in Fig.~\ref{fig:example}A. The calculation of the second moment of the recurrence time, $\tau^{(2)}_1(K)$, in \S\ref{sec:example} works through the recursive procedure introduced in \cite{chebotarev:02}, which addresses some of these calculational challenges. The resulting formula in Appendix~\ref{app:example} is impressively complicated for just the second moment, especially in comparison to the first moment. The treatment of this example should be of broader interest in showing how spanning forest calculations can be undertaken with the best available current methods.

Perhaps the most interesting mathematical feature of the present paper is the emergence of operators on linear framework graphs. The graph-theoretic approach of the linear framework offers an attractive setting in which to formally define such operators on the corresponding Markov processes, as we did in providing a rigorous basis for Hill's idea. We introduced $\Hs_u$ and $\Us_v$ quite independently of each other, the former to formalise Hill's approach and the latter to provide a setting for Kac's approach. These two historical approaches are widely separated both in time and in intellectual context, so it came as a considerable surprise to discover the inverse relationship between our two operators (Lemma~\ref{t-hug}), as expressed in Eqn.~\ref{e-hsuu}. As noted above, this inverse property is crucial for showing how the Kac-like formulas for recurrence times in Theorem~\ref{thm:kac-moment} can be deduced from the Hill-like formulas for FPTs in Theorem~\ref{thm:hill-moment-Z}. The unexpected algebraic relationship between the two operators does raise the question as to whether there is an \textit{operator algebra} of some nature lurking somewhere in the background. We are not aware of such a development, either within graph theory or the theory of Markov processes. Of course, the two operators introduced here have different domains of definition, so any such algebra would have to be carefully specified. Nevertheless, the centrality of the operators to the results of the present paper suggest that exploring other kinds of operators on linear framework graphs may be a productive direction for future work. 

\clearpage
\appendix

\section{Moments of recurrence times for the example}
\label{app:example}

In \S\ref{sec:example} we showed how to use the results of the present paper to calculate the first moment, $\tau^{(1)}_1$, and the second moment, $\tau^{(2)}_1$ of the recurrence time to vertex $1$ for the example graph $K$ in Fig.~\ref{fig:example}A. For the sake of completeness, we provide here the complete expressions for these quantities as manifestly integrally positive rational algebraic functions of the edge labels.

The first moment is relatively straightforward. Let us write $\tau^{(1)}_1 = R/S$. Putting together Eqns.~\ref{e-t1K}, \ref{eq:ex-w1} and \ref{eq:ex-trees}, we see that, 
\begin{align*}
    R & = k_{2,1} k_{3,2} k_{4,3} k_{5,4} + k_{2,1} k_{3,2} k_{4,3} k_{5,1} + k_{2,1} k_{3,2} k_{4,5} k_{5,1} + k_{2,1} k_{3,4} k_{4,5} k_{5,1} \\
    & \quad + k_{2,3} k_{3,4} k_{4,5} k_{5,1} + k_{1,5} k_{3,2} k_{4,3} k_{5,4} + k_{1,2} k_{3,2} k_{4,3} k_{5,4} + k_{1,2} k_{3,2} k_{4,3} k_{5,1} \\
    & \quad + k_{1,2} k_{3,2} k_{4,5} k_{5,1} + k_{1,2} k_{3,4} k_{4,5} k_{5,1} + k_{1,5} k_{2,1} k_{4,3} k_{5,4} + k_{1,5} k_{2,3} k_{4,3} k_{5,4} \\
    & \quad + k_{1,2} k_{2,3} k_{4,3} k_{5,4} + k_{1,2} k_{2,3} k_{4,3} k_{5,1} + k_{1,2} k_{2,3} k_{4,5} k_{5,1} + k_{1,5} k_{2,1} k_{3,2} k_{5,4} \\
    & \quad + k_{1,5} k_{2,1} k_{3,4} k_{5,4} + k_{1,5} k_{2,3} k_{3,4} k_{5,4} + k_{1,2} k_{2,3} k_{3,4} k_{5,4} + k_{1,2} k_{2,3} k_{3,4} k_{5,1} \\
    & \quad + k_{1,5} k_{2,1} k_{3,2} k_{4,3} + k_{1,5} k_{2,1} k_{3,2} k_{4,5} + k_{1,5} k_{2,1} k_{3,4} k_{4,5} + k_{1,5} k_{2,3} k_{3,4} k_{4,5} \\
    & \quad + k_{1,2} k_{2,3} k_{3,4} k_{4,5}
\end{align*}
and 
\begin{align*}
    S & = \left( k_{1,2} + k_{1,5} \right) \left( k_{2,1} k_{3,2} k_{4,3} k_{5,4} + k_{2,1} k_{3,2} k_{4,3} k_{5,1} + k_{2,1} k_{3,2} k_{4,5} k_{5,1} \right. \\
    & \qquad \left. + k_{2,1} k_{3,4} k_{4,5} k_{5,1} + k_{2,3} k_{3,4} k_{4,5} k_{5,1} \right).
\end{align*}

The second moment is substantially more complicated. Let us write $\tau^{(2)}_1 = 2P/Q$. We can combine Eqn~\ref{eq:ex-tau2} with the spanning tree weights in Eqns.~\ref{eq:ex-w1} and \ref{eq:ex-trees} and the exchange factors, $\psi_{i_2 \mid i_1, u}$, calculated at the end of \S\ref{sec:example} to see that, 
\begin{align*}
    P & =
    k_{1,2}^{2} k_{2,1} k_{2,3} k_{3,2} k_{3,4} k_{4,3} k_{4,5}
    + k_{1,2}^{2} k_{2,1} k_{2,3} k_{3,2} k_{3,4} k_{4,5}^{2}
    + k_{1,2}^{2} k_{2,1} k_{2,3} k_{3,2} k_{3,4} k_{4,5} k_{5,1} \\
    & \hspace{-0.5em} + 2 k_{1,2}^{2} k_{2,1} k_{2,3} k_{3,2} k_{3,4} k_{4,5} k_{5,4}
    + k_{1,2}^{2} k_{2,1} k_{2,3} k_{3,2} k_{3,4} k_{5,1}^{2}
    + 2 k_{1,2}^{2} k_{2,1} k_{2,3} k_{3,2} k_{3,4} k_{5,1} k_{5,4} \\
    & \hspace{-0.5em} + k_{1,2}^{2} k_{2,1} k_{2,3} k_{3,2} k_{3,4} k_{5,4}^{2}
    + k_{1,2}^{2} k_{2,1} k_{2,3} k_{3,4}^{2} k_{4,5}^{2}
    + k_{1,2}^{2} k_{2,1} k_{2,3} k_{3,4}^{2} k_{4,5} k_{5,1} \\
    & \hspace{-0.5em} + 2 k_{1,2}^{2} k_{2,1} k_{2,3} k_{3,4}^{2} k_{4,5} k_{5,4}
    + k_{1,2}^{2} k_{2,1} k_{2,3} k_{3,4}^{2} k_{5,1}^{2}
    + 2 k_{1,2}^{2} k_{2,1} k_{2,3} k_{3,4}^{2} k_{5,1} k_{5,4} \\
    & \hspace{-0.5em} + k_{1,2}^{2} k_{2,1} k_{2,3} k_{3,4}^{2} k_{5,4}^{2}
    + k_{1,2}^{2} k_{2,1} k_{2,3} k_{3,4} k_{4,3} k_{4,5} k_{5,1}
    + 2 k_{1,2}^{2} k_{2,1} k_{2,3} k_{3,4} k_{4,3} k_{4,5} k_{5,4} \\
    & \hspace{-0.5em} + 2 k_{1,2}^{2} k_{2,1} k_{2,3} k_{3,4} k_{4,3} k_{5,1}^{2}
    + 4 k_{1,2}^{2} k_{2,1} k_{2,3} k_{3,4} k_{4,3} k_{5,1} k_{5,4}
    + 2 k_{1,2}^{2} k_{2,1} k_{2,3} k_{3,4} k_{4,3} k_{5,4}^{2} \\
    & \hspace{-0.5em} + k_{1,2}^{2} k_{2,1} k_{2,3} k_{3,4} k_{4,5}^{2} k_{5,1}
    + k_{1,2}^{2} k_{2,1} k_{2,3} k_{3,4} k_{4,5} k_{5,1}^{2}
    + k_{1,2}^{2} k_{2,1} k_{2,3} k_{3,4} k_{4,5} k_{5,1} k_{5,4} \\
    & \hspace{-0.5em} + k_{1,2}^{2} k_{2,1} k_{2,3} k_{4,3}^{2} k_{5,1}^{2}
    + 2 k_{1,2}^{2} k_{2,1} k_{2,3} k_{4,3}^{2} k_{5,1} k_{5,4}
    + k_{1,2}^{2} k_{2,1} k_{2,3} k_{4,3}^{2} k_{5,4}^{2} \\
    & \hspace{-0.5em} + 2 k_{1,2}^{2} k_{2,1} k_{2,3} k_{4,3} k_{4,5} k_{5,1}^{2}
    + 2 k_{1,2}^{2} k_{2,1} k_{2,3} k_{4,3} k_{4,5} k_{5,1} k_{5,4}
    + k_{1,2}^{2} k_{2,1} k_{2,3} k_{4,5}^{2} k_{5,1}^{2} \\
    & \hspace{-0.5em} + k_{1,2}^{2} k_{2,3}^{2} k_{3,4}^{2} k_{4,5}^{2}
    + k_{1,2}^{2} k_{2,3}^{2} k_{3,4}^{2} k_{4,5} k_{5,1}
    + 2 k_{1,2}^{2} k_{2,3}^{2} k_{3,4}^{2} k_{4,5} k_{5,4} \\
    & \hspace{-0.5em} + k_{1,2}^{2} k_{2,3}^{2} k_{3,4}^{2} k_{5,1}^{2}
    + 2 k_{1,2}^{2} k_{2,3}^{2} k_{3,4}^{2} k_{5,1} k_{5,4}
    + k_{1,2}^{2} k_{2,3}^{2} k_{3,4}^{2} k_{5,4}^{2} \\
    & \hspace{-0.5em} + k_{1,2}^{2} k_{2,3}^{2} k_{3,4} k_{4,3} k_{4,5} k_{5,1}
    + 2 k_{1,2}^{2} k_{2,3}^{2} k_{3,4} k_{4,3} k_{4,5} k_{5,4}
    + 2 k_{1,2}^{2} k_{2,3}^{2} k_{3,4} k_{4,3} k_{5,1}^{2} \\
    & \hspace{-0.5em} + 4 k_{1,2}^{2} k_{2,3}^{2} k_{3,4} k_{4,3} k_{5,1} k_{5,4}
    + 2 k_{1,2}^{2} k_{2,3}^{2} k_{3,4} k_{4,3} k_{5,4}^{2}
    + k_{1,2}^{2} k_{2,3}^{2} k_{3,4} k_{4,5}^{2} k_{5,1} \\
    & \hspace{-0.5em} + k_{1,2}^{2} k_{2,3}^{2} k_{3,4} k_{4,5} k_{5,1}^{2}
    + k_{1,2}^{2} k_{2,3}^{2} k_{3,4} k_{4,5} k_{5,1} k_{5,4}
    + k_{1,2}^{2} k_{2,3}^{2} k_{4,3}^{2} k_{5,1}^{2} \\
    & \hspace{-0.5em} + 2 k_{1,2}^{2} k_{2,3}^{2} k_{4,3}^{2} k_{5,1} k_{5,4}
    + k_{1,2}^{2} k_{2,3}^{2} k_{4,3}^{2} k_{5,4}^{2}
    + 2 k_{1,2}^{2} k_{2,3}^{2} k_{4,3} k_{4,5} k_{5,1}^{2} \\
    & \hspace{-0.5em} + 2 k_{1,2}^{2} k_{2,3}^{2} k_{4,3} k_{4,5} k_{5,1} k_{5,4}
    + k_{1,2}^{2} k_{2,3}^{2} k_{4,5}^{2} k_{5,1}^{2}
    + k_{1,2}^{2} k_{2,3} k_{3,2} k_{3,4} k_{4,3} k_{4,5} k_{5,1} \\
    & \hspace{-0.5em} + 2 k_{1,2}^{2} k_{2,3} k_{3,2} k_{3,4} k_{4,3} k_{4,5} k_{5,4}
    + 2 k_{1,2}^{2} k_{2,3} k_{3,2} k_{3,4} k_{4,3} k_{5,1}^{2} \\
    & \hspace{-0.5em} + 4 k_{1,2}^{2} k_{2,3} k_{3,2} k_{3,4} k_{4,3} k_{5,1} k_{5,4}
    + 2 k_{1,2}^{2} k_{2,3} k_{3,2} k_{3,4} k_{4,3} k_{5,4}^{2}
    + k_{1,2}^{2} k_{2,3} k_{3,2} k_{3,4} k_{4,5}^{2} k_{5,1} \\
    & \hspace{-0.5em} + k_{1,2}^{2} k_{2,3} k_{3,2} k_{3,4} k_{4,5} k_{5,1}^{2}
    + k_{1,2}^{2} k_{2,3} k_{3,2} k_{3,4} k_{4,5} k_{5,1} k_{5,4}
    + 2 k_{1,2}^{2} k_{2,3} k_{3,2} k_{4,3}^{2} k_{5,1}^{2} \\
    & \hspace{-0.5em} + 4 k_{1,2}^{2} k_{2,3} k_{3,2} k_{4,3}^{2} k_{5,1} k_{5,4}
    + 2 k_{1,2}^{2} k_{2,3} k_{3,2} k_{4,3}^{2} k_{5,4}^{2}
    + 4 k_{1,2}^{2} k_{2,3} k_{3,2} k_{4,3} k_{4,5} k_{5,1}^{2} \\
    & \hspace{-0.5em} + 4 k_{1,2}^{2} k_{2,3} k_{3,2} k_{4,3} k_{4,5} k_{5,1} k_{5,4}
    + 2 k_{1,2}^{2} k_{2,3} k_{3,2} k_{4,5}^{2} k_{5,1}^{2}
    + k_{1,2}^{2} k_{2,3} k_{3,4}^{2} k_{4,5}^{2} k_{5,1} \\
    & \hspace{-0.5em} + k_{1,2}^{2} k_{2,3} k_{3,4}^{2} k_{4,5} k_{5,1}^{2}
    + k_{1,2}^{2} k_{2,3} k_{3,4}^{2} k_{4,5} k_{5,1} k_{5,4}
    + k_{1,2}^{2} k_{2,3} k_{3,4} k_{4,3} k_{4,5} k_{5,1}^{2} \\
    & \hspace{-0.5em} + k_{1,2}^{2} k_{2,3} k_{3,4} k_{4,3} k_{4,5} k_{5,1} k_{5,4}
    + k_{1,2}^{2} k_{2,3} k_{3,4} k_{4,5}^{2} k_{5,1}^{2}
    + k_{1,2}^{2} k_{3,2}^{2} k_{4,3}^{2} k_{5,1}^{2} \\
    & \hspace{-0.5em} + 2 k_{1,2}^{2} k_{3,2}^{2} k_{4,3}^{2} k_{5,1} k_{5,4}
    + k_{1,2}^{2} k_{3,2}^{2} k_{4,3}^{2} k_{5,4}^{2}
    + 2 k_{1,2}^{2} k_{3,2}^{2} k_{4,3} k_{4,5} k_{5,1}^{2} \\
    & \hspace{-0.5em} + 2 k_{1,2}^{2} k_{3,2}^{2} k_{4,3} k_{4,5} k_{5,1} k_{5,4}
    + k_{1,2}^{2} k_{3,2}^{2} k_{4,5}^{2} k_{5,1}^{2}
    + 2 k_{1,2}^{2} k_{3,2} k_{3,4} k_{4,3} k_{4,5} k_{5,1}^{2} \\
    & \hspace{-0.5em} + 2 k_{1,2}^{2} k_{3,2} k_{3,4} k_{4,3} k_{4,5} k_{5,1} k_{5,4}
    + 2 k_{1,2}^{2} k_{3,2} k_{3,4} k_{4,5}^{2} k_{5,1}^{2}
    + k_{1,2}^{2} k_{3,4}^{2} k_{4,5}^{2} k_{5,1}^{2} \\
    & \hspace{-0.5em} + k_{1,2} k_{1,5} k_{2,1}^{2} k_{3,2}^{2} k_{4,3}^{2}
    + 2 k_{1,2} k_{1,5} k_{2,1}^{2} k_{3,2}^{2} k_{4,3} k_{4,5}
    + k_{1,2} k_{1,5} k_{2,1}^{2} k_{3,2}^{2} k_{4,3} k_{5,4} \\
    & \hspace{-0.5em} + k_{1,2} k_{1,5} k_{2,1}^{2} k_{3,2}^{2} k_{4,5}^{2}
    + 2 k_{1,2} k_{1,5} k_{2,1}^{2} k_{3,2}^{2} k_{4,5} k_{5,4}
    + k_{1,2} k_{1,5} k_{2,1}^{2} k_{3,2}^{2} k_{5,1} k_{5,4} \\
    & \hspace{-0.5em} + k_{1,2} k_{1,5} k_{2,1}^{2} k_{3,2}^{2} k_{5,4}^{2}
    + 2 k_{1,2} k_{1,5} k_{2,1}^{2} k_{3,2} k_{3,4} k_{4,3} k_{4,5}
    + k_{1,2} k_{1,5} k_{2,1}^{2} k_{3,2} k_{3,4} k_{4,3} k_{5,4} \\
    & \hspace{-0.5em} + 2 k_{1,2} k_{1,5} k_{2,1}^{2} k_{3,2} k_{3,4} k_{4,5}^{2}
    + 4 k_{1,2} k_{1,5} k_{2,1}^{2} k_{3,2} k_{3,4} k_{4,5} k_{5,4} \\
    & \hspace{-0.5em} + 2 k_{1,2} k_{1,5} k_{2,1}^{2} k_{3,2} k_{3,4} k_{5,1} k_{5,4}
    + 2 k_{1,2} k_{1,5} k_{2,1}^{2} k_{3,2} k_{3,4} k_{5,4}^{2}
    + k_{1,2} k_{1,5} k_{2,1}^{2} k_{3,2} k_{4,3}^{2} k_{5,4} \\
    & \hspace{-0.5em} + k_{1,2} k_{1,5} k_{2,1}^{2} k_{3,2} k_{4,3} k_{4,5} k_{5,4}
    + k_{1,2} k_{1,5} k_{2,1}^{2} k_{3,2} k_{4,3} k_{5,1} k_{5,4}
    + k_{1,2} k_{1,5} k_{2,1}^{2} k_{3,2} k_{4,3} k_{5,4}^{2} \\
    & \hspace{-0.5em} + k_{1,2} k_{1,5} k_{2,1}^{2} k_{3,4}^{2} k_{4,5}^{2}
    + 2 k_{1,2} k_{1,5} k_{2,1}^{2} k_{3,4}^{2} k_{4,5} k_{5,4}
    + k_{1,2} k_{1,5} k_{2,1}^{2} k_{3,4}^{2} k_{5,1} k_{5,4} \\
    & \hspace{-0.5em} + k_{1,2} k_{1,5} k_{2,1}^{2} k_{3,4}^{2} k_{5,4}^{2}
    + 2 k_{1,2} k_{1,5} k_{2,1}^{2} k_{3,4} k_{4,3} k_{4,5} k_{5,4}
    + 2 k_{1,2} k_{1,5} k_{2,1}^{2} k_{3,4} k_{4,3} k_{5,1} k_{5,4} \\
    & \hspace{-0.5em} + 2 k_{1,2} k_{1,5} k_{2,1}^{2} k_{3,4} k_{4,3} k_{5,4}^{2}
    + k_{1,2} k_{1,5} k_{2,1}^{2} k_{4,3}^{2} k_{5,1} k_{5,4}
    + k_{1,2} k_{1,5} k_{2,1}^{2} k_{4,3}^{2} k_{5,4}^{2} \\
    & \hspace{-0.5em} + k_{1,2} k_{1,5} k_{2,1}^{2} k_{4,3} k_{4,5} k_{5,1} k_{5,4}
    + 3 k_{1,2} k_{1,5} k_{2,1} k_{2,3} k_{3,2} k_{3,4} k_{4,3} k_{4,5} \\
    & \hspace{-0.5em} + k_{1,2} k_{1,5} k_{2,1} k_{2,3} k_{3,2} k_{3,4} k_{4,3} k_{5,4}
    + 3 k_{1,2} k_{1,5} k_{2,1} k_{2,3} k_{3,2} k_{3,4} k_{4,5}^{2} \\
    & \hspace{-0.5em} + k_{1,2} k_{1,5} k_{2,1} k_{2,3} k_{3,2} k_{3,4} k_{4,5} k_{5,1}
    + 6 k_{1,2} k_{1,5} k_{2,1} k_{2,3} k_{3,2} k_{3,4} k_{4,5} k_{5,4} \\
    & \hspace{-0.5em} + k_{1,2} k_{1,5} k_{2,1} k_{2,3} k_{3,2} k_{3,4} k_{5,1}^{2}
    + 4 k_{1,2} k_{1,5} k_{2,1} k_{2,3} k_{3,2} k_{3,4} k_{5,1} k_{5,4} \\
    & \hspace{-0.5em} + 3 k_{1,2} k_{1,5} k_{2,1} k_{2,3} k_{3,2} k_{3,4} k_{5,4}^{2}
   + k_{1,2} k_{1,5} k_{2,1} k_{2,3} k_{3,2} k_{4,3}^{2} k_{5,4} \\
    & \hspace{-0.5em} + k_{1,2} k_{1,5} k_{2,1} k_{2,3} k_{3,2} k_{4,3} k_{4,5} k_{5,4}
    + k_{1,2} k_{1,5} k_{2,1} k_{2,3} k_{3,2} k_{4,3} k_{5,1} k_{5,4} \\
    & \hspace{-0.5em} + k_{1,2} k_{1,5} k_{2,1} k_{2,3} k_{3,2} k_{4,3} k_{5,4}^{2}
    + 3 k_{1,2} k_{1,5} k_{2,1} k_{2,3} k_{3,4}^{2} k_{4,5}^{2}
    + k_{1,2} k_{1,5} k_{2,1} k_{2,3} k_{3,4}^{2} k_{4,5} k_{5,1} \\
    & \hspace{-0.5em} + 6 k_{1,2} k_{1,5} k_{2,1} k_{2,3} k_{3,4}^{2} k_{4,5} k_{5,4}
    + k_{1,2} k_{1,5} k_{2,1} k_{2,3} k_{3,4}^{2} k_{5,1}^{2}
    + 4 k_{1,2} k_{1,5} k_{2,1} k_{2,3} k_{3,4}^{2} k_{5,1} k_{5,4} \\
    & \hspace{-0.5em} + 3 k_{1,2} k_{1,5} k_{2,1} k_{2,3} k_{3,4}^{2} k_{5,4}^{2}
    + k_{1,2} k_{1,5} k_{2,1} k_{2,3} k_{3,4} k_{4,3} k_{4,5} k_{5,1} \\
    & \hspace{-0.5em} + 6 k_{1,2} k_{1,5} k_{2,1} k_{2,3} k_{3,4} k_{4,3} k_{4,5} k_{5,4}
    + 2 k_{1,2} k_{1,5} k_{2,1} k_{2,3} k_{3,4} k_{4,3} k_{5,1}^{2} \\
    & \hspace{-0.5em} + 8 k_{1,2} k_{1,5} k_{2,1} k_{2,3} k_{3,4} k_{4,3} k_{5,1} k_{5,4}
    + 6 k_{1,2} k_{1,5} k_{2,1} k_{2,3} k_{3,4} k_{4,3} k_{5,4}^{2} \\
    & \hspace{-0.5em} + k_{1,2} k_{1,5} k_{2,1} k_{2,3} k_{3,4} k_{4,5}^{2} k_{5,1} 
    + k_{1,2} k_{1,5} k_{2,1} k_{2,3} k_{3,4} k_{4,5} k_{5,1}^{2} \\
    & \hspace{-0.5em} + k_{1,2} k_{1,5} k_{2,1} k_{2,3} k_{3,4} k_{4,5} k_{5,1} k_{5,4}
    + k_{1,2} k_{1,5} k_{2,1} k_{2,3} k_{4,3}^{2} k_{5,1}^{2} \\
    & \hspace{-0.5em} + 4 k_{1,2} k_{1,5} k_{2,1} k_{2,3} k_{4,3}^{2} k_{5,1} k_{5,4}
    + 3 k_{1,2} k_{1,5} k_{2,1} k_{2,3} k_{4,3}^{2} k_{5,4}^{2}
    + 2 k_{1,2} k_{1,5} k_{2,1} k_{2,3} k_{4,3} k_{4,5} k_{5,1}^{2} \\
    & \hspace{-0.5em} + 4 k_{1,2} k_{1,5} k_{2,1} k_{2,3} k_{4,3} k_{4,5} k_{5,1} k_{5,4}
    + k_{1,2} k_{1,5} k_{2,1} k_{2,3} k_{4,5}^{2} k_{5,1}^{2}
    + k_{1,2} k_{1,5} k_{2,1} k_{3,2}^{2} k_{4,3}^{2} k_{5,4} \\
    & \hspace{-0.5em} + k_{1,2} k_{1,5} k_{2,1} k_{3,2}^{2} k_{4,3} k_{4,5} k_{5,4}
    + k_{1,2} k_{1,5} k_{2,1} k_{3,2}^{2} k_{4,3} k_{5,1} k_{5,4}
    + k_{1,2} k_{1,5} k_{2,1} k_{3,2}^{2} k_{4,3} k_{5,4}^{2} \\
    & \hspace{-0.5em} + k_{1,2} k_{1,5} k_{2,1} k_{3,2} k_{3,4} k_{4,3} k_{4,5} k_{5,4}
    + k_{1,2} k_{1,5} k_{2,1} k_{3,2} k_{3,4} k_{4,3} k_{5,1} k_{5,4} \\
    & \hspace{-0.5em} + k_{1,2} k_{1,5} k_{2,1} k_{3,2} k_{3,4} k_{4,3} k_{5,4}^{2}
    + k_{1,2} k_{1,5} k_{2,1} k_{3,2} k_{4,3}^{2} k_{5,1} k_{5,4}
    + k_{1,2} k_{1,5} k_{2,1} k_{3,2} k_{4,3}^{2} k_{5,4}^{2} \\
    & \hspace{-0.5em} + k_{1,2} k_{1,5} k_{2,1} k_{3,2} k_{4,3} k_{4,5} k_{5,1} k_{5,4}
    + 2 k_{1,2} k_{1,5} k_{2,3}^{2} k_{3,4}^{2} k_{4,5}^{2}
    + k_{1,2} k_{1,5} k_{2,3}^{2} k_{3,4}^{2} k_{4,5} k_{5,1} \\
    & \hspace{-0.5em} + 4 k_{1,2} k_{1,5} k_{2,3}^{2} k_{3,4}^{2} k_{4,5} k_{5,4}
    + k_{1,2} k_{1,5} k_{2,3}^{2} k_{3,4}^{2} k_{5,1}^{2}
    + 3 k_{1,2} k_{1,5} k_{2,3}^{2} k_{3,4}^{2} k_{5,1} k_{5,4} \\
    & \hspace{-0.5em} + 2 k_{1,2} k_{1,5} k_{2,3}^{2} k_{3,4}^{2} k_{5,4}^{2}
    + k_{1,2} k_{1,5} k_{2,3}^{2} k_{3,4} k_{4,3} k_{4,5} k_{5,1}
    + 4 k_{1,2} k_{1,5} k_{2,3}^{2} k_{3,4} k_{4,3} k_{4,5} k_{5,4} \\
    & \hspace{-0.5em} + 2 k_{1,2} k_{1,5} k_{2,3}^{2} k_{3,4} k_{4,3} k_{5,1}^{2}
    + 6 k_{1,2} k_{1,5} k_{2,3}^{2} k_{3,4} k_{4,3} k_{5,1} k_{5,4}
    + 4 k_{1,2} k_{1,5} k_{2,3}^{2} k_{3,4} k_{4,3} k_{5,4}^{2} \\
    & \hspace{-0.5em} + k_{1,2} k_{1,5} k_{2,3}^{2} k_{3,4} k_{4,5}^{2} k_{5,1}
    + k_{1,2} k_{1,5} k_{2,3}^{2} k_{3,4} k_{4,5} k_{5,1}^{2}
    + k_{1,2} k_{1,5} k_{2,3}^{2} k_{3,4} k_{4,5} k_{5,1} k_{5,4} \\
    & \hspace{-0.5em} + k_{1,2} k_{1,5} k_{2,3}^{2} k_{4,3}^{2} k_{5,1}^{2}
    + 3 k_{1,2} k_{1,5} k_{2,3}^{2} k_{4,3}^{2} k_{5,1} k_{5,4}
    + 2 k_{1,2} k_{1,5} k_{2,3}^{2} k_{4,3}^{2} k_{5,4}^{2} \\
    & \hspace{-0.5em} + 2 k_{1,2} k_{1,5} k_{2,3}^{2} k_{4,3} k_{4,5} k_{5,1}^{2}
    + 3 k_{1,2} k_{1,5} k_{2,3}^{2} k_{4,3} k_{4,5} k_{5,1} k_{5,4}
    + k_{1,2} k_{1,5} k_{2,3}^{2} k_{4,5}^{2} k_{5,1}^{2} \\
    & \hspace{-0.5em} + k_{1,2} k_{1,5} k_{2,3} k_{3,2} k_{3,4} k_{4,3} k_{4,5} k_{5,1}
    + 4 k_{1,2} k_{1,5} k_{2,3} k_{3,2} k_{3,4} k_{4,3} k_{4,5} k_{5,4} \\
    & \hspace{-0.5em} + 2 k_{1,2} k_{1,5} k_{2,3} k_{3,2} k_{3,4} k_{4,3} k_{5,1}^{2}
    + 6 k_{1,2} k_{1,5} k_{2,3} k_{3,2} k_{3,4} k_{4,3} k_{5,1} k_{5,4} \\
    & \hspace{-0.5em} + 4 k_{1,2} k_{1,5} k_{2,3} k_{3,2} k_{3,4} k_{4,3} k_{5,4}^{2}
    + k_{1,2} k_{1,5} k_{2,3} k_{3,2} k_{3,4} k_{4,5}^{2} k_{5,1} \\
    & \hspace{-0.5em} + k_{1,2} k_{1,5} k_{2,3} k_{3,2} k_{3,4} k_{4,5} k_{5,1}^{2}
    + k_{1,2} k_{1,5} k_{2,3} k_{3,2} k_{3,4} k_{4,5} k_{5,1} k_{5,4} \\
    & \hspace{-0.5em} + 2 k_{1,2} k_{1,5} k_{2,3} k_{3,2} k_{4,3}^{2} k_{5,1}^{2} 
    + 6 k_{1,2} k_{1,5} k_{2,3} k_{3,2} k_{4,3}^{2} k_{5,1} k_{5,4}
    + 4 k_{1,2} k_{1,5} k_{2,3} k_{3,2} k_{4,3}^{2} k_{5,4}^{2} \\
    & \hspace{-0.5em} + 4 k_{1,2} k_{1,5} k_{2,3} k_{3,2} k_{4,3} k_{4,5} k_{5,1}^{2}
    + 6 k_{1,2} k_{1,5} k_{2,3} k_{3,2} k_{4,3} k_{4,5} k_{5,1} k_{5,4} \\
    & \hspace{-0.5em} + 2 k_{1,2} k_{1,5} k_{2,3} k_{3,2} k_{4,5}^{2} k_{5,1}^{2}
    + k_{1,2} k_{1,5} k_{2,3} k_{3,4}^{2} k_{4,5}^{2} k_{5,1}
    + k_{1,2} k_{1,5} k_{2,3} k_{3,4}^{2} k_{4,5} k_{5,1}^{2} \\
    & \hspace{-0.5em} + k_{1,2} k_{1,5} k_{2,3} k_{3,4}^{2} k_{4,5} k_{5,1} k_{5,4}
    + k_{1,2} k_{1,5} k_{2,3} k_{3,4} k_{4,3} k_{4,5} k_{5,1}^{2} \\
    & \hspace{-0.5em} + k_{1,2} k_{1,5} k_{2,3} k_{3,4} k_{4,3} k_{4,5} k_{5,1} k_{5,4}
    + k_{1,2} k_{1,5} k_{2,3} k_{3,4} k_{4,5}^{2} k_{5,1}^{2}
    + k_{1,2} k_{1,5} k_{3,2}^{2} k_{4,3}^{2} k_{5,1}^{2} \\
    & \hspace{-0.5em} + 3 k_{1,2} k_{1,5} k_{3,2}^{2} k_{4,3}^{2} k_{5,1} k_{5,4}
    + 2 k_{1,2} k_{1,5} k_{3,2}^{2} k_{4,3}^{2} k_{5,4}^{2}
    + 2 k_{1,2} k_{1,5} k_{3,2}^{2} k_{4,3} k_{4,5} k_{5,1}^{2} \\
    & \hspace{-0.5em} + 3 k_{1,2} k_{1,5} k_{3,2}^{2} k_{4,3} k_{4,5} k_{5,1} k_{5,4}
    + k_{1,2} k_{1,5} k_{3,2}^{2} k_{4,5}^{2} k_{5,1}^{2}
    + 2 k_{1,2} k_{1,5} k_{3,2} k_{3,4} k_{4,3} k_{4,5} k_{5,1}^{2} \\
    & \hspace{-0.5em} + 3 k_{1,2} k_{1,5} k_{3,2} k_{3,4} k_{4,3} k_{4,5} k_{5,1} k_{5,4}
    + 2 k_{1,2} k_{1,5} k_{3,2} k_{3,4} k_{4,5}^{2} k_{5,1}^{2}
    + k_{1,2} k_{1,5} k_{3,4}^{2} k_{4,5}^{2} k_{5,1}^{2} \\
    & \hspace{-0.5em} + k_{1,2} k_{2,1} k_{2,3} k_{3,2} k_{3,4} k_{4,3} k_{4,5} k_{5,1}
    + k_{1,2} k_{2,1} k_{2,3} k_{3,2} k_{3,4} k_{4,3} k_{4,5} k_{5,4} \\
    & \hspace{-0.5em} + k_{1,2} k_{2,1} k_{2,3} k_{3,2} k_{3,4} k_{4,3} k_{5,1}^{2}
    + 2 k_{1,2} k_{2,1} k_{2,3} k_{3,2} k_{3,4} k_{4,3} k_{5,1} k_{5,4} \\
    & \hspace{-0.5em} + k_{1,2} k_{2,1} k_{2,3} k_{3,2} k_{3,4} k_{4,3} k_{5,4}^{2}
    + k_{1,2} k_{2,1} k_{2,3} k_{3,2} k_{3,4} k_{4,5}^{2} k_{5,1}
    + k_{1,2} k_{2,1} k_{2,3} k_{3,2} k_{3,4} k_{4,5} k_{5,1}^{2} \\
    & \hspace{-0.5em} + k_{1,2} k_{2,1} k_{2,3} k_{3,2} k_{3,4} k_{4,5} k_{5,1} k_{5,4}
    + k_{1,2} k_{2,1} k_{2,3} k_{3,2} k_{4,3}^{2} k_{5,1}^{2} \\
    & \hspace{-0.5em} + 2 k_{1,2} k_{2,1} k_{2,3} k_{3,2} k_{4,3}^{2} k_{5,1} k_{5,4} 
    + k_{1,2} k_{2,1} k_{2,3} k_{3,2} k_{4,3}^{2} k_{5,4}^{2}
    + 2 k_{1,2} k_{2,1} k_{2,3} k_{3,2} k_{4,3} k_{4,5} k_{5,1}^{2} \\
    & \hspace{-0.5em} + 2 k_{1,2} k_{2,1} k_{2,3} k_{3,2} k_{4,3} k_{4,5} k_{5,1} k_{5,4}
    + k_{1,2} k_{2,1} k_{2,3} k_{3,2} k_{4,5}^{2} k_{5,1}^{2}
    + k_{1,2} k_{2,1} k_{2,3} k_{3,4}^{2} k_{4,5}^{2} k_{5,1} \\
    & \hspace{-0.5em} + k_{1,2} k_{2,1} k_{2,3} k_{3,4}^{2} k_{4,5} k_{5,1}^{2}
    + k_{1,2} k_{2,1} k_{2,3} k_{3,4}^{2} k_{4,5} k_{5,1} k_{5,4}
    + k_{1,2} k_{2,1} k_{2,3} k_{3,4} k_{4,3} k_{4,5} k_{5,1}^{2} \\
    & \hspace{-0.5em} + k_{1,2} k_{2,1} k_{2,3} k_{3,4} k_{4,3} k_{4,5} k_{5,1} k_{5,4}
    + k_{1,2} k_{2,1} k_{2,3} k_{3,4} k_{4,5}^{2} k_{5,1}^{2}
    + k_{1,2} k_{2,1} k_{3,2}^{2} k_{4,3}^{2} k_{5,1}^{2} \\
    & \hspace{-0.5em} + 2 k_{1,2} k_{2,1} k_{3,2}^{2} k_{4,3}^{2} k_{5,1} k_{5,4}
    + k_{1,2} k_{2,1} k_{3,2}^{2} k_{4,3}^{2} k_{5,4}^{2}
    + 2 k_{1,2} k_{2,1} k_{3,2}^{2} k_{4,3} k_{4,5} k_{5,1}^{2} \\
    & \hspace{-0.5em} + 2 k_{1,2} k_{2,1} k_{3,2}^{2} k_{4,3} k_{4,5} k_{5,1} k_{5,4}
    + k_{1,2} k_{2,1} k_{3,2}^{2} k_{4,5}^{2} k_{5,1}^{2}
    + 2 k_{1,2} k_{2,1} k_{3,2} k_{3,4} k_{4,3} k_{4,5} k_{5,1}^{2} \\
    & \hspace{-0.5em} + 2 k_{1,2} k_{2,1} k_{3,2} k_{3,4} k_{4,3} k_{4,5} k_{5,1} k_{5,4}
    + 2 k_{1,2} k_{2,1} k_{3,2} k_{3,4} k_{4,5}^{2} k_{5,1}^{2}
    + k_{1,2} k_{2,1} k_{3,4}^{2} k_{4,5}^{2} k_{5,1}^{2} \\
    & \hspace{-0.5em} + k_{1,2} k_{2,3}^{2} k_{3,4}^{2} k_{4,5}^{2} k_{5,1}
    + k_{1,2} k_{2,3}^{2} k_{3,4}^{2} k_{4,5} k_{5,1}^{2}
    + k_{1,2} k_{2,3}^{2} k_{3,4}^{2} k_{4,5} k_{5,1} k_{5,4} \\
    & \hspace{-0.5em} + k_{1,2} k_{2,3}^{2} k_{3,4} k_{4,3} k_{4,5} k_{5,1}^{2}
    + k_{1,2} k_{2,3}^{2} k_{3,4} k_{4,3} k_{4,5} k_{5,1} k_{5,4} 
    + k_{1,2} k_{2,3}^{2} k_{3,4} k_{4,5}^{2} k_{5,1}^{2} \\
    & \hspace{-0.5em} + k_{1,2} k_{2,3} k_{3,2} k_{3,4} k_{4,3} k_{4,5} k_{5,1}^{2}
    + k_{1,2} k_{2,3} k_{3,2} k_{3,4} k_{4,3} k_{4,5} k_{5,1} k_{5,4} \\
    & \hspace{-0.5em} + k_{1,2} k_{2,3} k_{3,2} k_{3,4} k_{4,5}^{2} k_{5,1}^{2}
    + k_{1,2} k_{2,3} k_{3,4}^{2} k_{4,5}^{2} k_{5,1}^{2}
    + k_{1,5}^{2} k_{2,1}^{2} k_{3,2}^{2} k_{4,3}^{2} \\
    & \hspace{-0.5em} + 2 k_{1,5}^{2} k_{2,1}^{2} k_{3,2}^{2} k_{4,3} k_{4,5}
    + k_{1,5}^{2} k_{2,1}^{2} k_{3,2}^{2} k_{4,3} k_{5,4}
    + k_{1,5}^{2} k_{2,1}^{2} k_{3,2}^{2} k_{4,5}^{2} \\
    & \hspace{-0.5em} + 2 k_{1,5}^{2} k_{2,1}^{2} k_{3,2}^{2} k_{4,5} k_{5,4}
    + k_{1,5}^{2} k_{2,1}^{2} k_{3,2}^{2} k_{5,1} k_{5,4}
    + k_{1,5}^{2} k_{2,1}^{2} k_{3,2}^{2} k_{5,4}^{2} \\
    & \hspace{-0.5em} + 2 k_{1,5}^{2} k_{2,1}^{2} k_{3,2} k_{3,4} k_{4,3} k_{4,5}
    + k_{1,5}^{2} k_{2,1}^{2} k_{3,2} k_{3,4} k_{4,3} k_{5,4}
    + 2 k_{1,5}^{2} k_{2,1}^{2} k_{3,2} k_{3,4} k_{4,5}^{2} \\
    & \hspace{-0.5em} + 4 k_{1,5}^{2} k_{2,1}^{2} k_{3,2} k_{3,4} k_{4,5} k_{5,4}
    + 2 k_{1,5}^{2} k_{2,1}^{2} k_{3,2} k_{3,4} k_{5,1} k_{5,4}
    + 2 k_{1,5}^{2} k_{2,1}^{2} k_{3,2} k_{3,4} k_{5,4}^{2} \\
    & \hspace{-0.5em} + k_{1,5}^{2} k_{2,1}^{2} k_{3,2} k_{4,3}^{2} k_{5,4}
    + k_{1,5}^{2} k_{2,1}^{2} k_{3,2} k_{4,3} k_{4,5} k_{5,4}
    + k_{1,5}^{2} k_{2,1}^{2} k_{3,2} k_{4,3} k_{5,1} k_{5,4} \\
    & \hspace{-0.5em} + k_{1,5}^{2} k_{2,1}^{2} k_{3,2} k_{4,3} k_{5,4}^{2}
    + k_{1,5}^{2} k_{2,1}^{2} k_{3,4}^{2} k_{4,5}^{2}
    + 2 k_{1,5}^{2} k_{2,1}^{2} k_{3,4}^{2} k_{4,5} k_{5,4} \\
    & \hspace{-0.5em} + k_{1,5}^{2} k_{2,1}^{2} k_{3,4}^{2} k_{5,1} k_{5,4}
    + k_{1,5}^{2} k_{2,1}^{2} k_{3,4}^{2} k_{5,4}^{2}
    + 2 k_{1,5}^{2} k_{2,1}^{2} k_{3,4} k_{4,3} k_{4,5} k_{5,4} \\
    & \hspace{-0.5em} + 2 k_{1,5}^{2} k_{2,1}^{2} k_{3,4} k_{4,3} k_{5,1} k_{5,4}
    + 2 k_{1,5}^{2} k_{2,1}^{2} k_{3,4} k_{4,3} k_{5,4}^{2}
    + k_{1,5}^{2} k_{2,1}^{2} k_{4,3}^{2} k_{5,1} k_{5,4} \\
    & \hspace{-0.5em} + k_{1,5}^{2} k_{2,1}^{2} k_{4,3}^{2} k_{5,4}^{2}
    + k_{1,5}^{2} k_{2,1}^{2} k_{4,3} k_{4,5} k_{5,1} k_{5,4}
    + 2 k_{1,5}^{2} k_{2,1} k_{2,3} k_{3,2} k_{3,4} k_{4,3} k_{4,5} \\
    & \hspace{-0.5em} + k_{1,5}^{2} k_{2,1} k_{2,3} k_{3,2} k_{3,4} k_{4,3} k_{5,4}
    + 2 k_{1,5}^{2} k_{2,1} k_{2,3} k_{3,2} k_{3,4} k_{4,5}^{2} \\
    & \hspace{-0.5em} + 4 k_{1,5}^{2} k_{2,1} k_{2,3} k_{3,2} k_{3,4} k_{4,5} k_{5,4}
    + 2 k_{1,5}^{2} k_{2,1} k_{2,3} k_{3,2} k_{3,4} k_{5,1} k_{5,4} \\
    & \hspace{-0.5em} + 2 k_{1,5}^{2} k_{2,1} k_{2,3} k_{3,2} k_{3,4} k_{5,4}^{2}
    + k_{1,5}^{2} k_{2,1} k_{2,3} k_{3,2} k_{4,3}^{2} k_{5,4}
    + k_{1,5}^{2} k_{2,1} k_{2,3} k_{3,2} k_{4,3} k_{4,5} k_{5,4} \\
    & \hspace{-0.5em} + k_{1,5}^{2} k_{2,1} k_{2,3} k_{3,2} k_{4,3} k_{5,1} k_{5,4}
    + k_{1,5}^{2} k_{2,1} k_{2,3} k_{3,2} k_{4,3} k_{5,4}^{2}
    + 2 k_{1,5}^{2} k_{2,1} k_{2,3} k_{3,4}^{2} k_{4,5}^{2} \\
    & \hspace{-0.5em} + 4 k_{1,5}^{2} k_{2,1} k_{2,3} k_{3,4}^{2} k_{4,5} k_{5,4}
    + 2 k_{1,5}^{2} k_{2,1} k_{2,3} k_{3,4}^{2} k_{5,1} k_{5,4}
    + 2 k_{1,5}^{2} k_{2,1} k_{2,3} k_{3,4}^{2} k_{5,4}^{2} \\
    & \hspace{-0.5em} + 4 k_{1,5}^{2} k_{2,1} k_{2,3} k_{3,4} k_{4,3} k_{4,5} k_{5,4}
    + 4 k_{1,5}^{2} k_{2,1} k_{2,3} k_{3,4} k_{4,3} k_{5,1} k_{5,4} \\
    & \hspace{-0.5em} + 4 k_{1,5}^{2} k_{2,1} k_{2,3} k_{3,4} k_{4,3} k_{5,4}^{2}
    + 2 k_{1,5}^{2} k_{2,1} k_{2,3} k_{4,3}^{2} k_{5,1} k_{5,4}
    + 2 k_{1,5}^{2} k_{2,1} k_{2,3} k_{4,3}^{2} k_{5,4}^{2} \\
    & \hspace{-0.5em} + 2 k_{1,5}^{2} k_{2,1} k_{2,3} k_{4,3} k_{4,5} k_{5,1} k_{5,4}
    + k_{1,5}^{2} k_{2,1} k_{3,2}^{2} k_{4,3}^{2} k_{5,4}
    + k_{1,5}^{2} k_{2,1} k_{3,2}^{2} k_{4,3} k_{4,5} k_{5,4} \\
    & \hspace{-0.5em} + k_{1,5}^{2} k_{2,1} k_{3,2}^{2} k_{4,3} k_{5,1} k_{5,4}
    + k_{1,5}^{2} k_{2,1} k_{3,2}^{2} k_{4,3} k_{5,4}^{2}
    + k_{1,5}^{2} k_{2,1} k_{3,2} k_{3,4} k_{4,3} k_{4,5} k_{5,4} \\
    & \hspace{-0.5em} + k_{1,5}^{2} k_{2,1} k_{3,2} k_{3,4} k_{4,3} k_{5,1} k_{5,4}
    + k_{1,5}^{2} k_{2,1} k_{3,2} k_{3,4} k_{4,3} k_{5,4}^{2}
    + k_{1,5}^{2} k_{2,1} k_{3,2} k_{4,3}^{2} k_{5,1} k_{5,4} \\
    & \hspace{-0.5em} + k_{1,5}^{2} k_{2,1} k_{3,2} k_{4,3}^{2} k_{5,4}^{2}
    + k_{1,5}^{2} k_{2,1} k_{3,2} k_{4,3} k_{4,5} k_{5,1} k_{5,4}
    + k_{1,5}^{2} k_{2,3}^{2} k_{3,4}^{2} k_{4,5}^{2} \\
    & \hspace{-0.5em} + 2 k_{1,5}^{2} k_{2,3}^{2} k_{3,4}^{2} k_{4,5} k_{5,4}
    + k_{1,5}^{2} k_{2,3}^{2} k_{3,4}^{2} k_{5,1} k_{5,4}
    + k_{1,5}^{2} k_{2,3}^{2} k_{3,4}^{2} k_{5,4}^{2} \\
    & \hspace{-0.5em} + 2 k_{1,5}^{2} k_{2,3}^{2} k_{3,4} k_{4,3} k_{4,5} k_{5,4}
    + 2 k_{1,5}^{2} k_{2,3}^{2} k_{3,4} k_{4,3} k_{5,1} k_{5,4}
    + 2 k_{1,5}^{2} k_{2,3}^{2} k_{3,4} k_{4,3} k_{5,4}^{2} \\
    & \hspace{-0.5em} + k_{1,5}^{2} k_{2,3}^{2} k_{4,3}^{2} k_{5,1} k_{5,4}
    + k_{1,5}^{2} k_{2,3}^{2} k_{4,3}^{2} k_{5,4}^{2}
    + k_{1,5}^{2} k_{2,3}^{2} k_{4,3} k_{4,5} k_{5,1} k_{5,4} \\
    & \hspace{-0.5em} + 2 k_{1,5}^{2} k_{2,3} k_{3,2} k_{3,4} k_{4,3} k_{4,5} k_{5,4}
    + 2 k_{1,5}^{2} k_{2,3} k_{3,2} k_{3,4} k_{4,3} k_{5,1} k_{5,4} \\ 
    & \hspace{-0.5em} + 2 k_{1,5}^{2} k_{2,3} k_{3,2} k_{3,4} k_{4,3} k_{5,4}^{2}
    + 2 k_{1,5}^{2} k_{2,3} k_{3,2} k_{4,3}^{2} k_{5,1} k_{5,4}
    + 2 k_{1,5}^{2} k_{2,3} k_{3,2} k_{4,3}^{2} k_{5,4}^{2} \\
    & \hspace{-0.5em} + 2 k_{1,5}^{2} k_{2,3} k_{3,2} k_{4,3} k_{4,5} k_{5,1} k_{5,4}
    + k_{1,5}^{2} k_{3,2}^{2} k_{4,3}^{2} k_{5,1} k_{5,4}
    + k_{1,5}^{2} k_{3,2}^{2} k_{4,3}^{2} k_{5,4}^{2} \\
    & \hspace{-0.5em} + k_{1,5}^{2} k_{3,2}^{2} k_{4,3} k_{4,5} k_{5,1} k_{5,4}
    + k_{1,5}^{2} k_{3,2} k_{3,4} k_{4,3} k_{4,5} k_{5,1} k_{5,4}
    + k_{1,5} k_{2,1}^{2} k_{3,2}^{2} k_{4,3}^{2} k_{5,1} \\
    & \hspace{-0.5em} + k_{1,5} k_{2,1}^{2} k_{3,2}^{2} k_{4,3}^{2} k_{5,4}
    + 2 k_{1,5} k_{2,1}^{2} k_{3,2}^{2} k_{4,3} k_{4,5} k_{5,1}
    + k_{1,5} k_{2,1}^{2} k_{3,2}^{2} k_{4,3} k_{4,5} k_{5,4} \\
    & \hspace{-0.5em} + k_{1,5} k_{2,1}^{2} k_{3,2}^{2} k_{4,3} k_{5,1} k_{5,4} 
    + k_{1,5} k_{2,1}^{2} k_{3,2}^{2} k_{4,3} k_{5,4}^{2}
    + k_{1,5} k_{2,1}^{2} k_{3,2}^{2} k_{4,5}^{2} k_{5,1} \\
    & \hspace{-0.5em} + k_{1,5} k_{2,1}^{2} k_{3,2}^{2} k_{4,5} k_{5,1} k_{5,4} 
    + 2 k_{1,5} k_{2,1}^{2} k_{3,2} k_{3,4} k_{4,3} k_{4,5} k_{5,1} 
    + k_{1,5} k_{2,1}^{2} k_{3,2} k_{3,4} k_{4,3} k_{4,5} k_{5,4} \\
    & \hspace{-0.5em} + k_{1,5} k_{2,1}^{2} k_{3,2} k_{3,4} k_{4,3} k_{5,1} k_{5,4}
    + k_{1,5} k_{2,1}^{2} k_{3,2} k_{3,4} k_{4,3} k_{5,4}^{2}
    + 2 k_{1,5} k_{2,1}^{2} k_{3,2} k_{3,4} k_{4,5}^{2} k_{5,1} \\
    & \hspace{-0.5em} + 2 k_{1,5} k_{2,1}^{2} k_{3,2} k_{3,4} k_{4,5} k_{5,1} k_{5,4}
    + k_{1,5} k_{2,1}^{2} k_{3,2} k_{4,3}^{2} k_{5,1} k_{5,4}
    + k_{1,5} k_{2,1}^{2} k_{3,2} k_{4,3}^{2} k_{5,4}^{2} \\
    & \hspace{-0.5em} + k_{1,5} k_{2,1}^{2} k_{3,2} k_{4,3} k_{4,5} k_{5,1} k_{5,4}
    + k_{1,5} k_{2,1}^{2} k_{3,4}^{2} k_{4,5}^{2} k_{5,1}
    + k_{1,5} k_{2,1}^{2} k_{3,4}^{2} k_{4,5} k_{5,1} k_{5,4} \\
    & \hspace{-0.5em} + k_{1,5} k_{2,1}^{2} k_{3,4} k_{4,3} k_{4,5} k_{5,1} k_{5,4}
    + 2 k_{1,5} k_{2,1} k_{2,3} k_{3,2} k_{3,4} k_{4,3} k_{4,5} k_{5,1} \\
    & \hspace{-0.5em} + k_{1,5} k_{2,1} k_{2,3} k_{3,2} k_{3,4} k_{4,3} k_{4,5} k_{5,4}
    + k_{1,5} k_{2,1} k_{2,3} k_{3,2} k_{3,4} k_{4,3} k_{5,1} k_{5,4} \\
    & \hspace{-0.5em} + k_{1,5} k_{2,1} k_{2,3} k_{3,2} k_{3,4} k_{4,3} k_{5,4}^{2}
    + 2 k_{1,5} k_{2,1} k_{2,3} k_{3,2} k_{3,4} k_{4,5}^{2} k_{5,1} \\
    & \hspace{-0.5em} + 2 k_{1,5} k_{2,1} k_{2,3} k_{3,2} k_{3,4} k_{4,5} k_{5,1} k_{5,4}
    + k_{1,5} k_{2,1} k_{2,3} k_{3,2} k_{4,3}^{2} k_{5,1} k_{5,4} \\
    & \hspace{-0.5em} + k_{1,5} k_{2,1} k_{2,3} k_{3,2} k_{4,3}^{2} k_{5,4}^{2}
    + k_{1,5} k_{2,1} k_{2,3} k_{3,2} k_{4,3} k_{4,5} k_{5,1} k_{5,4}
    + 2 k_{1,5} k_{2,1} k_{2,3} k_{3,4}^{2} k_{4,5}^{2} k_{5,1} \\
    & \hspace{-0.5em} + 2 k_{1,5} k_{2,1} k_{2,3} k_{3,4}^{2} k_{4,5} k_{5,1} k_{5,4}
    + 2 k_{1,5} k_{2,1} k_{2,3} k_{3,4} k_{4,3} k_{4,5} k_{5,1} k_{5,4} \\
    & \hspace{-0.5em} + k_{1,5} k_{2,1} k_{3,2}^{2} k_{4,3}^{2} k_{5,1} k_{5,4}
    + k_{1,5} k_{2,1} k_{3,2}^{2} k_{4,3}^{2} k_{5,4}^{2}
    + k_{1,5} k_{2,1} k_{3,2}^{2} k_{4,3} k_{4,5} k_{5,1} k_{5,4} \\
    & \hspace{-0.5em} + k_{1,5} k_{2,1} k_{3,2} k_{3,4} k_{4,3} k_{4,5} k_{5,1} k_{5,4} + k_{1,5} k_{2,3}^{2} k_{3,4}^{2} k_{4,5}^{2} k_{5,1} + k_{1,5} k_{2,3}^{2} k_{3,4}^{2} k_{4,5} k_{5,1} k_{5,4} \\
    & \hspace{-0.5em} + k_{1,5} k_{2,3}^{2} k_{3,4} k_{4,3} k_{4,5} k_{5,1} k_{5,4} + k_{1,5} k_{2,3} k_{3,2} k_{3,4} k_{4,3} k_{4,5} k_{5,1} k_{5,4} + k_{2,1}^{2} k_{3,2}^{2} k_{4,3}^{2} k_{5,1}^{2} \\
    & \hspace{-0.5em} + 2 k_{2,1}^{2} k_{3,2}^{2} k_{4,3}^{2} k_{5,1} k_{5,4} + k_{2,1}^{2} k_{3,2}^{2} k_{4,3}^{2} k_{5,4}^{2} + 2 k_{2,1}^{2} k_{3,2}^{2} k_{4,3} k_{4,5} k_{5,1}^{2} \\
    & \hspace{-0.5em} + 2 k_{2,1}^{2} k_{3,2}^{2} k_{4,3} k_{4,5} k_{5,1} k_{5,4} + k_{2,1}^{2} k_{3,2}^{2} k_{4,5}^{2} k_{5,1}^{2} + 2 k_{2,1}^{2} k_{3,2} k_{3,4} k_{4,3} k_{4,5} k_{5,1}^{2} \\
    & \hspace{-0.5em} + 2 k_{2,1}^{2} k_{3,2} k_{3,4} k_{4,3} k_{4,5} k_{5,1} k_{5,4} + 2 k_{2,1}^{2} k_{3,2} k_{3,4} k_{4,5}^{2} k_{5,1}^{2} + k_{2,1}^{2} k_{3,4}^{2} k_{4,5}^{2} k_{5,1}^{2} \\
    & \hspace{-0.5em} + 2 k_{2,1} k_{2,3} k_{3,2} k_{3,4} k_{4,3} k_{4,5} k_{5,1}^{2}
    + 2 k_{2,1} k_{2,3} k_{3,2} k_{3,4} k_{4,3} k_{4,5} k_{5,1} k_{5,4} \\
    & \hspace{-0.5em} + 2 k_{2,1} k_{2,3} k_{3,2} k_{3,4} k_{4,5}^{2} k_{5,1}^{2}
    + 2 k_{2,1} k_{2,3} k_{3,4}^{2} k_{4,5}^{2} k_{5,1}^{2}
    + k_{2,3}^{2} k_{3,4}^{2} k_{4,5}^{2} k_{5,1}^{2}
\end{align*}
and 
\begin{align*}
    Q & = k_{1,2}^{2} k_{2,1}^{2} k_{3,2}^{2} k_{4,3}^{2} k_{5,1}^{2} + 2 k_{1,2}^{2} k_{2,1}^{2} k_{3,2}^{2} k_{4,3}^{2} k_{5,1} k_{5,4} + k_{1,2}^{2} k_{2,1}^{2} k_{3,2}^{2} k_{4,3}^{2} k_{5,4}^{2} \\
    & \hspace{-0.5em} + 2 k_{1,2}^{2} k_{2,1}^{2} k_{3,2}^{2} k_{4,3} k_{4,5} k_{5,1}^{2} + 2 k_{1,2}^{2} k_{2,1}^{2} k_{3,2}^{2} k_{4,3} k_{4,5} k_{5,1} k_{5,4} + k_{1,2}^{2} k_{2,1}^{2} k_{3,2}^{2} k_{4,5}^{2} k_{5,1}^{2} \\
    & \hspace{-0.5em} + 2 k_{1,2}^{2} k_{2,1}^{2} k_{3,2} k_{3,4} k_{4,3} k_{4,5} k_{5,1}^{2} + 2 k_{1,2}^{2} k_{2,1}^{2} k_{3,2} k_{3,4} k_{4,3} k_{4,5} k_{5,1} k_{5,4} \\
    & \hspace{-0.5em} + 2 k_{1,2}^{2} k_{2,1}^{2} k_{3,2} k_{3,4} k_{4,5}^{2} k_{5,1}^{2} + k_{1,2}^{2} k_{2,1}^{2} k_{3,4}^{2} k_{4,5}^{2} k_{5,1}^{2} + 2 k_{1,2}^{2} k_{2,1} k_{2,3} k_{3,2} k_{3,4} k_{4,3} k_{4,5} k_{5,1}^{2} \\
    & \hspace{-0.5em} + 2 k_{1,2}^{2} k_{2,1} k_{2,3} k_{3,2} k_{3,4} k_{4,3} k_{4,5} k_{5,1} k_{5,4} + 2 k_{1,2}^{2} k_{2,1} k_{2,3} k_{3,2} k_{3,4} k_{4,5}^{2} k_{5,1}^{2} \\
    & \hspace{-0.5em} + 2 k_{1,2}^{2} k_{2,1} k_{2,3} k_{3,4}^{2} k_{4,5}^{2} k_{5,1}^{2} + k_{1,2}^{2} k_{2,3}^{2} k_{3,4}^{2} k_{4,5}^{2} k_{5,1}^{2} + 2 k_{1,2} k_{1,5} k_{2,1}^{2} k_{3,2}^{2} k_{4,3}^{2} k_{5,1}^{2} \\
    & \hspace{-0.5em} + 4 k_{1,2} k_{1,5} k_{2,1}^{2} k_{3,2}^{2} k_{4,3}^{2} k_{5,1} k_{5,4} + 2 k_{1,2} k_{1,5} k_{2,1}^{2} k_{3,2}^{2} k_{4,3}^{2} k_{5,4}^{2} \\
    & \hspace{-0.5em} + 4 k_{1,2} k_{1,5} k_{2,1}^{2} k_{3,2}^{2} k_{4,3} k_{4,5} k_{5,1}^{2} + 4 k_{1,2} k_{1,5} k_{2,1}^{2} k_{3,2}^{2} k_{4,3} k_{4,5} k_{5,1} k_{5,4} \\
    & \hspace{-0.5em} + 2 k_{1,2} k_{1,5} k_{2,1}^{2} k_{3,2}^{2} k_{4,5}^{2} k_{5,1}^{2} + 4 k_{1,2} k_{1,5} k_{2,1}^{2} k_{3,2} k_{3,4} k_{4,3} k_{4,5} k_{5,1}^{2} \\
    & \hspace{-0.5em} + 4 k_{1,2} k_{1,5} k_{2,1}^{2} k_{3,2} k_{3,4} k_{4,3} k_{4,5} k_{5,1} k_{5,4} + 4 k_{1,2} k_{1,5} k_{2,1}^{2} k_{3,2} k_{3,4} k_{4,5}^{2} k_{5,1}^{2} \\
    & \hspace{-0.5em} + 2 k_{1,2} k_{1,5} k_{2,1}^{2} k_{3,4}^{2} k_{4,5}^{2} k_{5,1}^{2} + 4 k_{1,2} k_{1,5} k_{2,1} k_{2,3} k_{3,2} k_{3,4} k_{4,3} k_{4,5} k_{5,1}^{2} \\
    & \hspace{-0.5em} + 4 k_{1,2} k_{1,5} k_{2,1} k_{2,3} k_{3,2} k_{3,4} k_{4,3} k_{4,5} k_{5,1} k_{5,4} + 4 k_{1,2} k_{1,5} k_{2,1} k_{2,3} k_{3,2} k_{3,4} k_{4,5}^{2} k_{5,1}^{2} \\
    & \hspace{-0.5em} + 4 k_{1,2} k_{1,5} k_{2,1} k_{2,3} k_{3,4}^{2} k_{4,5}^{2} k_{5,1}^{2} + 2 k_{1,2} k_{1,5} k_{2,3}^{2} k_{3,4}^{2} k_{4,5}^{2} k_{5,1}^{2} + k_{1,5}^{2} k_{2,1}^{2} k_{3,2}^{2} k_{4,3}^{2} k_{5,1}^{2} \\
    & \hspace{-0.5em} + 2 k_{1,5}^{2} k_{2,1}^{2} k_{3,2}^{2} k_{4,3}^{2} k_{5,1} k_{5,4} + k_{1,5}^{2} k_{2,1}^{2} k_{3,2}^{2} k_{4,3}^{2} k_{5,4}^{2} + 2 k_{1,5}^{2} k_{2,1}^{2} k_{3,2}^{2} k_{4,3} k_{4,5} k_{5,1}^{2} \\
    & \hspace{-0.5em} + 2 k_{1,5}^{2} k_{2,1}^{2} k_{3,2}^{2} k_{4,3} k_{4,5} k_{5,1} k_{5,4} + k_{1,5}^{2} k_{2,1}^{2} k_{3,2}^{2} k_{4,5}^{2} k_{5,1}^{2} + 2 k_{1,5}^{2} k_{2,1}^{2} k_{3,2} k_{3,4} k_{4,3} k_{4,5} k_{5,1}^{2} \\
    & \hspace{-0.5em} + 2 k_{1,5}^{2} k_{2,1}^{2} k_{3,2} k_{3,4} k_{4,3} k_{4,5} k_{5,1} k_{5,4} + 2 k_{1,5}^{2} k_{2,1}^{2} k_{3,2} k_{3,4} k_{4,5}^{2} k_{5,1}^{2} + k_{1,5}^{2} k_{2,1}^{2} k_{3,4}^{2} k_{4,5}^{2} k_{5,1}^{2} \\
    & \hspace{-0.5em} + 2 k_{1,5}^{2} k_{2,1} k_{2,3} k_{3,2} k_{3,4} k_{4,3} k_{4,5} k_{5,1}^{2} + 2 k_{1,5}^{2} k_{2,1} k_{2,3} k_{3,2} k_{3,4} k_{4,3} k_{4,5} k_{5,1} k_{5,4} \\
    & \hspace{-0.5em} + 2 k_{1,5}^{2} k_{2,1} k_{2,3} k_{3,2} k_{3,4} k_{4,5}^{2} k_{5,1}^{2} + 2 k_{1,5}^{2} k_{2,1} k_{2,3} k_{3,4}^{2} k_{4,5}^{2} k_{5,1}^{2} + k_{1,5}^{2} k_{2,3}^{2} k_{3,4}^{2} k_{4,5}^{2} k_{5,1}^{2}.
\end{align*}

As noted in Eqn.~\ref{e-trug}, the rational function for $\tau^{(1)}_1$ has total degree $4$ in the numerator and $5$ in the denominator, while that for $\tau^{(2)}_1$ has total degree $8$ in the numerator and $10$ in the denominator.

\clearpage
\section{Summary of the notation used in the paper}
\label{app:notation}

\begin{table}[!ht]
\centering
\begin{tabularx}{\textwidth}{cX}
    \hline
    \multicolumn{2}{c}{\emph{General mathematics}} \\
    \hline 
    $\0$, $\1$ & Zero or all-ones matrix or vector \\
    $\Ib$ & Identity matrix \\
    $\# A$ & Size of the set $A$ \\
    $\Mb_{[A,B]}$ & Submatrix of $\Mb$ containing the rows in $A$ and columns in $B$ \\
    $\vb_{[A]}$ & Vector of entries in $\vb$ indexed by $A$ \\
    $\theta_S : \left\{ 1, \dotsc, k \right\} \to S$ & Ordering bijection on $S = \left\{ s_1 < \dotsb < s_k \right\}$, $\theta_S(i) = s_i$ for $i = 1, \dotsc, k$ \\
    \hline
    \multicolumn{2}{c}{\emph{Graphs}} \\
    \hline
    $G$ & Graph \\
    $\Vs(G)$ & Vertices of $G$, taken to be $\left\{ 1, \dotsc, N \right\}$ unless otherwise specified \\
    $N$ & Number of vertices in $G$, unless otherwise specified \\
    $\Es(G)$ & Edges of $G$ \\
    $i \to j$, $i \to_G j$ & Edge from vertex $i$ to vertex $j$ in $G$ \\
    $\ell(i \to j)$, $\ell(i \to_G j)$ & Label on edge $i \to j$ in $G$, with units of $(\text{time})^{-1}$ \\
    $\Abar$ & Complement of the vertex subset $A \subseteq \Vs(G)$ \\
    $\Ls(G)$ & Laplacian matrix of $G$ (Eqn.~\ref{eq:lap}) \\
    $\Lb(G)$ & $-\Ls(G)^\T$ (Eqn.~\ref{eq:lap-t}) \\
    $\lambda_i(G)$ & Sum of outgoing edge labels from $i$ \\
    $i \tos j$, $i \tos_G j$ & Existence of a directed path of edges from $i$ to $j$ in $G$ \\
    $\Pp(i)$ & Vertices to which there is a directed path of edges from $i$, $\left\{ j \in \Vs(G) : i \tos j \right\}$ \\
    $w(\cdot)$ & The weight, or product of edge labels, of a subgraph, or the sum of those for a set of subgraphs \\
    $Z$ & If $G$ is a graph with terminal vertices (singleton terminal SCCs), the set of all such terminal vertices \\
    $\lambda_{i,Z}$ & If $G$ is a graph with terminal vertices, $Z \subseteq \Vs(G)$, and $i \in \Zbar$, the sum of the outgoing edge labels $i \to_G z$ for all $z \in Z$ \\
    \hline
    \multicolumn{2}{c}{\emph{Spanning trees and forests}} \\
    \hline
    $\Phi_A(G)$ & Set of spanning forests of $G$ rooted at $A \subseteq \Vs(G)$ \\
    $\Phi_{B \to A}(G)$ & Assuming $\# A = \# B$, the set of spanning forests of $G$ rooted at $A \subseteq \Vs(G)$ in which each $b \in B \subseteq \Vs(G)$ has a path to a distinct root in $A$ \\
    $\Qb^{(k)}(G)$ & The $k$-th Chebotarev--Agaev matrix of $G$ (Eqn.~\ref{eq:qij}) \\
    \hline
    \multicolumn{2}{c}{\emph{Graph operators}} \\
    \hline
    $\Hs_u[\cdot]$ & Hill operator (Def.~\ref{def:hill} and Fig.~\ref{fig:hill}) \\
    $\Us_v[\cdot]$ & Unravelling operator (Def.~\ref{def:unravel} and Fig.~\ref{fig:unravel}) \\
    $\lambdab$ & If $G$ is a graph with $T$ terminal vertices, $Z \subseteq \Vs(G)$, the $(N-T) \times 1$ vector in which the $j$-th entry is $\lambda_{\theta_{\Zbar}(j),Z}$ (Eqn.~\ref{e-lmbt}) \\
    $\deltab_u$ & If $G$ is a graph with $T$ terminal vertices, $Z \subseteq \Vs(G)$, and $u \in \Zbar$, the $(N-T) \times 1$ vector in which the $\theta_{\Zbar}^{-1}(u)$-th entry is $1$ and all other entries are zero (Eqn.~\ref{e-dltb}) \\
    $\psi_{k \mid v,u}$ & Exchange factor, $w(\Phi_{\{k\}}(\Hs_v[G])) / w(\Phi_{\{k\}}(\Hs_u[G]))$ (Def.~\ref{lem:exchange-def} and Lemma \ref{lem:exchange}) \\
    \hline
    \multicolumn{2}{c}{\emph{Markov processes}} \\
    \hline
    $\Xmp$ & Markov process associated with $G$ \\
    $\xb^*(G)$ & Steady-state probability vector of $\Xmp$ \\
    \hline
    \multicolumn{2}{c}{\emph{Splitting probabilities and FPTs (let $G$ be a graph with $T \geq 1$ terminal vertices, $Z \subseteq \Vs(G)$)}} \\
    \hline
    $\pi_{u,z}(G)$ & Splitting probability from $u$ to $z \in Z$ (Eqn.~\ref{eq:split} and Theorem \ref{thm:hill-split}) \\
    $\mu_{c,u,z}^{(r)}(G)$ & $r$-th moment of conditional FPT from $u$ to $z \in Z$ (Eqn.~\ref{eq:moment}, Theorem \ref{thm:hill-moment}) \\
    $\mu_{u,Z}^{(r)}(G)$ & $r$-th moment of unconditional FPT from $u$ to any terminal vertex (Eqn.~\ref{eq:moment-Z}, Theorem \ref{thm:hill-moment-Z}) \\
    \hline
    \multicolumn{2}{c}{\emph{Recurrence times (let $G$ be a strongly connected graph)}} \\
    \hline
    $\tau_u^{(r)}(G)$ & $r$-th moment of recurrence time to $u$ (Eqn.~\ref{eq:recur}) \\
    \hline
\end{tabularx}
\caption{Summary of the notation used in the paper.}
\label{table:notation}
\end{table}

\clearpage

\backmatter

\section*{Acknowledgements}

We are grateful to members of the Gunawardena lab for their comments and feedback.  K-MN and JG were supported by the US National Institutes of Health award R01GM122928.

\section*{Data availability}

This paper has no associated data.



\end{document}